\documentclass[11pt]{amsart}

\usepackage{amsmath,amsfonts,amssymb,amsthm}
\usepackage[english]{babel}

\def\z{\zeta}

\newcommand{\comment}[1]{}
\newcommand{\RR}{{\mathbb R}}%Reals
\newcommand{\CC}{{\mathbb C}}%Complex
\newcommand{\ZZ}{{\mathbb Z}}%Integers
\newcommand{\Id}{\operatorname{Id}}

\newtheorem{thm}{Theorem}[section]
\newtheorem{cor}[thm]{Corollary}
\newtheorem{lem}[thm]{Lemma}
\newtheorem{prop}[thm]{Proposition}

\theoremstyle{definition}
\newtheorem{defn}[thm]{Definition}
\newtheorem{rmk}[thm]{Remark}
\numberwithin{equation}{section}
\newtheorem*{eg}{Example}

\usepackage[utf8]{inputenc}
\usepackage{pinlabel}
\usepackage{graphicx}
\usepackage{color}
\usepackage[all]{xy}
\usepackage{url}
\usepackage{hyperref}
\xyoption{dvips}

\newcommand{\OP}[1]{\mathrm{#1}}

\def\cprime{$'$} \def\cprime{$'$} \def\cprime{$'$} \def\cprime{$'$}

\newcommand{\co}{\mskip0.5mu\colon\thinspace}

\begin{document}

\title[Legendrian ambient surgery]{Legendrian ambient surgery and Legendrian contact homology}

\author[G. Dimitroglou Rizell]{Georgios Dimitroglou Rizell}

\begin{abstract}
Let $L \subset Y$ be a Legendrian submanifold of a contact manifold, $S\subset L$ a framed embedded sphere bounding an isotropic disc $D_S \subset Y \setminus L$, and use $L_S$ to denote the manifold obtained from $L$ by a surgery on $S$. Given some additional conditions on $D_S$ we describe how to obtain a Legendrian embedding of $L_S$ into an arbitrarily small neighbourhood of $L \cup D_S \subset Y$ by a construction that we call Legendrian ambient surgery. In the case when the disc is subcritical, we produce an isomorphism of the Chekanov-Eliashberg algebra of $L_S$ with a version of the Chekanov-Eliashberg algebra of $L$ whose differential is twisted by a count of pseudo-holomorphic discs with boundary-point constraints on $S$. This isomorphism induces a one-to-one correspondence between the augmentations of the Chekanov-Eliashberg algebras of $L$ and $L_S$. 
\end{abstract}

\maketitle

\setcounter{tocdepth}{1}
\tableofcontents

\section{Introduction}
In the following we let $L \subset (Y,\lambda)$ denote a Legendrian submanifold of a $(2n+1)$-dimensional contact manifold with a fixed choice of contact form $\lambda \in \Omega^1(Y)$. In other words, $L$ is $n$-dimensional and satisfies $TL \subset \xi$, where $\xi:=\ker \lambda$ and $\lambda \wedge d\lambda^{\wedge n}$ is a volume-form on $Y$. Associated to the contact form $\lambda$ is the so-called Reeb vector-field $R$ defined uniquely by
\[\iota_R d\lambda=0, \quad \lambda(R)=1.\]
Periodic integral curves of $R$ are usually called \emph{periodic Reeb orbits}, while integral curves of $R$ having positive length and end-points on $L$ are called \emph{Reeb chords on $L$}. The set of Reeb chords on $L$ will be denoted by $\mathcal{Q}(L)$, and under certain genericity assumptions they constitute a discrete set.

The Legendrian submanifolds in a given contact manifold are plentiful due to the fact that the existence problem for Legendrian embeddings satisfy an h-principle; see e.g.~\cite{LegendrianTransverse} for Legendrian knots (i.e.~one-dimensional Legendrian embeddings), while the h-principle for Lagrangian immersions \cite{ClassificationLagrangeImmersions}, \cite{PartialDiffRel} can be used to produce Legendrian embeddings in higher dimensions. However, it is known that there is no \emph{one-parameter} version of the h-principle for Legendrian embeddings. In other words, determining whether two Legendrian embeddings are \emph{Legendrian isotopic}, that is smoothly isotopic through Legendrian embeddings, is in general a very subtle problem.

Legendrian contact homology is a powerful Legendrian isotopy invariant associated to a Legendrian submanifold which is defined by algebraically encoding counts of certain pseudo-holomorphic discs associated to the Legendrian submanifold. The invariant was introduced independently in \cite{DiffAlg} by Chekanov for Legendrian knots, and in \cite{IntroSFT} by Eliashberg, Givental, and Hofer in the more general framework of symplectic field theory (SFT for short).

We proceed to give a very brief outline of the construction of Legendrian contact homology; see Section \ref{sec:background} for more details. The Legendrian contact homology differential graded algebra (DGA for short) of a Legendrian submanifold $L$, also called \emph{Chekanov-Eliashberg algebra}, will be denoted by $(\mathcal{A}(L),\partial)$ and is defined as follows. The underlying algebra $\mathcal{A}(L)=\ZZ_2\langle\mathcal{Q}(L)\rangle$ is the unital graded $\ZZ_2$-algebra freely generated by the set of Reeb chords on $L$, each of which has an induced grading modulo the Maslov number of $L$. The differential $\partial \colon \mathcal{A}(L) \to \mathcal{A}(L)$ is a $\ZZ_2$-linear morphism of degree $-1$ satisfying the Leibniz rule. For each generator in $\mathcal{Q}(L)$ the differential is defined by counting pseudo-holomorphic discs (for a cylindrical almost complex structure) in the so-called \emph{symplectisation of $(Y,\lambda)$}, i.e.~the symplectic manifold $(\RR \times Y,d(e^t\lambda))$ where $t$ denotes the standard coordinate on the $\RR$-factor, having boundary on the Lagrangian cylinder $\RR \times L$. Roughly speaking, the coefficient of the word $b_1  \cdots  b_m$ in the expression $\partial(a)$ is obtained by counting non-trivial (i.e.~not of the form $\RR \times a$) such discs which are rigid up to a translations in the $\RR$-component and which have one boundary puncture asymptotic to $a$ as $t \to +\infty$, and $m$ boundary punctures asymptotic to $b_i$, $i=1,\hdots,m$, as $t \to -\infty$.

In favourable situations it has been shown that
\begin{itemize}
\item $\partial^2=0$, i.e.~the above map defines a differential; and
\item the homotopy type of $(\mathcal{A}(L),\partial)$ is independent of the choice of an almost complex structure and invariant under Legendrian isotopy.
\end{itemize}
It is expected that Legendrian contact homology will be possible to define in an arbitrary contact manifold once that appropriate abstract perturbations techniques have been carried out in the SFT setting (e.g.~using the polyfold machinery due to Hofer, Wysocki and Zehnder \cite{GeneralFredholm}). We will however be limiting ourself to contact manifolds $(Y,\lambda)$ for which the theory already has been established. See Section \ref{sec:assumptions} for a description of some conditions as well as examples.

The class of loose Legendrian submanifolds of dimension at least $n\ge 2$ was introduced in recent work by Murphy \cite{LooseLeg} and they were shown to satisfy the full h-principle in the same article. Loose Legendrian submanifolds can be seen as higher-dimensional analogues of stabilised Legendrian knots and, not surprisingly, their Legendrian contact homologies vanish.

In view of the above, in order to understand the space of Legendrian embeddings it is thus important to understand those embeddings which are \emph{not} loose. Here one immediately faces the following two obstacles: First, there is a lack of explicit constructions of Legendrian submanifolds (and there is no known h-principle which is guaranteed to produce non-loose Legendrian submanifolds). Second, given a Legendrian submanifold of dimension at least two, it is in general a highly non-trivial problem to compute its Legendrian contact homology. Namely, this amounts to counting pseudo-holomorphic discs, which typically means finding the solutions of a non-linear partial differential equation.

\subsection{Results}
Our goal is to provide operations on a given Legendrian submanifold for constructing new Legendrian submanifolds and, more importantly, then to describe how the Legendrian contact homology is affected.

\subsubsection{The Legendrian ambient surgery construction}
There is a well-known construction called cusp-connect sum which, given a Legendrian submanifold having two connected components, produces a Legendrian embedding of their connected sum (see e.g.~\cite[Section 4.2]{NonIsoLeg} and \cite{LegendrianTransverse}). In Section~\ref{sec:defin-an-ambi} we provide a generalisation of cusp-connect sum to handle the case of a general $k$-surgery by a construction that we call \emph{Legendrian ambient surgery}. A related construction has also appeared in \cite{Geography}.

The Legendrian ambient surgery construction works roughly as follows. Assume that we are given a framed embedded $k$-sphere $S \subset L$, $0 \le k \le n-1$, inside a Legendrian submanifold $L \subset (Y,\lambda)$, where $S$ moreover bounds an embedded isotropic $(k+1)$-disc $D_S \subset Y$ having interior disjoint from $L$. The latter disc is called an \emph{isotropic surgery disc}. Use $L_S$ to denote the smooth manifold obtained from $L$ by surgery on $S$ with the above framing. Under certain additional assumptions on the isotropic surgery disc (see Definition \ref{def:surgerydisc}), we can define a Legendrian embedding of $L_S$ contained inside an arbitrarily small neighbourhood of $L \cup D_S$. In the case when additional genericity assumptions are satisfied for both the contact form $\lambda$ and the Legendrian submanifold $L$, the produced Legendrian submanifold $L_S$ moreover satisfies $\mathcal{Q}(L_S)=\mathcal{Q}(L) \cup \{ c_S\}$, where $c_S$ is a new (transverse) Reeb chord appearing on $L_S$ which is of degree $|c_S|=n-k-1$. We say that $L_S \subset Y$ is obtained from $L$ by a \emph{Legendrian ambient surgery on $S$} (see Definition \ref{def:main}).

In \cite{OnConnectedSum} the cusp-connect sum was shown to be a well-defined operation on Legendrian knots. We establish a related result for Legendrian ambient 0-surgeries on a general Legendrian submanifold in Proposition \ref{prop:well-def}, showing that also this is a well-defined operation in a certain sense. This answers a question posed in \cite{NonIsoLeg}.

\subsubsection{The DGA morphism induced by an elementary Lagrangian cobordism}
By an \emph{exact Lagrangian cobordism} from the Legendrian manifold $L_- \subset (Y,\lambda)$ to $L_+ \subset (Y,\lambda)$ we mean a properly embedded submanifold $V \subset \RR \times Y$ of the form
\[V=\left((-\infty,A) \times L_-\right) \cup \overline{V} \cup \left((B,+\infty) \times L_+\right), \]
where $\overline{V} \subset [A,B] \times Y$ is compact, and the pull-back of $e^t \lambda$ to $V$ has a primitive which is globally constant when restricted to either of the subsets $V \cap \{ t \le A\}$ and $V \cap \{ t \ge B \}$.

Recall that, in accordance with the SFT formalism \cite{IntroSFT}, \cite{RationalSFT}, an exact Lagrangian cobordism together with the choice of a generic compatible almost complex structure gives rise to a unital DGA morphism
\[ \Phi_V \colon (\mathcal{A}(L_+),\partial_+) \to (\mathcal{A}(\Lambda_-),\partial_-) \]
of the involved Chekanov-Eliashberg algebras.

By an explicit construction, the Legendrian ambient surgery on $S$ inside $L \subset (Y,\lambda)$ producing $L_S \subset (Y,\lambda)$ also provides an exact Lagrangian cobordism $V_S \subset (\RR \times Y,d(e^t\lambda))$ from $L$ to $L_S$ which is diffeomorphic to the elementary handle-attachment of index $k+1$ corresponding to the surgery. The latter Lagrangian cobordism is called an \emph{elementary Lagrangian cobordism of index $k+1$}. Using
\[\Phi_{V_S} \co (\mathcal{A}(L_S),\partial_{L_S}) \to (\mathcal{A}(L),\partial)\]
to denote the DGA morphism induced by the elementary Lagrangian cobordism $V_S$, we provide the following computation in Section \ref{proof:surjection}.
\begin{thm}
\label{thm:surjection}
For an appropriate regular compatible almost complex structure on $\RR \times Y$, the DGA morphism $\Phi_{V_S}$ is a surjection which, moreover, satisfies
\[ \ker  \Phi_{V_S} = \begin{cases}\langle c_S \rangle, & k<n-1, \\
\langle c_S -1\rangle, & k=n-1.
\end{cases}\]
\end{thm}
In particular, for the above choice of almost complex structure, $(\mathcal{A}(L),\partial)$ is the quotient of $(\mathcal{A}(L_S),\partial_{L_S})$ by a two-sided algebra ideal generated by a single element.

By a \emph{(0-graded) augmentation} of a DGA $(\mathcal{A},\partial)$ we mean a unital DGA morphism
\[ \varepsilon \co (\mathcal{A},\partial) \to (\ZZ_2,0) \]
to the trivial DGA (which in particular thus preserves the grading). Observe that the set of augmentations of the Chekanov-Eliashberg algebra very well can be empty. For instance, this is the case when $(\mathcal{A}(L),\partial)$ is acyclic, i.e.~$1 \in \partial(\mathcal{A}(L))$. However, in the case when there is an augmentation, Chekanov's linearisation construction \cite{DiffAlg} provide Legendrian isotopy invariants given as chain complexes spanned by the set of Reeb chords (which have the advantage of being easier to handle algebraically compared to the full DGA).

Let $\mu \in H^1(V_S,\ZZ)$ denote the Maslov class of $V_S$, and assume that all gradings are taken in $\ZZ/\mu(H_1(L))$. In cases when the grading of the new chord satisfies $0 \neq |c_S| \in \ZZ/\mu(H_1(L))$, it follows by definition that $\varepsilon(c_S)=0$ for all augmentations. Recall that augmentations can be pulled back by the pre-composition with a unital DGA morphism. A purely algebraic consequence of the above theorem is thus the following result. 
\begin{cor}
\label{cor:aug}
If $k<n-1$ and, moreover
\[(n-k-1) \notin \mu(H_1(L)),\]
then the pull-back under $\Phi_{V_S}$ induces a bijection from the set of 0-graded augmentations of $(\mathcal{A}(L),\partial)$ to the set of 0-graded augmentations of $(\mathcal{A}(L_S),\partial_{L_S})$.
\end{cor}

\subsubsection{The Chekanov-Eliashberg algebra twisted by a submanifold}
In order to explain how the Legendrian contact homology of a Legendrian submanifold $L \subset (Y,\lambda)$ is affected by a Legendrian ambient surgery on $S \subset L$, we must first introduce a version of the Chekanov-Eliashberg algebra whose differential takes into account rigid pseudo-holomorphic discs with boundary-point constraints on a submanifold $S \subset L$.

More precisely, assume that we are given the following data: we let $S \subset L$ be a closed $k$-dimensional submanifold, $0 \le k < n-1$, which is disjoint from the endpoints of the Reeb chords on $L$. We also assume that $S$ has a non-vanishing normal vectorfield tangent to $L$, and we fix one such vectorfield $\mathbf{v} \in \Gamma(TL)$ to be part of the data. The \emph{Chekanov-Eliashberg algebra twisted by $(S,\mathbf{v})$} is now defined to be the DGA $(\mathcal{A}(L;S),\partial_{S,\mathbf{v}})$ defined as follows. The underlying algebra $\mathcal{A}(L;S)=\ZZ_2\langle \mathcal{Q}(L) \cup \{s\} \rangle$ is the unital free algebra generated by the Reeb chords (graded as before) together with a formal generator $s$ in degree $|s|=n-k-1$.

To define the twisted differential we will first need to describe the appropriate moduli spaces. We fix a cylindrical almost complex structure $J_{\OP{cyl}}$ on the symplectisation and lift $\mathbf{v}$ to a vectorfield normal to $\RR \times S \subset \RR \times L$ which is invariant under translation of the $t$-coordinate. Furher, we let $\mathbf{w}=(w_1,\hdots,w_{m+1}) \in (\ZZ_{\ge 0})^{m+1}$ be an $(m+1)$-tuple of non-negative integers, $\mathbf{b}=b_1 \cdots b_m$ a word of generators in $\mathcal{Q}(L)$, and $A \in H_1(L)$ a homology class.

Consider the moduli space $\mathcal{M}_{a;\mathbf{b};A}(L;J_{\OP{cyl}})$ defined in Section \ref{sec:bgmoduli} consisting of $J_{\OP{cyl}}$-holomorphic discs
\[u \co (D^2,\partial D^2) \to (\RR \times Y,\RR \times L)\]
having one positive boundary-puncture $p_0$ asymptotic to $a$, and $m$ negative boundary-punctures $p_i$ asymptotic to $b_i$, where $i=1,\hdots,m$. Here we require that the boundary-punctures satisfy $p_1 < \cdots <p_m$ with respect to the total order on $\partial D^2 \setminus \{p_0\}$ induced by the orientation.

Let $\delta>0$ and let $g$ be a Riemannian metric on $L$. Similarly to the constructions in \cite[Section 8.2.D]{EffectLegendrian}, in Section \ref{sec:moduli} we define the moduli spaces
\[\mathcal{M}^{g,\delta}_{a;\mathbf{b},\mathbf{w};A}(L;S,\mathbf{v};J_{\OP{cyl}}) \subset \mathcal{M}_{a;\mathbf{b};A}(L;J_{\OP{cyl}})\]
consisting of the solutions $u$ satisfying the following boundary-point constraints at parallel copies of $\RR \times S \subset \RR \times L$. There are $w:=w_1+\cdots+w_{m+1}$ distinct boundary points $\{ q_i \} \subset \partial D^2$, satisfying $q_1<\cdots<q_w$ with respect to the order on $\partial D^2 \setminus \{p_0\}$, for which
\[u(q_i) \in \RR \times  \exp_S((i-1) \delta \mathbf{v}) \subset \RR \times L, \quad i=1,\hdots,w.\]
Here $\exp_p$ denotes the exponential map at $p \in L$ induced by $g$. Moreover, we require that $w_{i+1}$ of the points $\{ q_i \}$ are situated on the boundary arc between the punctures $p_i$ and $p_{i+1}$ (here we set $p_{m+1}:=p_0$).

Since the data used to define these moduli spaces is invariant under translations of the $t$-coordinate, they carry a natural $\RR$-action induced by such translations.

The differential is now defined as follows. For the formal generator $s$ we define $\partial_{S,\mathbf{v}}(s):=0$, while for a Reeb chord generator $a \in \mathcal{Q}(L)$ we define
\begin{eqnarray*}
\lefteqn{\partial_{S,\mathbf{v}}(a):=}\\
& & \sum_{|a|-|\mathbf{s}|-|\mathbf{b}|+\mu(A)=1} |\mathcal{M}^{g,\delta}_{a;\mathbf{b},\mathbf{w};A}(L;S,\mathbf{v};J_{\OP{cyl}})/\mathbb{R} |s^{w_1}b_1s^{w_2}\cdots s^{w_m}b_ms^{w_{m+1}},
\end{eqnarray*}
where $\mathbf{s}:=s^{w_1 +\cdots+w_{m+1}}$ and $J_{\OP{cyl}}$ is a regular cylindrical almost complex structure. Finally, the differential is extended to the whole algebra using the Leibniz rule.

\enlargethispage{1em}
Observe that the two-sided algebra ideal $\langle s \rangle$ generated by $s$ is preserved by $\partial_{S,\mathbf{v}}$. Furthermore, it follows simply by the definition of $\partial$ that the equality
\[ (\mathcal{A}(L;S),\partial_{S,\mathbf{v}})/\langle s \rangle = (\mathcal{A}(L),\partial)\]
holds under the canonical identification of generators.

The following theorem is a consequence of Lemma \ref{lem:bdy} together with Proposition \ref{prop:twistiso}.
\begin{thm}
\label{thm:dga}
Let $S \subset L$ be a $k$-dimensional submanifold with a non-vanishing normal vectorfield $\mathbf{v}$, where $k<n-1$. For a generic cylindrical almost complex structure on $\RR \times Y$, the DGA $(\mathcal{A}(L;S),\partial_{S,\mathbf{v}})$ is well-defined. Furthermore, its tame-isomorphism class is independent of the choice of a generic pair $(g,\delta)$, and invariant under isotopy of the pair $(S,\mathbf{v})$.
\end{thm}

\begin{rmk}
\label{rem:invariance}
\begin{enumerate}
\item The invariance proof of the Chekanov-Eliashberg algebra should be possible to extend to show that the homotopy type of the twisted Chekanov-Eliashberg algebra is invariant under Legendrian isotopy and independent of the choice of a cylindrical almost complex structure. The proof depends on an abstract perturbation argument, and is omitted.
\item In the case when there is an embedded null-cobordism of $\{1\} \times S$ inside $[0,1] \times L$ along which $\mathbf{v}$ extends as a non-vanishing normal vectorfield, Corollary \ref{cor:nullcob} implies that $(\mathcal{A}(L;S),\partial_{S,\mathbf{v}})$ is tame-isomorphic to the free product of the Chekanov-Eliashberg algebra with the trivial DGA generated by $s$.
\item The homotopy type of the Chekanov-Eliashberg algebra of $L$ twisted by $S$ will in general depend on the homotopy class of the non-vanishing normal vectorfield $\mathbf{v}$, as follows by the computation in Section \ref{sec:whitneyex} (see the example in the case $k=(n-1)/2$).
\end{enumerate}
\end{rmk}

\begin{rmk}
Taking the limit $\delta \to 0$, the solutions in the above moduli spaces converge to $J_{\OP{cyl}}$-holomorphic discs having possibly several boundary points that are mapped to $\RR \times S$. If, for a convergent sequence of such solutions, a cluster of $m$ of boundary points mapping to $\RR \times \exp_S(i \delta \mathbf{v})$ (for different $i \in \ZZ_{\ge 0}$) collide, the solution in the limit can be seen to have a tangency to $\RR \times \exp_S(\RR\mathbf{v})$ of order $m$. We refer to \cite[Part II, Sections 8 and 9]{MyThesis} for more details concerning the moduli spaces of discs with jet-constraints at boundary points. Following the ideas in \cite{ExoticSpheres}, transversality for these moduli spaces is there shown by using an almost complex structure $J_{\OP{cyl}}$ that is integrable in a neighbourhood of $\RR \times S$. It is also shown that there is an isomorphism between the moduli space with appropriate jet-constraints at boundary points and the above moduli space with boundary-point constraints at parallel copies, given that $\delta>0$ is sufficiently small (see \cite[Lemma 8.3]{EffectLegendrian} for a similar statement). Moduli spaces of pseudo-holomorphic curves with jet-constraints at boundary points were also treated in \cite{HolomorphicJets} using a different approach.
\end{rmk}

\subsubsection{The Legendrian contact homology after a subcritical Legendrian ambient surgery ($k<n-1$)}
Assume that $0\le k < n-1$ and that $L_S \subset (Y,\lambda)$ is obtained by a Legendrian ambient surgery on a framed $k$-sphere $S$ inside $L \subset (Y,\lambda)$. Since the isotropic core disc is of dimension $k+1<n$ in this case, and is thus subcritical isotropic, such a Legendrian ambient surgery will be called \emph{subcritical}. For a Legendrian submanifold $L_S$ obtained by a subcritical Legendrian ambient surgery, its Chekanov-Eliashberg algebra can be computed as follows using data on $L$; see Section~\ref{sec:isomorphism} for the proof.
\begin{thm}
\label{thm:isom}
Suppose that $k<n-1$ and that $L_S$ is obtained from $L$ by a Legendrian ambient surgery on a framed embedded sphere $S \subset L$. Let $\mathbf{v}$ be a non-vanishing normal vectorfield to $S$ that is constant with respect to this frame. There is a tame isomorphism of DGAs
\[ \Psi\co (\mathcal{A}(L_S),\partial_{L_S}) \to (\mathcal{A}(L;S),\partial_{S,\mathbf{v}}),\]
where $\Psi(c_S)=s$, and for which the DGA morphism
\[ \Phi_{V_S}\circ\Psi^{-1} \co (\mathcal{A}(L;S),\partial_{S,\mathbf{v}}) \to (\mathcal{A}(L),\partial)=(\mathcal{A}(L;S),\partial_{S,\mathbf{v}})/\langle s \rangle\]
is the natural quotient projection. 
\end{thm}
The map $\Psi$ is defined by counting rigid pseudo-holomorphic discs in $\RR \times Y$ having boundary on $V_S$ and boundary-point constraints at parallel copies of the core disc $C_S \subset V_S \subset \RR \times Y$ of the elementary Lagrangian cobordism $V_S$ (see Section \ref{sec:defcobordism} for its definition). Note that the pair $C_S \subset V_S$ coincides with $(-\infty,-1] \times S \subset (-\infty,-1] \times L$ inside the subset $\{t \le -1\}$.

In the case of a cusp-connected sum on Legendrian surfaces inside a one-jet space $J^1M^2$, an alternative computation of the DGA is also provided by the results in \cite{BorderedLeg}.

\subsubsection{The Legendrian contact homology after a critical Legendrian ambient surgery ($k=n-1$)} In the case when $S \subset L$ is a framed $(n-1)$-sphere inside a Legendrian $n$-dimensional submanifold $L\subset (Y,\lambda)$ an isotropic surgery disc is \emph{Legendrian} and we say that a Legendrian ambient surgery performed on $S$ is \emph{critical}.

Unfortunately we are not able to prove a complete result analogous to Theorem \ref{thm:isom} in the case of a critical Legendrian surgery. However, we refer to Proposition \ref{prop:critical} for a partial result along the same lines, but where a different algebraic set-up has been used.

In this context the main difference between the critical and the subcritical case is that the Chekanov-Eliashberg algebra twisted by a hypersurface $S \subset L$ cannot be defined by the same counts as in the case when the submanifold is of codimension greater than one. The reason is that taking an arbitrary number $m>0$ of parallel copies $S_0=S, S_1, S_2, \hdots, S_{m-1} \subset L$ of $S$, the moduli space with boundary-point constraints at these $m$ hypersurfaces still has the same expected dimension as the moduli space without any boundary-point constraints. 

\begin{rmk}
In the case when $L_S \subset J^1M$ is obtained from an ambient surgery on a Legendrian knot, the result in \cite{LegKnotsLagCob} instead enables us to compute the DGA morphism $\Phi_{V_S}$, and consequently also the Chekanov-Eliashberg algebra of $L$, in terms of data on $L_S$.
\end{rmk}

\begin{rmk}

Given that the Legendrian submanifold $L \subset (J^1M,\lambda_0)$ admits a linear-at-infinity generating family as defined in \cite{QuasiFunctions}, there exists a long exact sequence due to Bourgeois, Sabloff, and Traynor \cite{Geography} relating the generating family homologies of $L_S$ and $L$. Here $L_S$ is obtained by a Legendrian ambient $k$-surgery on $L$ for any $0 \le k \le n-1$, under the additional assumption that the generating family on $L$ can be extended over the Lagrangian elementary cobordism $V_S$ (see Remark \ref{rmk:nogenfam} for an example when this is not possible). Generating family homology is a Legendrian isotopy invariant for Legendrian submanifolds of $J^1M$ admitting a generating family; see \cite{CombFronts}, \cite{GenFunPol}, and \cite{GeneratingFamilies}. For such a Legendrian submanifold it is expected that its Chekanov-Eliashberg algebra has an augmentation whose induced linearisation moreover computes the generating family homology. By the results in \cite{GeneratingFamilies}, this correspondence is known to be true for Legendrian knots inside $J^1\RR$.
\end{rmk}

\section{Examples and computations}
\subsection{Twisting by an even number of points}
Part (2) of Remark \ref{rem:invariance} shows that
\begin{prop}
\label{prop:free}
Let $L$ be connected and of dimension at least two, and let $S \subset L$ be an even number of points. It follows that, for any choice of $\mathbf{v}$, the DGA $(\mathcal{A}(L;S),\partial_{S,\mathbf{v}})$ is isomorphic to the free product of the Chekanov-Eliashberg algebra with the trivial DGA having one generator in degree $n-1$. \end{prop}
In the case when the number of points $S \subset L$ is odd, the situation gets more complicated. See the example below for the case of the Whitney sphere.

\subsection{Computations for the Whitney sphere}
\label{sec:whitneyex}

The \emph{Whitney immersion} is the exact Lagrangian immersion
\begin{gather*}S^n \to \widetilde{L}_{\OP{Wh}} \subset \CC^n,\\
 (\mathbf{x},y) \mapsto (1+iy)\mathbf{x}, \quad  (\mathbf{x},y) \in S^n \subset \RR^n\times \RR.
\end{gather*}
There exists a lift of this immersion to a Legendrian embedding $L_{\OP{Wh}} \subset (\CC^n \times \RR, dz-\sum y_idx_i)$. This Legendrian submanifold will be referred to as the \emph{Whitney $n$-sphere}. We use $c$ to denote its unique Reeb chord, which is of degree $|c|=n$, and corresponds to the unique transverse double-point of $\widetilde{L}_{\OP{Wh}}$ situated at the origin. 

In the following we suppose that $S \subset L_{\OP{Wh}}$ is an embedded $k$-dimensional submanifold, where $k<n-1$.

{\bf Case $ n-k-1 \nmid n-1$: }
By degree reasons it follows that, for any choice of non-vanishing normal vectorfield $\mathbf{v}$ to $S$, the twisted Chekanov-Eliashberg algebra
\[ (\mathcal{A}( L_{\OP{Wh}};S),\partial_{S,\mathbf{v}}) \]
is the free product of the Chekanov-Eliashberg algebra with the trivial DGA having one generator in degree $|s|=n-k-1$.

{\bf Case $k=0$: }
In the case when $|S|$ is even, Proposition \ref{prop:free} implies that
\[ (\mathcal{A}( L_{\OP{Wh}};S),\partial_{S,\mathbf{v}}) \]
is the free product of the Chekanov-Eliashberg algebra with the trivial DGA having one generator in degree $|s|=n-1$. In the case when $|S|$ is odd, we must however rely on an actual count of pseudo-holomorphic discs. Using an almost complex structure as described in Section \ref{sec:whitney} below, Lemma \ref{lem:whitney} shows that there is a unique holomorphic disc passing through a given point for the standard almost complex structure. In particular, the differential is thus given by
\[\partial_{S,\mathbf{v}}(c)=|S|s.\]

In the following we let
\[ S^{n-1}_{\mathfrak{Re}} \subset \mathfrak{Re}(\CC^n)=\RR^n \subset \CC^n\]
denote the unit $(n-1)$-sphere inside the real-part of $\CC^n$, and observe that $S^{n-1}_{\mathfrak{Re}} \subset \widetilde{L}_{\OP{Wh}}$ is an embedded sphere of codimension one. We now consider the special case when $S \subset S^{n-1}_{\mathfrak{Re}}$ is the embedded $k$-dimensional sphere being the boundary of the isotropic disc
\[D_S := \left\{ \begin{array}{l}
x_1^2+\cdots+x_{k+1}^2 \le 1, \\ x_{k+2}=\cdots=x_n=0, \\ \mathbf{y}=0
\end{array}\right\} \subset \mathfrak{Re}(\CC^n).\]
This isotropic disc lifts to an isotropic disc embedded in $(\CC^n \times \RR,\lambda_0)$ having boundary on $L_{\OP{Wh}}$. Furthermore, one can find a Lagrangian frame of the symplectic normal bundle of $D_S$ that makes it into an isotropic surgery disc compatible with some choice of frame of the normal bundle of $S \subset L_{\OP{Wh}}$. In other words, we can perform a Legendrian ambient surgery on $S$, producing the Legendrian submanifold $(L_{\OP{Wh}})_S$. Theorem \ref{thm:isom} implies that
\[ (\mathcal{A}((L_{\OP{Wh}})_S),\partial_{(L_{\OP{Wh}})_S}) \simeq (\mathcal{A}(L_{\OP{Wh}};S),\partial_{S,\mathbf{v}}),\]
where $\mathbf{v}$ may be taken to be (the restriction of) the vectorfield $\sum_{i=1}^n x_i \partial_{y_i}$ in $TL_{\OP{Wh}}$ along $S^{n-1}_{\mathfrak{Re}}$ (see condition (c) in Definition \ref{def:surgerydisc}).

{\bf Case $0\le k<(n-1)/2$: } Since $\pi_k(S^n)=0$, and since $S$ is of codimension at least $k+2$, it follows that $S$ is null-cobordant in $[0,1] \times L_{\OP{Wh}}$ by a null-cobordism along which $\mathbf{v}$ extends as a non-vanishing normal vectorfield. Part (2) of Remark \ref{rem:invariance} now shows that $(\mathcal{A}(L_{\OP{Wh}};S),\partial_{S,\mathbf{v}})$ is the free product of the Chekanov-Eliashberg algebra with the trivial DGA having one generator in degree $|s|=n-k-1$.

{\bf Case $k=n-2$: } It can be explicitly checked that $S$ is null-cobordant in $[0,1] \times S^{n-1}_{\mathfrak{Re}} \subset [0,1] \times L_{\OP{Wh}}$ by a null-cobordism to which $\mathbf{v}$ extends as a non-vanishing normal vectorfield. Part (2) of Remark \ref{rem:invariance} again shows that $(\mathcal{A}(L_{\OP{Wh}};S),\partial_{S,\mathbf{v}})$ is the free product of the Chekanov-Eliashberg algebra with the trivial DGA having one generator in degree $|s|=1$.

{\bf Case $k=1$: }
Since $k<n-1$ by assumption, it follows that $n \ge 3$. This case is thus covered by the case $n-k-1 \nmid n-1$ above. Note that, when $n>3$, the Lagrangian frame of the isotropic surgery disc $D_S$ can be chosen to be compatible with any given frame of the normal bundle of $S$, as follows from part (3) of Remark \ref{rem:1disc}. In particular, a Legendrian ambient surgery can be performed for which the corresponding elementary Lagrangian cobordism $V_S$ is not spin.

\begin{rmk}
\label{rmk:nogenfam}
A Legendrian immersion in $\CC^n \times \RR$ or, equivalently, an exact Lagrangian immersion in $(\CC^n,\omega_0)$ that is defined by a generating family $F \colon \mathfrak{Re}(\CC^n) \times \RR^N \to \RR$ satisfies additional topological properties. For instance, its tangent bundle is stably trivial (and hence it is spin). Recall that the Whitney sphere admits an explicit generating family. In the case when the above elementary Lagrangian cobordism $V_S$ is not spin, this generating family can thus \emph{not} be extended over the cobordism.
\end{rmk}

{\bf Case $k=(n-1)/2$: } In this case $|s|=(n-k-1)=(n-1)/2$. Since the normal bundle of $S$ is trivial, the choice of a non-vanishing normal vectorfield (up to homotopy) lives in
\[\pi_{(n-1)/2}(S^{n-k-1})=\pi_{(n-1)/2}(S^{(n-1)/2}) \simeq \ZZ.\]
One obtains an explicit identification of groups by the algebraic count of zeros of the orthogonal projection of this vectorfield to the normal bundle of $S \subset S^{n-1}_{\mathfrak{Re}}$ (recall that $S^{n-1}_{\mathfrak{Re}} \subset L_{\OP{Wh}}$ is of codimension one). We use $\mathbf{v}_m$ to denote the non-vanishing normal vectorfield whose projection has an algebraic number $m \in \ZZ$ of zeros (for a fixed choice of orientation). Note that the non-vanishing vectorfield $\mathbf{v}$ above is homotopic to $\mathbf{v}_0$.

For the almost complex structure in Section \ref{sec:whitney} below, the descriptions of the holomorphic discs in $\CC^n$ with boundary on $\widetilde{L}_{\OP{Wh}}$ given by Lemma \ref{lem:whitney} can be used for explicitly computing the differential. For a generic perturbation of the vectorfields $\mathbf{v}_m$, one readily checks that
\[ \partial_{S,\mathbf{v}_m}(c)=m s^2.\]
In particular, it follows that the homotopy type of the Chekanov-Eliashberg algebra twisted by $S$ depends on the homotopy class of the non-vanishing normal vectorfield.

\subsection{Holomorphic discs on the Whitney sphere}
\label{sec:whitney}

As in Section \ref{sec:integrable} we will consider the unique cylindrical almost complex structure $J_0$ on the symplectisation $\RR\times (\CC^n \times \RR)$ for which the canonical projection $\RR\times (\CC^n \times \RR) \to \CC^n$ is holomorphic.

By \cite[Theorem 2.1]{Lifting} this projection induces a bijection between the moduli space $\mathcal{M}_{c;\emptyset}(\RR \times L_{\OP{Wh}};J_0)$ and the moduli space of holomorphic polygons in $\CC^n$ having boundary on $\widetilde{L}_{\OP{Wh}}$. The following lemma can thus be used to compute the Chekanov-Eliashberg algebra of the Whitney sphere (possibly twisted by some submanifold).

For any point $\mathbf{x} \in S^{n-1}_{\mathfrak{Re}} \subset \mathfrak{Re}\CC^n$ in the real unit-sphere we consider the complex halfplane
\[H_{\mathbf{x}} := \{ x+iy; \: x\ge 0 \} \mathbf{x}  \subset \CC \mathbf{x} .\]
Observe that $\CC \mathbf{x} $ intersects $\widetilde{L}_{\OP{Wh}}$ in a figure-eight curve, and that each halfplane $H_{\mathbf{x}}$ intersects $\widetilde{L}_{\OP{Wh}}$ in a closed curve.

\begin{figure}[htp]
\centering
\labellist
\pinlabel $c$ at 91 54
\pinlabel $\widetilde{L}_{\OP{Wh}}\cap\CC\mathbf{x}$ at 60 96

\pinlabel $\pi_{\mathbf{x}}(\widetilde{L}^+_{\OP{Wh}})$ at 60 33
\pinlabel $\pi_{\mathbf{x}}(\widetilde{L}^-_{\OP{Wh}})$ at 60 63
\pinlabel $x$ at 217 47
\pinlabel $y$ at 105 109
\pinlabel $1$ at 188 55
\pinlabel $\color{red}\pi_{\mathbf{x}}(S^{n-1}_{\mathfrak{Re}})$ at 154 57
\endlabellist
\includegraphics{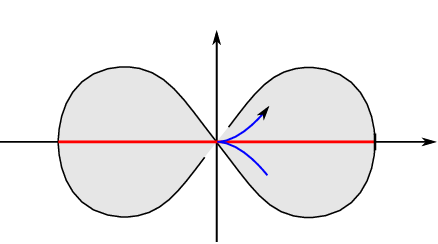}
\caption{The image of $\widetilde{L}_{\OP{Wh}}=\widetilde{L}^+_{\OP{Wh}} \cup \widetilde{L}^-_{\OP{Wh}}$ under the orthogonal projection $\pi_{\mathbf{x}} \co \CC^n \to \CC \mathbf{x} \simeq \CC$, where $\mathbf{x} \in S^{n-1}_{\mathfrak{Re}}$. The arrow denotes the behaviour of the boundary of the hypothetical disc in the proof of Lemma \ref{lem:whitney}.}
\label{fig:whitney}
\end{figure}

\begin{lem}
\label{lem:whitney}
The moduli space of (non-trivial) holomorphic discs in $(\CC^n,i)$ having boundary on $\widetilde{L}_{\OP{Wh}}$ and one puncture mapping to the double-point $c$ is $(n-1)$-dimensional and transversely cut out. Moreover, every solution is contained in a unique one-dimensional complex halfplane $H_{\mathbf{x} } \subset \CC\mathbf{x} $ as above.
\end{lem}
\begin{proof}
It is easily seen that each such one-dimensional complex halfplane contains a unique such disc. Conversely, we now show that every non-trivial holomorphic disc
\[u \co (D^2,\partial D^2) \to (\CC^n,\widetilde{L}_{\OP{Wh}})\]
having one boundary-puncture is of this form.

In the following we let $\pi_{\mathbf{x}} \co \CC^n \to \CC \mathbf{x} \simeq \CC$ denote the orthogonal projection onto the one-dimensional complex vectorspace spanned by $\mathbf{x} \in \mathfrak{Re}\CC^n \subset  \CC^n$.

The image of the Whitney sphere $\pi_{\mathbf{x}}(\widetilde{L}_{\OP{Wh}}) \subset \CC$ under these projections can be seen to be a filled figure-eight curve as shown in Figure \ref{fig:whitney}. The composition $\pi_{\mathbf{x}} \circ u \co D^2 \to \CC$ is holomorphic and maps the boundary into $\pi_{\mathbf{x}}(\widetilde{L}_{\OP{Wh}})$. Furthermore, the boundary-puncture is mapped to the origin. By the open mapping theorem, it thus follows that

\[\pi_{\mathbf{x}} \circ u(D^2 \setminus \partial D^2)\subset\OP{int}\pi_{\mathbf{x}}(\widetilde{L}_{\OP{Wh}}),\]
unless $\pi_{\mathbf{x}}\circ u \equiv 0$ vanishes identically (i.e.~maps constantly to the double-point of the figure-eight curve).

The boundary of the disc $u$ has the following behaviour at the boundary-puncture $p_0 \in \partial D^2$ mapping to the double-point $0 \in \widetilde{L}_{\OP{Wh}}$.  Let $\widetilde{L}_{\OP{Wh}}^\pm \subset \widetilde{L}_{\OP{Wh}}$ denote the image of the hemisphere $\{ \pm y \ge 0\} \cap S^n$ under the Whitney immersion. For $\epsilon>0$ sufficiently small, we have that
\[u(e^{i\theta}p_0) \in \begin{cases} \widetilde{L}_{\OP{Wh}}^+ \subset \CC^n, & 0>\theta>-\epsilon,\\
\{0\} \in \CC^n, & \theta=0,\\
\widetilde{L}_{\OP{Wh}}^- \subset \CC^n, & \epsilon>\theta>0.
\end{cases}\]
From this, it follows that the boundary of $u$ must intersect the sphere
\[S^{n-1}_{\mathfrak{Re}} \subset  \widetilde{L}_{\OP{Wh}}^- \cap \widetilde{L}_{\OP{Wh}}^+ \subset \widetilde{L}_{\OP{Wh}}\]
of codimension one in at least one point, say $\mathbf{x}_0 \in \mathfrak{Re}\CC^n$.

Since $u$ is holomorphic, we may furthermore assume that the following property holds for at least one point $p \in \partial D^2$ that maps to $\mathbf{x}_0$. There exists some $\epsilon>0$ for which
\begin{equation} \label{eq:behaviour} u(e^{i\theta}p) \in \begin{cases} \widetilde{L}_{\OP{Wh}}^- \subset \CC^n, & 0>\theta>-\epsilon,\\
\{\mathbf{x}_0\} \in \CC^n, & \theta=0,\\
\widetilde{L}_{\OP{Wh}}^+ \subset \CC^n, & \epsilon>\theta>0.
\end{cases}\end{equation}

We will now show that the projection $\pi_{\mathbf{x}_1} \circ u$ must vanish whenever $\CC\mathbf{x}_1$ is in the orthogonal complement of $\CC\mathbf{x}_0$. From this it follows that $u$ is a map of the required form.

Too see this claim, observe that the holomorphic map $\pi_{\mathbf{x}_1} \circ u$ vanishes at the above boundary point $p \in \partial D^2$ that maps to $\mathbf{x_0}$. Moreover, the behaviour~(\ref{eq:behaviour}) of $u$ near $p$ implies that, unless $\pi_{\mathbf{x}_1} \circ u$ vanishes constantly, one of $\pm\pi_{\mathbf{x}_1} \circ u$ maps the oriented boundary near $p$ as depicted by the arrow in Figure \ref{fig:whitney}. In the latter case, the open mapping theorem can be used to get a contradiction with the fact that $\pi_{\mathbf{x}_1} \circ u(D^2 \setminus \partial D^2) \subset \OP{int} \pi_{\mathbf{x}_1}(\widetilde{L}_{\OP{Wh}})$. In conclusion, $\pi_{\mathbf{x}_1} \circ u \equiv 0$.

Finally, the transversality of this moduli space follows by an argument similar to the proof of Lemma \ref{lem:transversedisc}.
\end{proof}

\section{Background}
\label{sec:background}
In this section we give an outline of the theory of Legendrian contact homology \cite{IntroSFT,DiffAlg,ContHomP} as well as certain constructions from relative symplectic field theory \cite{RationalSFT}.

Both the construction and the invariance of Legendrian contact homology have been worked out in the case $(J^1\RR,\lambda_0)$ by Chekanov \cite{DiffAlg}, and in more general contactisations $(P \times \RR,dz+\theta)$ by Ekholm, Etnyre and Sullivan \cite{ContHomP}, where $(P,d\theta)$ is an exact symplectic manifold having finite geometry at infinity.

The above versions of Legendrian contact homology are defined by counting pseudo-holomorphic discs in the symplectic manifold $(P,d\theta)$ having boundary on the projection of $\Lambda$ (which is an exact Lagrangian immersion). We will be using the version of Legendrian contact homology obtained as a special case of the (more general) invariant of a Legendrian submanifold called relative symplectic field theory \cite{RationalSFT}, which is defined by counting punctured pseudo-holomorphic disc in the symplectisation $\RR \times Y$ having boundary on $\RR \times \Lambda$. In this setting (using current technology) we are able to consider a larger, but still quite restrictive, class of contact manifolds $(Y,\lambda)$ --- this includes certain closed contact manifolds.

It should also be noted that the two a priori different counts of pseudo-holomorphic discs for a Legendrian submanifold of a contactisation $(P \times \RR,dz+\theta)$ described above give rise to the same invariant, as was shown in \cite{Lifting}.

\subsection{Technical assumptions made on the contact manifold}
\label{sec:assumptions}
First, we will always make the genericity assumption that every periodic Reeb orbit on $(Y,\lambda)$ is transversely cut out in the sense that the linearised Poincar\'{e} return map has no eigenvalue equal to one. Observe that this is the case for a generic choice of contact form $\lambda$.

Second, we will only consider contact manifolds of one of the following two types:
\begin{enumerate}
\item $(Y,\lambda)$ is closed, every periodic Reeb orbit is non-degenerate, and each finite-energy $J_{\OP{cyl}}$-holomorphic plane in $(\RR \times Y,d(e^t\lambda))$ has expected dimension at least two for a cylindrical almost complex structure $J_{\OP{cyl}}$.
\item $(Y,\lambda)=(P \times \RR, dz+\theta)$, where $z$ denotes the coordinate on the $\RR$-factor, and $(P,d\theta)$ is an exact symplectic manifold having finite geometry at infinity for some compatible almost complex structure (see \cite[Definition 2.1]{ContHomP}). This contact manifold is the so-called \emph{contactisation} of the exact non-closed symplectic manifold $(P,d\theta)$.
\end{enumerate}
We refer to Remark \ref{rem:lchsft} below and \cite[Appendix B.2]{RationalSFT} for the argument why these conditions are sufficient in order to define Legendrian contact homology.

\begin{rmk} \begin{enumerate}
\item The first condition is automatically satisfied when there are \emph{no contractible Reeb orbits} --- in which case $(Y,\lambda)$ is called \emph{hypertight}. In general the expected dimension can be expressed in terms of the Conley-Zehnder indices of the contractible Reeb orbits of $(Y,\lambda)$; see e.g.~\cite[Proposition 1.7.1]{IntroSFT}. 
\item The Reeb vectorfield of the contactisation is given by $R=\partial_z$ for the above choice of contact form and, in particular, there are no periodic Reeb orbits.
\end{enumerate}
\end{rmk}

Moreover, in addition to the above conditions, for simplicity we will also make the assumption that the contact manifold $(Y,\xi=\ker\lambda)$ has vanishing first Chern class.
\begin{eg} We now give important examples of contact manifolds satisfying the above conditions.
\begin{enumerate}

\item The standard contact $(2n+1)$-sphere $(S^{2n+1},\lambda)$, where $S^{2n+1} \subset \CC^{2(n+1)}$ is the unit-sphere, is a contact manifold of the first kind given that we choose
\[\lambda:=\frac{1}{2}\sum_{i=1}^{n+1} a_i(x_idy_i-y_idx_i)\]
for rationally independent $a_i > 0$, $i=1,\hdots,n+1$.
\item 
The jet-space $J^1M$ of a smooth (possibly non-closed) manifold $M$ can be endowed with the contact structure
\[ (J^1M \simeq T^*M \times \RR, \lambda_0), \quad \lambda_0 := dz+\theta_M,\]
where $z$ is the coordinate on the $\RR$-factor and $-\theta_M$ is the Liouville form on $T^*M$. This is the contactisation of the exact symplectic manifold $(T^*M,d\theta_M)$, which furthermore can be seen to have finite geometry at infinity.

Specialising to the case $M=\RR^n$ and $\theta_{\RR^n}=-\sum_{i=1}^n y_i dx_i$, where $x_i$ are coordinates on $\RR^n$ and $y_i=df(\partial_{x_i})$ are the induced coordinates on $T^*\RR^n$, we obtain the so-called standard contact $(2n+1)$-space.
\end{enumerate}
\end{eg}

\subsection{Definition of the Chekanov-Eliashberg algebra}

For a closed Legendrian submanifold $L \subset (Y,\lambda)$ we will use $\mathcal{Q}(L)$ to denote the set of Reeb-chords on $L$. A Reeb chord $c$ has an associated \emph{action} given by
\[ \ell(c):=\int_c \lambda>0.\]

We will make the following genericity assumptions for the Reeb chords on $L$. The flow $\phi^t_R \colon (Y,\lambda) \to (Y,\lambda)$ induced by $R$ (which necessarily preserves $\lambda$), is required to satisfy the property that $D\phi^T_R(T_pL)$ and $T_{\phi^T_R(p)}L$ intersect transversely inside $\xi_{\phi^T_R(p)}$ for every Reeb chord $[0,T] \ni t \mapsto \phi^t_R(p)$ on $L$. We will also assume that the starting point of a Reeb chord is different from its endpoint, i.e.~that no Reeb chord is the part of a periodic Reeb orbit.

A Legendrian submanifold $L$ satisfying the above property is called \emph{chord generic}. In particular, the Reeb chords on a closed chord generic Legendrian submanifold comprise a discrete set and consequently $\{ c\in \mathcal{Q}(L); \ \ell(c) \le M\}$ is finite for every $M\ge 0$. Observe that (given that the periodic Reeb orbits are generic) the above transversality can be achieved after an arbitrarily $C^1$-small Legendrian perturbation, i.e.~a small isotopy through Legendrian embeddings.

\subsubsection{The algebra and the grading}
\label{sec:grading}

The underlying unital graded algebra of the DGA is defined as follows. We let 
\[ \mathcal{A}(L) = \ZZ_2 \langle \mathcal{Q}(L)\rangle\]
be the unital non-commutative algebra over $\ZZ_2$ which is freely generated by the Reeb chords $\mathcal{Q}(L)$ on $L$.

We say that a Reeb chord $c$ on $L$ is \emph{pure} if it has both its endpoints on the same component. Otherwise, we say that it is \emph{mixed}.

To grade a generator $c$ corresponding to a pure Reeb chord we fix a \emph{capping path} $\gamma_c$ for $c$ on $L$, by which we mean a path on $L$ which starts at the end-point of $c$, and ends at the starting-point of $c$. We grade each Reeb-chord generator $c$ by the formula
\[ |c|:=\nu(\Gamma_c)-1 \in \ZZ/\mu(H_1(L)),\]
where $\nu(\Gamma_c)$ denotes the \emph{Conley-Zehnder index} of the path of Lagrangian tangent planes $TL \subset \xi$ along $\gamma_c$, and where $\mu \in H^1(L;\ZZ)$ is the \emph{Maslov class}. We refer to \cite[Section 2.2]{ContHomR} for the definitions.

In the case when $c$ is mixed, the grading is only well-defined after making additional choices. We proceed by making the following modification of the notion of a capping path. Fix two components $L_1, L_2 \subset L$ and a path $\gamma \subset Y$ connecting a point $p_1 \in L_1$ with a point $p_2 \in L_2$. We also choose a path $\Gamma$ of Lagrangian tangent planes in $\xi$ along $\gamma$ starting at $T_{p_2}L_2\subset \xi$ and ending at $T_{p_1}L_1 \subset \xi$. We define a capping path of a mixed Reeb chord $c$ starting at $L_1$ and ending at $L_2$ to be a path $\gamma_2$ from the end-point of $c$ to $p_2$, followed by $\gamma$, and ultimately followed by a path $\gamma_2$ from $p_1$ to the starting-point of~$c$.

Finally, we define the grading by applying the above formula to the path $\Gamma_c$ of Lagrangian tangent planes obtained by taking the path of tangent planes $TL_2 \subset \xi$ along $\gamma_2$, followed by the path $\Gamma$ chosen above, and ultimately followed by the path of tangent planes $TL_1 \subset \xi_1$ along $\gamma_1$.

In the case $(J^1M,\lambda_0)$, the following description of the Conley-Zehnder index and Maslov class is given in \cite[Lemma 3.4]{NonIsoLeg}. The canonical projection
\[\Pi_{\OP{F}} \co J^1M \to M \times \RR \]
is called the \emph{front projection}. We may moreover assume that $L$ is generic under the front projection, which implies that the singularities of codimension one of $\Pi_{\OP{F}}(L)$ consist of self-intersections and cusp-edges.

Suppose that $c$ is a Reeb chord on $L$ contained above a point $p \in M$. Let $f_s$ and $f_e$ be the smooth real-valued functions corresponding to the $z$-coordinates of the sheets of $L$ at the starting-point and end-point of the Reeb chord $c$, respectively, which are well-defined in a neighbourhood of $p \in M$. Phrased differently, in a neighbourhood of $p \in M$ the front projection of the two sheets of $L$ containing the endpoints of $c$ are given as the graphs of $f_s$ and $f_e$, respectively.

The Conley-Zehnder index of the path $\Gamma_c$ of Lagrangian tangent planes along $\gamma_c$ can now be expressed as
\begin{equation} \label{eq:cz} \nu(\Gamma_c)=D(\gamma_c)-U(\gamma_c)+\mathrm{index}_p(f_e-f_s)-1,
\end{equation}
where $D(\gamma_c)$ and $U(\gamma_c)$ denote the number of cusp-edges in the front projection traversed by (a generic perturbation of) $\gamma_c$ in downward and upward direction, respectively, and where $\mathrm{index}_p(f_e-f_s)$ denotes the Morse index of the function $f_e-f_s$ at $p$.

Similarly, the Maslov class evaluated on $[\gamma] \in H_1(L)$ can be computed by the formula
\begin{equation} \label{eq:maslov} \mu([\gamma])=D(\gamma)-U(\gamma),\end{equation}
where $\gamma$ is a generic smooth closed curve on $L$, and where $D$ and $U$ are as above.

\subsubsection{The differential}
\label{sec:differential}

Fix a cylindrical almost complex structure $J_{\OP{cyl}}$ on the symplectisation $\RR \times Y$. For a Reeb-chord generator $a \in \mathcal{Q}(L)$, we define
\[ \partial(a):=\sum_{|a|-|\mathbf{b}|+\mu(A)=1} |\mathcal{M}_{a;\mathbf{b};A}(\RR \times L;J_{\OP{cyl}})/\RR|\mathbf{b},\]
where $\mathcal{M}_{a;\mathbf{b};A}(\RR \times L;J_{\OP{cyl}})$ denotes  the moduli space of $J_{\OP{cyl}}$-holomorphic discs defined in Section \ref{sec:bgmoduli} below, and where the sum is taken over all possible words $\mathbf{b}=b_1 \cdots  b_m$ of Reeb chords (including the empty word) and all homology classes $A \in H_1(L)$.

Observe that, since $J_{\OP{cyl}}$ is cylindrical, these spaces have a natural $\RR$-action induced by translation of the $t$-coordinate.

Here we assume that $J_{\OP{cyl}}$ is chosen so that the above moduli spaces all are transversely cut-out, and hence of their expected dimensions (see Section \ref{sec:bgtrans} below). In general, we will call such an almost complex structure \emph{regular}, where it should be understood from the context to which moduli spaces this refers.

That this count is well-defined follows from the Gromov-Hofer compactness theorem for these moduli spaces, which is outlined in Section \ref{sec:bgcomp} below. Note that the compactness theorem applies for the moduli spaces appearing in the definition of $\partial(a)$ since, for a fixed $a \in \mathcal{Q}(L)$, the total energy of a solution in a moduli space $\mathcal{M}_{a;\mathbf{b};A}(\RR \times L;J_{\OP{cyl}})$ is bounded from above by $2\ell(a)$ according to Proposition \ref{prop:energy}.

\enlargethispage{1.5em}
Finally, we extend the differential to all of $\mathcal{A}(L)$ via the Leibniz rule
\[ \partial(ab)= \partial(a)b+a\partial(b).\]
The dimension formula below shows that
\[\dim \mathcal{M}_{a;\mathbf{b};A}(\RR \times L;J_{\OP{cyl}})=|a|-|\mathbf{b}|+\mu(A)\]
and that $\partial$ hence is of degree $-1$.

The following standard argument shows that $\partial^2=0$. The compactness result together with pseudo-holomorphic gluing shows that the coefficient in front of the word $\mathbf{b}$ in the expression $\partial^2(a)$ is given by the number of boundary points of the compact one-dimensional moduli space
\[\bigcup_{\dim \mathcal{M}_{a;\mathbf{b};A}(\RR \times L;J_{\OP{cyl}}) =2} \mathcal{M}_{a;\mathbf{b};A}(\RR \times L;J_{\OP{cyl}})/\RR.\]
In particular this coefficient is even, and hence vanishing.

\begin{rmk}
\label{rem:lchsft} Recall that we assume that our contact manifolds satisfy the property that every $J_{\OP{cyl}}$-holomorphic plane in $\RR \times Y$ is of expected dimension at least two. The reason is to prevent the following potential scenario. A priori, a one-dimensional moduli space as above could have a boundary point corresponding to a broken configuration consisting of a (possibly broken) $J_{\OP{cyl}}$-holomorphic disc together with one or more (possibly broken) $J_{\OP{cyl}}$-holomorphic planes (i.e.~one punctured $J_{\OP{cyl}}$-holomorphic spheres). By the additivity of the formula for expected dimension, together with the assumptions on $(Y,\lambda)$ made in Section \ref{sec:assumptions}, the configuration must contain a $J_{\OP{cyl}}$-holomorphic disc of negative expected dimension. However, the existence of such a disc would contradict the above transversality assumption, and hence this kind of breaking does not occur. This is important, since our boundary operator does not take any $J_{\OP{cyl}}$-holomorphic planes into account.
\end{rmk}

One finally defines the \emph{Legendrian contact homology of $L$} to be the homology
\[HC_\bullet(L):=H_\bullet (\mathcal{A}(L),\partial)\]
of the Chekanov-Eliashberg algebra.

\subsubsection{The DGA morphism induced by an exact Lagrangian cobordism and invariance}
\label{sec:cobchain}
Let $V \subset \RR \times Y$ be an exact Lagrangian cobordism from $L_-$ to $L_+$ which is cylindrical outside of the subset $[A,B] \times Y$. We will describe the associated unital DGA morphism
\[ \Phi_V \co (\mathcal{A}(L_+),\partial_+) \to (\mathcal{A}(L_-),\partial_+).\]

Suppose that $\partial_\pm$ is defined using a cylindrical almost complex structure $J_\pm$. Choose a compatible almost complex structure $J$ on $V$ which coincides with the cylindrical almost complex structure $J_+$ and $J_-$ in the sets $[B,+\infty) \times Y$ and $(-\infty,A] \times Y$, respectively. 

The map $\Phi$ is defined on the Reeb-chord generator $a \in \mathcal{Q}(L_+)$ by the formula
\[\Phi_V(a)=\sum_{|a|-|\mathbf{b}|+\mu(A)=0} |\mathcal{M}_{a;\mathbf{b};A}(V;J)|\mathbf{b},\]
where the sum is taken over all words $\mathbf{b}=b_1 \cdots  b_m$ of Reeb chords in $\mathcal{Q}(L_-)$ (including the empty word) and homology classes $A \in H_1(V)$. Here we assume that $J$ is chosen so that the above moduli spaces are transversely cut out (see Section \ref{sec:bgtrans} below), i.e.~that $J$ is regular.

The Gromov-Hofer compactness implies that the above sum is well-defined. Note that the compactness result applies for the moduli spaces appearing in the definition of $\Phi_V(a)$ since the total energy of a solution in $\mathcal{M}_{a;\mathbf{b};A}(V;J)$ is bounded from above by $2e^{B-A} \ell(a)$ according to Proposition~\ref{prop:energy}.

We extend $\Phi_V$ to a unital algebra map. The dimension formula
\[\dim \mathcal{M}_{a;\mathbf{b};A}(V;J)=|a|-|\mathbf{b}|+\mu(A) \]
shows that $\Phi$ is of degree zero. Observe that we are required to use a grading in the group
\[ \ZZ / \mu(H_1(V))\]
for the Chekanov-Eliashberg algebras of both $L_+$ and $L_-$ in order for $\Phi_V$ to respect the grading. Moreover, the gradings of the mixed chords have to be chosen with some care.

The fact that $\Phi_V$ is a chain-map now follows similarly to the proof that $\partial^2=0$. Namely, the coefficient in front of the word $\mathbf{b} \in \mathcal{A}(L_-)$ in the expression $(\partial_- \circ \Phi_V-\Phi_V \circ \partial_+)(a)$ is given by the number of boundary points of the compact one-dimensional moduli space
\[\bigcup_{\dim \mathcal{M}_{a;\mathbf{b};A}(V;J)=1} \mathcal{M}_{a;\mathbf{b};A}(V;J).\]
In particular these coefficients are even, and thus vanishing.

\begin{eg}
\label{ex:id}
A direct consequence of Example \ref{ex:trivial} below is that $\Phi_{\RR \times L}=\Id_{\mathcal{A}(L)}$ in the case when a cylindrical almost complex structure has been used to define the DGA morphism.
\end{eg}

\begin{thm}[\cite{RationalSFT}]
\label{thm:morphisminvariance}
Let $V \subset Y$ be an exact Lagrangian cobordism from the Legendrian submanifold $L_-$ to $L_+$. For a regular almost complex structure as above, the induced map
\[ \Phi_V\co (\mathcal{A}(L_+),\partial_+) \to (\mathcal{A}(L_-),\partial_-)\]
is a DGA morphism whose homotopy-class is invariant under compactly supported deformations of the almost complex structure and compactly supported Hamiltonian isotopies of $V$.
\end{thm}
This result follows from the more general invariance for relative symplectic field theory proven in \cite[Section 4]{RationalSFT}. Also, see \cite[Lemma 3.13]{LegKnotsLagCob} for a description that carries over to our setting.

\begin{rmk}
The proof of this invariance theorem contains one substantial difficulty. One must take into account pseudo-holomorphic discs of formal dimension $-1$, which arise in generic one-parameter families of moduli spaces. The abstract perturbation scheme outlined in \cite[Section 4]{RationalSFT} is crucial for getting control of the situation. Again, for the same reason as in Remark \ref{rem:lchsft}, our assumptions on $(Y,\lambda)$ make it possible to disregard pseudo-holomorphic planes from the analysis.
\end{rmk}

The above invariance theorem is also the main ingredient in the proof of the following invariance result for Legendrian contact homology.

\begin{thm}[\cite{RationalSFT}]
\label{thm:invariance}
Let $L \subset Y$ be a Legendrian submanifold. The homotopy type of its Chekanov-Eliashberg algebra $(\mathcal{A}(L),\partial)$ is
independent of the choice of a regular cylindrical almost complex structure, and invariant under Legendrian isotopy. In particular, $HC_\bullet(L)$ is a Legendrian isotopy invariant.
\end{thm}
\begin{proof}[Sketch of proof]
Let $V_1$ be an exact Lagrangian cobordism from $L_-$ to $L$, and $V_2$ be an exact Lagrangian cobordism form $L$ to $L_+$. After a suitable translation, one can form their concatenation
\[V_1 \odot V_2 := (V_1 \cap \{t \le 0\}) \cup (V_2 \cap \{t \ge 0\})\]
which is an exact Lagrangian cobordism from $L_-$ to $L_+$. Observe that $V_1 \odot V_2$ is diffeomorphic to the cobordism obtained by gluing $V_1$ to $V_2$ along their common end.

Suppose that $L_1$ is Legendrian isotopic to $L_2$ and fix such an isotopy. Arguing as in \cite[Theorem 1.1]{LagrConc} or \cite[Lemma A.2]{RationalSFT}, the isotopy induces Lagrangian cobordisms $U,V,W \subset \RR \times Y$, each diffeomorphic to cylinders, satisfying
\begin{itemize}
\item $V$ is an exact Lagrangian cobordism from $L_1$ to $L_2$ which, moreover, is smoothly isotopic to the trace of the isotopy from $L_1$ to $L_2$.
\item $U,W$ are exact Lagrangian cobordisms from $L_2$ to $L_1$ which, moreover, are smoothly isotopic to the trace of the above isotopy in inverse time.
\item The concatenations $U \odot V$ and $V \odot W$ are exact Lagrangian cobordisms isotopic to $\RR \times L_2$ and $\RR \times L_1$, respectively, by compactly supported Hamiltonian isotopies.
\end{itemize}

Using the splitting construction \cite[Section 3.4]{CompSFT} together with the compactness result one can show the following, given that we use suitable almost complex structures. First, we have the identities
\begin{gather*}
\Phi_{U \odot V}=\Phi_U \circ \Phi_V, \\
\Phi_{V \odot W}=\Phi_V \circ \Phi_W,
\end{gather*}
while, using Theorem \ref{thm:morphisminvariance} together with Example \ref{ex:id}, we get that
\begin{gather*}
\Phi_{U \odot V}\sim \Phi_{\RR \times L_2} \sim \Id_{\mathcal{A}(L_2)}, \\
\Phi_{V \odot W}\sim \Phi_{\RR \times L_1} \sim \Id_{\mathcal{A}(L_1)}.
\end{gather*}
In other words, $\Phi_U$ and $\Phi_W$ are left and right homotopy inverses, respectively, of
\[ \Phi_V \co (\mathcal{A}(L_1),\partial_1) \to (\mathcal{A}(L_2),\partial_2).\vspace{-2em}\]
\end{proof}

\subsubsection{Tame isomorphisms}
Let $\mathcal{A}$ and $\mathcal{A'}$ be unital algebras over $\ZZ_2$ which are freely generated by the sets of generators $\{ a_i \}_{i \in I}$ and $\{ a_i' \}_{i \in I}$, respectively. A unital isomorphism $\Phi \co \mathcal{A} \to \mathcal{A}'$ of DGAs is called \emph{tame} if, after some identification of the generators of $\mathcal{A}$ and $\mathcal{A}'$, it can 
be written as a composition of \emph{elementary automorphisms}, i.e.~automorphisms which are of the form
\[\Psi(a_i)= \begin{cases} a_i, &  i \neq j,
\\  a_j+B, & i = j, \end{cases} \]
for some fixed $j$, where $B$ is an element of the unital subalgebra generated by $\{a_i ; \  i \neq j \}$.

\subsection{Definitions of the moduli spaces}
\label{sec:bgmoduli}
In this section we give an overview of the moduli spaces appearing in the above constructions. We also recall some important properties.

In the following we let $V \subset \RR \times Y$ be an exact Lagrangian cobordism from the Legendrian submanifold $L_-$ to $L_+$ which is cylindrical outside of the subset $I \times Y$, where $I=[A,B]$. We allow the case when $L_-=\emptyset$, as well as the case when $I= \emptyset$ and $V=\RR \times L$. For each Reeb chord in $\mathcal{Q}(L_\pm)$ we will fix a capping path on $L_\pm$ as in Section \ref{sec:differential} above. In case when $V= \RR \times L$ we will moreover use the same capping paths on both ends of $V$.

Observe that a Reeb chord $c$ has a natural parametrisation
\begin{gather*}
\eta_c(t) \co [0,\ell(c)] \to c \subset Y, \\
\dot{\eta_c}(t)=R_{\eta_c(t)},
\end{gather*}
by the Reeb flow.

Let $\dot{D}^2=D^2 \setminus P$, be the closed unit disc with $m+1$ fixed boundary points $P \subset \partial D^2$ removed and let
\[u=(a,v) \co (\dot{D}^2,\partial\dot{D}^2) \to (\RR \times Y,V)\]
be a continuous map. For a conformal structure on $\dot{D}^2$, there is an induced conformal identification of
\[[0,+\infty) \times [0,1] \subset \CC=\{ s+it \}\]
with a set of the form $\overline{U} \setminus \{p\} \subset \dot{D}^2$, where $\overline{U} \subset D^2$ is a compact neighbourhood of $p \in P$.
\begin{defn}
\begin{enumerate}
\item We say that $p \in P$ is a \emph{positive puncture asymptotic to the Reeb chord $c \in \mathcal{Q}(L_+)$} if there exists a $S_0 \in \RR$ for which
\[\lim_{s \to +\infty}( a(s,t)-s/\ell(c),v(s,t))=(S_0,\eta_c(t/\ell(c)))\]
uniformly in the above coordinates (given some choice of metric on~$Y$).
\item We say that $p \in P$ is a \emph{negative puncture asymptotic to the Reeb chord $c \in \mathcal{Q}(L_-)$} if there exists a $S_0 \in \RR$ for which
\[\lim_{s \to +\infty}( a(s,t)+s/\ell(c),v(s,t))=(S_0,\eta_c((1-t)/\ell(c))) \]
uniformly in the above coordinates (given some choice of metric on~$Y$).
\end{enumerate}
\end{defn}

Consider the compactification $\overline{V}$ of $V$ obtained by first making the topological identification $V \simeq V \cap ((A-1,B+1) \times Y)$ and then taking
\[\overline{V}:=\mathrm{cl}(V \cap ((A-1,B+1) \times Y)) \subset [A-1,B+1] \times Y.\]
Let $u$ be a map with the asymptotic properties as above. By abuse of notation, we say that the boundary of $u$ is in homology class $A \in H_1(V)$ if closing up the boundary in $\overline{V}$ using the capping paths produces a cycle in class $A' \in H_1(\overline{V})$, such that $A$ is identified with $A'$ under the isomorphism $H_1(V) \to H_1(\overline{V})$ induced by the inclusion.

Consider a compatible almost complex structure $J$ on $(\RR \times Y,d(e^t\lambda))$ which is cylindrical outside of $I \times Y$. Let $a$ be a Reeb chord on $L_+$, $\mathbf{b}=b_1 \cdots  b_m$ a word of Reeb chords on $L_-$ (we allow the empty word), $A \in H_1(V)$ a homology class, and $\dot{D}^2=D^2 \setminus \{ p_0,\hdots,p_m\}$ the closed unit disc with $m+1$ boundary points removed. We require the punctures to be ordered by $p_1<\cdots<p_m$ relative the order on $\partial D^2 \setminus \{p_0\}$ induced by the orientation.
\begin{defn}
We use $\mathcal{M}_{a;\mathbf{b};A}(V;J)$ to denote the moduli space of $J$-holomorphic maps
\[u =(a,v) \co (\dot{D}^2,\partial \dot{D}^2) \to (\RR \times Y,V),\]
i.e.~solutions to the non-linear Cauchy-Riemann equation
\[ du+J du\circ i=0\]
for some conformal structure $(\dot{D}^2,i)$, for which
\begin{itemize}
\item $p_0$ is a positive boundary puncture asymptotic to $a \in \mathcal{Q}(L_+)$,
\item $p_i$ is a negative boundary puncture asymptotic to $b_i \in \mathcal{Q}(L_-)$ for $i=1,\hdots,m$, and
\item the boundary of $u$ is in homology class $A \in H_1(V)$.
\end{itemize}
Moreover, we identify two such solutions which are related by a (conformal) reparametrisation of the domain.
\end{defn}

\begin{eg}
For any cylindrical almost complex structure $J_{\OP{cyl}}$ and $V= \RR \times L$, each trivial strip $\RR \times c$ being the cylinder over a Reeb chord $c \in \mathcal{Q}(L)$ is an element of $\mathcal{M}_{c;c;0}(\RR \times L;J_{\OP{cyl}})$.\end{eg}

\begin{rmk}
In the case when $V=\RR \times L$ and when the almost complex structure is cylindrical, translations of the $t$-coordinate induces an action by $\RR$ on the above moduli spaces.
\end{rmk}

\begin{rmk}
An elementary application of Stoke's theorem shows the following.
Since $V \subset (\RR \times Y,d(e^t\lambda))$ is exact there are no non-constant $J$-holomorphic curves in $\RR \times Y$, either closed or with boundary on $V$, having compact image. Moreover, any $J$-holomorphic disc as above must have at least one positive puncture. The latter fact can be seen to follow from either part (2) or (3) of Proposition \ref{prop:energy} below.
\end{rmk}

\subsubsection{Energies of pseudo-holomorphic curves}

Let $\varphi_I(t) \co \RR \to \RR_{\geq 0}$ be the continuous and piecewise smooth function satisfying $\varphi_I(t)=e^t$ for $t \in I$ and $\varphi'(t)=0$ for $t \notin I$. We let $\mathcal{C}_I^-$ consist of the smooth functions $\rho\co \RR \to\RR_{\geq 0}$ supported in $\{ t \le A\}$ and satisfying $\int_\RR \rho(t)=1$. Similarly, we let $\mathcal{C}_I^+$ consist of the smooth functions $\rho\co \RR \to\RR_{\geq 0}$ supported in $\{ t \ge B\}$ and satisfying $\int_\RR \rho(t)=e^{B-A} \ge 1$. Finally, we define $\mathcal{C}_I:=\mathcal{C}_I^+ \cup \mathcal{C}_I^-$.
\begin{defn}
For a smooth map
\[u \co (\dot{D}^2,\partial \dot{D}^2) \to (\RR \times Y,V),\]
satisfying the above asymptotic properties near its punctures, we define the quantities
\begin{align*}
&E_{d\lambda}(u):= \int_u d\lambda, \\
&E_{d(\varphi_I\lambda)}(u):= e^{-A}\int_u d(\varphi(t)\lambda),\\
&E_{\lambda,I}(u):=\sup_{\rho \in \mathcal{C}_I} \int_u \rho(t) dt \wedge \lambda,\\
&E_I(u):=E_{d(\varphi_I\lambda)}(u)+E_{\lambda,I}(u),
\end{align*}
which will be called the \emph{$d\lambda$-energy, $d(\varphi_I\lambda)$-energy, $\lambda$-energy, and total energy}, respectively. 
\end{defn}

Observe that in the case when $V=\RR \times L$ is cylindrical and $I=\emptyset$, we have $\varphi_I(t) \equiv 1$ and the $d\lambda$-energy thus coincides with the $d(\varphi_I\lambda)$-energy. The following properties of the above energies are standard, see e.g.~\cite[Lemma B.3]{RationalSFT}.
\begin{prop}
\label{prop:energy}
Let $u \co (\dot{D}^2,\partial \dot{D}^2) \to (\RR \times Y,V)$ be a smooth map as above with positive punctures asymptotic to
\[a_1, \hdots, a_{m_+} \in \mathcal{Q}(L_+)\]
and negative punctures asymptotic to
\[b_1, \hdots, b_{m_-} \in \mathcal{Q}(L_-).\]
Here we allow either of $m_+$ and $m_-$ to vanish. We assume that $V$ is an exact Lagrangian cobordism which is cylindrical outside of the subset $I \times Y$, where $I=[A,B]$.
\begin{enumerate} 
\item In the case when $V=\RR \times L$ we have
\[E_{d\lambda}(u)=\sum_{i=1}^{m_+} \ell(a_i)-\sum_{i=1}^{m_-}\ell(b_i).\]
If $u$ is $J$-holomorphic for a cylindrical $J$, it moreover follows that $E_{d\lambda}(u) \ge 0$ where equality holds if and only if the image of $u$ is contained in $\RR \times c$ for a Reeb chord $c$.
\item We have
\[E_{d(\varphi_I\lambda)}(u)=e^{B-A}\sum_{i=1}^{m_+} \ell(a_i)-\sum_{i=1}^{m_-} \ell(b_i).\]
Given that $u$ is $J$-holomorphic, where $J$ is cylindrical outside of $I \times Y$, it moreover follows that
\[E_{d(\varphi_I\lambda)}(u) \ge 0.\]
Again, equality holds if and only if the image of $u$ is contained in $\RR \times c$ for a Reeb chord $c$.
\item If $u$ is $J$-holomorphic for a compatible almost complex structure $J$ which is cylindrical outside of $I \times Y$, it follows that
\[0< E_{\lambda,I}(u) = e^{B-A}\sum_{i=1}^{m_+} \ell(a_i).\]
\end{enumerate}
\end{prop}
\begin{proof}
(1): This follows by an elementary application of Stoke's theorem. Observe that $\lambda$ vanishes on $V=\RR \times L$.

(2): By the assumptions on $V$ it follows that that $\varphi_I(t)\lambda$ pulls-back to an exact one-form on $V$ which moreover has a primitive that is constant when restricted to either of the subsets $V \cap \{t \le A\}$ and $V \cap \{t \ge B \}$. The equality then follows by Stoke's theorem.

(3): Since $V$ is exact $u$ must have at least one boundary puncture. Using its asymptotic behaviour, one easily computes $E_{\lambda,I}(u)>0$. We proceed to calculate the expression of $E_{\lambda,I}(u)$ in terms of the action of the asymptotic Reeb chords.

Consider a compactly supported bump-function $\rho(t) \in \mathcal{C}_I^+$. Using the fact that $\int_\RR \rho(t) dt=e^{B-A}$, together with the asymptotic properties of $u$, one computes
\[ \lim_{N \to +\infty} \int_u \rho(t-N) dt \wedge \lambda=e^{B-A}\sum_{i=1}^{m_+} \ell(a_i).\]
In other words, we have the inequality
\[ E_{\lambda,I}(u) \ge e^{B-A}\sum_{i=1}^{m_+} \ell(a_i).\] 

It remains to show the opposite inequality. We start by taking $\rho(t) \in \mathcal{C}_I^-$, which we use to define the continuous function
\begin{align*}
&P(t) :=\int_{-\infty}^t (\rho(s)+\chi_I(s)e^{-A}e^s)ds,\\
&\chi_I(s):=\begin{cases} 1, & s \in I=[A,B],\\
0, & s \notin I=[A,B].
\end{cases}
\end{align*}
Observe that $P(t)=e^{-A}e^t$ for $t \in [A,B]$, while $P(t)=e^{B-A}$ for $t \ge B$.

It follows that
\[\int_u \rho(t) dt \wedge \lambda = \int_u d(P(t)\lambda)-\int_u (\chi_I(t) e^tdt \wedge \lambda +  P(t)d\lambda) \le \int_u d(P(t)\lambda),\]
where the latter inequality holds since the two-form
\[ \chi_I(t) e^{-A}e^tdt \wedge \lambda +  P(t)d\lambda \]
is positive on $J$-complex lines (here we must use the assumption that $J$ is cylindrical outside of $I \times Y$). We thus get
\[ \int_u \rho(t) dt \wedge \lambda \le \int_u d(P(t)\lambda)=\int_{\partial u} P(t)\lambda+e^{B-A}\sum_{i=1}^{m_+}\ell(a_i),\]
where the last equality follows from the fact that $\lim_{t \to -\infty}P(t)\lambda=0$, while $\lim_{t \to +\infty}P(t)\lambda=e^{B-A}\lambda$.

By the assumptions on $V$ it follows that $P(t)\lambda$ pulls back to an exact one-form on $V$ which has a primitive that is constant when restricted to either of the subsets $V \cap \{t \le A\}$ and $V \cap \{t \ge B\}$. From this it follows that
\[ \int_u \rho(t) dt \wedge \lambda \le  \int_{\partial u} P(t)\lambda+e^{B-A}\sum_{i=1}^{m_+}\ell(a_i)= e^{B-A}\sum_{i=1}^{m_+} \ell(a_i). \]

For $\rho(t) \in \mathcal{C}_I^+$, if we instead take $P(t):=\int_{-\infty}^t \rho(s)ds$, the above calculation gives the same upper bound. In conclusion, we have the inequality
\[E_{\lambda,I}(u) \le e^{B-A}\sum_{i=1}^{m_+} \ell(a_i) \]
which, together with the opposite inequality proven above, finally gives
\[E_{\lambda,I}(u) = e^{B-A}\sum_{i=1}^{m_+} \ell(a_i). \vspace{-2em}\]
\end{proof}

\subsubsection{Gromov-Hofer compactness}
\label{sec:bgcomp}
Gromov's famous compactness\linebreak theorem was the starting-point for pseudo-holomorphic curve techniques in symplectic geometry. In its original form \cite[Section 1.5.B]{Gromov} it establishes compactness for the moduli spaces of closed pseudo-holomorphic curves inside a compact symplectic manifold. This was later generalised to the setting in which symplectic field theory is formulated \cite{CompSFT}, where the symplectic manifold is non-compact and the closed pseudo-holomorphic curves are allowed to have interior punctures asymptotic to periodic Reeb orbits. This version of the compactness theorem is sometimes referred to as Gromov-Hofer compactness.

We will need a version of the latter compactness theorem for pseudo-holomorphic curves  with boundary and boundary-punctures asymptotic to Reeb chords. An outline of this result is given in \cite[Section 11.3]{CompSFT}, and it is treated thoroughly in \cite{CompAbbas}.

In order to formulate the compactness theorem we need the concept of a \emph{pseudo-holomorphic building}, which was introduced in \cite[Section 7]{CompSFT}. In our situation, the following less general definition will suit our needs. Again, assume that $J$ is a compatible almost complex structure on $\RR \times Y$ which coincides with the cylindrical almost complex structures $J_-$ and $J_+$ in the sets $\{t \le A\}$ and $\{t \ge B \}$, respectively.

\begin{defn} A \emph{pseudo-holomorphic building} consists of a finite collection of pseudo-holomorphic discs in $\RR \times Y$, where each disc has an associated level $i \in \{i_0,i_0+1,\hdots,i_0+j\}$, $i_0 \le 0$, and a disc in level $i$ is contained in the moduli space
\begin{itemize}
\item $\mathcal{M}_{a;\mathbf{b};A}(V;J)$ for $i=0$,
\item $\mathcal{M}_{a;\mathbf{b};A}(\RR \times L_+;J_+)$ for $i>0$; and
\item $\mathcal{M}_{a;\mathbf{b};A}(\RR \times L_-;J_-)$ for $i<0$.
\end{itemize}
A choice of a bijection between the positive punctures of the discs in the $i$:th level and the negative punctures of the discs in the $(i+1)$:th level for each $i<i_0+j$ is also part of the data. Moreover, if two punctures correspond under this bijection, we require them to be asymptotic to the same Reeb chord. We call a level \emph{trivial} if it consists only of trivial strips, and we assume that there is at most one trivial level.
\end{defn}

By \cite{CompSFT}, \cite{CompAbbas}, any sequence $u_n$ of $J$-holomorphic curves in moduli spaces $\mathcal{M}_{a;\mathbf{b};A}(V;J)$ having a uniform bound on the total energy $E(u_n)\le E$, has a subsequence converging (in the appropriate sense) to a pseudo-holomorphic building as above.

\subsubsection{Transversality}
\label{sec:bgtrans}
One says that a compatible almost complex structure $J$ on $\RR \times Y$ is \emph{regular} if the moduli spaces $\mathcal{M}_{a;\mathbf{b};A}(V;J)$ all are transversely cut out smooth manifolds. Their dimension is then given by
\[ \dim \mathcal{M}_{a;\mathbf{b};A}(V;J) = |a|-|b_1| - \cdots - |b_m|+\mu(A),\]
as follows from the computation of the Fredholm index in \cite[Section 6]{ContHomR} (also see \cite[Proposition 2.3]{ContHomP}). We say that a solution is \emph{rigid} if it lives inside a zero-dimensional transversely cut-out moduli space.

\begin{eg}
\label{ex:trivial}
For a regular cylindrical almost complex structure $J_{\OP{cyl}}$, it follows that a rigid solution in $\mathcal{M}_{a;\mathbf{b};A}(\RR \times L;J_{\OP{cyl}})$ must be translation-invariant, and hence equal to a trivial strip $\RR \times c$.
\end{eg}

Since the above moduli spaces consist of discs with exactly one positive puncture, each solution can be seen to be embedded inside some subset of the form $[N,+\infty) \times Y$ (here we have used the fact that no Reeb chord is contained inside a periodic Reeb orbit). Standard techniques can now be applied to show that these pseudo-holomorphic discs are simple and, hence, that the set of regular almost complex structures $J$ form a Baire subset inside the set of all compatible almost complex structures coinciding with a given cylindrical almost complex structure outside of a compact set; see for example \cite[Theorem 2.14]{FloerConc} together with \cite[Theorem 3.1.6]{JholCurves}.

However, if one wishes to find a regular \emph{cylindrical} almost complex structure, the argument needs to be refined. This was done in \cite{FredholmTheory} for curves without boundary. This result can be generalised to the moduli spaces under consideration here.
\begin{prop}
\label{prop:trans}
\begin{enumerate}
\item There is a Baire set of cylindrical almost complex structures which are regular for the moduli spaces $\mathcal{M}_{a;\mathbf{b};A}(\RR \times L;J_{\OP{cyl}})$ as above.
\item The moduli spaces of the form $\mathcal{M}_{a;\mathbf{b};A}(\RR \times L;J_{\OP{cyl}})$ may be supposed to be transversely cut out after a cylindrical perturbation of $J_{\OP{cyl}}$ supported in an arbitrarily small neighbourhood of $\RR \times a$.
\end{enumerate}
\end{prop}
\begin{proof}
First, observe that the moduli spaces $\mathcal{M}_{a;a;A}(\RR \times L;J_{\OP{cyl}})$ always are transversely out, as follows by an explicit calculation. We thus need to apply a transversality argument for the moduli spaces consisting of non-trivial discs, i.e.~discs that are not contained in a trivial strip.

The proof of \cite[Theorem 4.1]{FredholmTheory}, which shows the first statement in the case when the domain is a Riemann surface without boundary, can be generalised to the current setting. To that end, it is crucial that a solution $u=(\alpha,v)$ satisfies the property that $v$ is somewhere injective in the following sense; there should exist a point $z_0 \in \dot{D}^2 \setminus \partial \dot{D}^2$ for which
\[\begin{cases}\pi_\xi \circ D_{z_0}v \neq 0,\\
v^{-1}(v(z_0))=\{z_0\},
\end{cases}
\]
is satisfied, where $\pi_\xi \co TY \to \xi \subset TY$ denotes the linear projection to the contact-planes along the Reeb vectorfield. To show the second statement of the proposition, it will be enough to infer that $v(z_0)$ may be taken to be arbitrarily close to the Reeb chord $a$.

To find a point $z_0$ satisfying the above, we argue as follows. Let $a \subset Y$ denote the asymptotic Reeb chord of the positive puncture of $u$. Observe that this is the only puncture asymptotic to $a$, as follows by the formula for the $d\lambda$-area given in Proposition \ref{prop:energy} together with the assumption that $u$ is not a trivial strip.

Let $Q \subset \dot{D}^2$ denote the set of points for which $\pi_\xi \circ Dv=0$. Observe that $u$ has positive $d\lambda$-energy by assumption, from which it follows that $Q \subsetneq \dot{D}$. The generalised similarity principle \cite[Theorem 12 in A.6]{SympInvHam} can be applied to any locally defined section $\pi_\xi \circ Dv(X)$, where $X \in TD$ is a locally defined holomorphic vectorfield (see the proof of the second statement of \cite[Theorem 5.2]{PropPseudI}). It follows that the set of limit-points $Q'$ of $Q$ satisfies the property that
\[Q' \cap (\dot{D}^2 \setminus \partial \dot{D}^2) \subset \dot{D}^2 \setminus \partial \dot{D}^2\]
is open. Since this set is closed as well, the non-triviality of $u$ implies that there are no limit-points of $Q$ inside $\dot{D}^2\setminus \partial \dot{D}^2$.

Take a neighbourhood $U \subset \dot{D}^2$ of the positive puncture, where $U$ can be conformally identified with 
\[\{ s +it; \: s \ge 0, \:0 \le t \le 1\} \subset \CC,\]
and where $s+it$ denotes the standard holomorphic coordinate.

By the asymptotic properties of $u$, after possibly shrinking the neighbourhood $U$, we may moreover assume that $u|_U$ is an embedding. In particular, each pseudo-holomorphic disc $u|_{\{ 0 \le s \le A \}}$ satisfies the assumptions of \cite[Theorem 1.14]{PropPseudoIII}. Applying this result we conclude that, for any $A>0$, the subset
\[ \left\{\begin{array}{l} 0 \le s \le A, \\ \pi_\xi \circ D_{s+it} v \neq 0, \\ (v|_{\{0 \le s \le A\}})^{-1}(v(s+it))=\{s+it\}
\end{array}\right\} \subset U \]
is open and dense. From this it now follows that
\[ \left\{\begin{array}{l} \pi_\xi \circ D_{s+it} v \neq 0, \\ (v|_U)^{-1}(v(s+it))=\{s+it\}
\end{array}\right\} \subset U \]
is dense as well.

We will now show the existence of the required point $z_0 \in \dot{D}^2\setminus\partial \dot{D}^2$ assuming the existence of a point $y_0 \in a \setminus v(\dot{D}^2 \setminus U)$, which is a fact that we will establish below. The asymptotic properties of $v$ together with the above property of $v|_U$ shows that there exists some $z_0 \in U \setminus \partial\dot{D}^2$ satisfying $\pi_\xi \circ D_{z_0}v \neq 0$ and $(v|_U)^{-1}(v(z_0))=\{z_0\}$, and for which $v(z_0)$ is arbitrarily close to $y_0 \in a \setminus v(\dot{D}^2 \setminus U)$. It now follows that this choice of $z_0$ may be supposed to satisfy $v^{-1}(v(z_0))=\{z_0\}$ as well, and hence $z_0 \in \dot{D}^2 \setminus \partial\dot{D}^2$ is the sought point.

It remains show to the existence of a point $y_0 \in a \setminus v(\dot{D}^2 \setminus U)$. To that end, let $z' \in \dot{D}^2$ be a limit point of $v^{-1}(a) \subset \dot{D}^2$. Since $\RR \times a$ is pseudo-holomorphic for any cylindrical almost complex structure, an application of the similarity principle in \cite[Lemma 4.2]{ExistenceInjective} shows the following. The limit point must satisfy $z' \in Q \cup \partial\dot{D}^2$. If not, $u$ would map a non-empty open set into $\RR \times a$, thus contradicting the fact that
\[Q \setminus \partial\dot{D}^2 \subset \dot{D}^2\setminus \partial\dot{D}^2\]
is a discrete subset.

Consequently, the points in $\dot{D}^2\setminus \partial \dot{D}^2$ that are mapped to $a$ by $v$ form a discrete subset of $\dot{D}^2\setminus \partial \dot{D}^2$. Use $\mathring{a}$ to denote the Reeb chord $a$ with both end-points removed. Since the boundary condition of $u$ implies that
\[v(\partial\dot{D}^2) \cap a \subset a \setminus \mathring{a},\]
the existence of a point $y_0 \in \mathring{a} \setminus v(\dot{D}^2 \setminus U)$ now follows.
\end{proof}

Another approach to attain a regular cylindrical almost complex structure is to generalise the method used in \cite[Lemma 4.5]{ContHomP} to the symplectisation, using the techniques in \cite{Lifting}. The idea is to consider almost complex structures that are integrable in some neighbourhood of the Reeb chords, for which $L$ moreover satisfies a real-analyticity condition.

Pseudo-holomorphic buildings consisting of transversely cut-out levels can be glued to form transversely cut-out solution \cite[Proposition 4.6]{ContHomP}. To that end, the following properties of a glued pseudo-holomorphic buildings is elementary, but important.
\begin{lem}
\label{lem:energybuilding}
Let $u$ be a pseudo-holomorphic map obtained by gluing the levels in a pseudo-holomorphic building. It follows that
\begin{enumerate}
\item The expected dimension of $u$ is equal to the sum of the expected dimensions of the solutions in each level.
\item $E_{d(\varphi_I\lambda)}(u)$ is the sum of the $E_{d(\varphi_I\lambda)}$-energies of the solutions in level 0 and the $E_{d\lambda}$-energies of the solutions in the other levels.
\end{enumerate}
\end{lem}

Together with Gromov-Hofer compactness it thus follows that, for a regular almost complex structure, the moduli spaces in Section \ref{sec:bgmoduli} can be compactified to smooth manifolds with boundary (with corners), such that the boundary points are in bijection with pseudo-holomorphic buildings of the appropriate type.

\section{Legendrian ambient surgery}
\label{sec:defin-an-ambi}
Let $(Y,\lambda)$ be a contact manifold of dimension $2n+1$ with contact distribution $\xi:=\ker \lambda$, and let $L \subset Y$ be a Legendrian submanifold, which thus satisfies $\dim L=n$.

Here we describe the following construction. Suppose that we are given a Legendrian embedding $L \subset Y$ containing a framed embedded sphere $S \subset L$, together with a so-called isotropic surgery disc $D_S \subset Y$ compatible with the frame of the normal bundle of $S$ (see Definition \ref{def:surgerydisc}). Using $L_S$ to denote the manifold obtained from $L$ by surgery on the framed sphere $S$, the above data will be used to construct a Legendrian embedding $L_S \subset Y$ contained in an arbitrarily small neighbourhood of $L \cup D_S$. We say that $L_S$ is obtained from $L$ by a Legendrian ambient surgery on $S$ (see Definition \ref{def:main}).

Finally, the construction also produces an exact Lagrangian cobordism $V_S \subset \RR \times Y$ from $L$ to $L_S$ which is diffeomorphic to the elementary cobordism induced by the surgery.

\subsection{The isotropic surgery disc}
\label{sec:data}

For a linear subspace $A \subset (W,\omega)$ of a symplectic vectorspace, recall the definition of the \emph{symplectic complement of $A$}, which is the subspace
\[A^\omega := \{ w \in W; \ \omega(A,w)=\{0\} \}.\]

An isotropic submanifold $D \subset (Y,d\lambda)$ has an associated \emph{symplectic normal bundle} identified with $(TD)^{d\lambda}/TD \subset \xi$. In case when $\dim{D}=k+1$, this is naturally a symplectic $2(n-k-1)$-dimensional bundle over $D$ whose symplectic form is given by the restriction of $d\lambda$. We will say that a $(n-k-1)$-dimensional subframe of this vector-bundle is \emph{Lagrangian} if it spans Lagrangian subbundle. The following result is standard.
\begin{lem}
\label{lem:lagrangianframe} A symplectic trivialisation of an $2l$-dimensional symplectic bundle, i.e.~a trivialisation which identifies each fibre with the standard symplectic vectorspace 
\[\left(\RR^l \oplus i\RR^l = \CC^l,\omega_0=\sum_{i=1}^l dx_i \wedge dy_i \right),\]
 induces a Lagrangian frame by taking the first $l$-vectors of the trivialising frame (i.e.~the subframe that spans the real-part of $\CC^l$). Conversely, any Lagrangian frame can be extended to such a symplectic trivialisation.
\end{lem}

The Legendrian ambient surgery will depend on the following data. 
\begin{defn}
\label{def:surgerydisc}
Let $S \subset L$ be an embedded $k$-sphere with a choice of frame $F$ of its normal-bundle $NS \subset TL|_S$. An \emph{isotropic surgery disc compatible with the framed sphere $S \subset L$} is an embedded isotropic closed $(k+1)$-disc $D_S \subset Y$, together with the choice of a Lagrangian frame of its symplectic normal bundle, satisfying
\begin{enumerate}
\renewcommand{\labelenumi}{(\alph{enumi})}

\item $\partial D_S = S$ and $\OP{int}D_S \subset Y \setminus L$.
\item Any outward normal vectorfield to $D_S$ is nowhere tangent to $L$, or equivalently, $(TD_S)^{d\lambda}|_S \cap NS$ is $(n-k-1)$-dimensional.

\item Let $H$ denote a vectorfield in $NS$ satisfying $d\lambda(G,H)>0$ for any outward normal $G$ to $D_S$ (these vector fields form a convex and non-empty set). We require that the frame obtained by adjoining $H$ to the Lagrangian frame of $((TD_S)^{d\lambda}/TD_S)|_S$ is a frame of $NS$ which is homotopic to $F$.
\end{enumerate}
\end{defn}

In particular, the last condition implies that the Lagrangian frame of the symplectic normal bundle $(TD_S)^{d\lambda}/TD_S$ of $D_S$ is required to be tangent to $L$ along $\partial D_S=S$.  Moreover, assuming that $D_S$ satisfies (a) and (b), the following can be said about condition (c).
\begin{rmk} 
\label{rem:1disc} 
\begin{enumerate}
\item In the case when $k=n-1$ the disc $D_S$ is Lagrangian and hence its symplectic normal bundle is zero-dimensional. The last condition is thus equivalent to the requirement that $H=e^{f}\mathbf{n}$, where $\{\mathbf{n}\}$ denotes the chosen frame of the one-dimensional normal bundle $NS$ and where $f \co S \to \RR$.
\item In the case when $k=0$ and $n>1$, there is always a Lagrangian frame that makes $D_S$ into an isotropic surgery disc compatible with the framed sphere $S$. To that end, observe that $U(n-1)$ is connected. Furthermore, up to homotopy relative the boundary, the choice of such a Lagrangian frame lives in $\pi_1 U(n-1) \simeq \ZZ$.
\item In the case when $k=1$ and $n > 3$, the fact that $\pi_1 SO(n-1) \simeq \ZZ_2$ implies that there are exactly two oriented frames of $NS$ up to homotopy, say $F_0$ and $F_1$. Suppose there exists an isotropic surgery disc $D_S$ compatible with the frame $F_0$ of $NS$. In this case, there exists a different Lagrangian frame of the symplectic normal bundle of $D_S$ which is compatible with the frame $F_1$ of $NS$. To see this, observe that the inclusion $SO(n-2) \subset U(n-2)$ induces the trivial map on $\pi_i$, $i \le 1$, and that the inclusion $SO(n-2) \subset SO(n-1)$ is 1-connected.
\end{enumerate}
\end{rmk}

The case when $k<n-1$ will be called the \emph{subcritical case}. Observe that in this case the isotropic surgery $(k+1)$-disc $D_S$ is subcritical. As follows by work of Gromov and Lees, see for example \cite[Section 14.1]{Hprinciple}, there is an h-principle for subcritical isotropic embeddings. In particular, the existence and deformations of such submanifolds can be formulated in purely homotopy-theoretic terms.

\begin{rmk}
\label{rem:hprinciple}
 It follows that the existence of a subcritical surgery disc, and its deformations up to isotopies fixing $L$, are governed by an h-principle.
\end{rmk}

\subsection{The standard model of a Lagrangian handle}
\label{sec:standard-model}
Recall that the jet-space $(J^1M,\lambda_0)$, $\lambda_0:=dz+\theta_M$, has the natural projections
\begin{gather*}\Pi_{\OP{F}} \co J^1M \to M \times \RR,\\
\Pi_{\OP{Lag}}\co J^1M \to T^*M,
\end{gather*}
called the \emph{front projection} and the \emph{Lagrangian projection}, respectively. A Legendrian submanifold of $J^1M$ is determined by its image under $\Pi_{\OP{F}}$ and, up to a translation of the $z$-coordinate, it is determined by its image under $\Pi_{\OP{Lag}}$ as well.

The image $\Pi_{\OP{Lag}}(L) \subset (T^*M,d\theta)$ of a Legendrian submanifold $L \subset J^1M$ is an exact immersed Lagrangian submanifold, and the Reeb chords on $L$ are in bijective correspondence with the self-intersections of $\Pi_{\OP{Lag}}(L)$. In the front projection, Reeb chords correspond to two sheets of $\Pi_{\OP{F}}(L)$ having a common tangency above a given point in the base $M$.

We will use the above projections to construct a Legendrian standard model $L_{0,k} \subset J^1\RR^n$ of a neighbourhood of a framed $k$-sphere $S_{0,k} \subset L_{0,k}$, together with a Legendrian submanifold $L_{\epsilon,k} \subset J^1\RR^n$ that is obtained from $L_{0,k}$ by surgery on $S_{0,k}$. This will be the standard model of a Legendrian ambient surgery.

The construction will also provide an exact Lagrangian cobordism $W_{\epsilon,k} \subset \RR \times J^1\RR^n$ from $L_{0,k}$ to $L_{\epsilon,k}$ which is diffeomorphic to the manifold obtained by a $(k+1)$-handle attachment on $(-\infty,-1] \times L_{0,k}$  along
\[S_{0,k} \subset L_{0,k}=\partial((-\infty,-1] \times L_{0,k}).\]

\subsubsection{Identifying the symplectisation with a cotangent bundle}
Let $(\mathbf{q},\mathbf{p},z)$ be canonical coordinates on $\left(J^1\RR^n,dz-\sum_{i=1}^{n}p_idq_i\right)$, and recall that $t$ is the coordinate of the $\RR$-factor of the symplectisation $\RR \times J^1(\RR^n)$. Also, we endow $T^*(\RR^n \times \RR_{>0})$ with the canonical coordinates
\[((\mathbf{x},x_{n+1}),(\mathbf{y},y_{n+1}))=(((x_1,\hdots,x_n),x_{n+1}),((y_1,\hdots,y_n),y_{n+1})),\]
where $(\mathbf{x},x_{n+1})$ form standard coordinates on the base $\RR^n \times \RR_{>0}$ and where $(\mathbf{y},y_{n+1})$ form coordinates induced by the coframe $\{dx_i\}$. The symplectomorphism
\begin{gather*}
(T^*(\RR^n \times \RR_{>0}),d\theta_{\RR^n \times \RR_{>0}}) \to (\RR \times J^1\RR^n,d(e^t\lambda_0)), \\
((\mathbf{x},x_{n+1}),(\mathbf{y},y_{n+1})) \mapsto (\log {x_{n+1}},(\mathbf{x},\mathbf{y}/x_{n+1},y_{n+1})),
\end{gather*}
is exact, given that we take the primitive of the symplectic form $d\theta_{\RR^n \times \RR_{>0}}$ on $T^*(\RR^n \times \RR_{>0})$ to be
\[ -y_1dx_1-\cdots-y_ndx_n+x_{n+1}dy_{n+1}=\theta_{\RR^n \times \RR_{>0}}+d(x_{n+1}y_{n+1}).\]

Since $-y_1dx_1-\cdots-y_ndx_n+x_{n+1}dy_{n+1}$ and $\theta_{\RR^n \times \RR_{>0}}$ differ by an exact one-form, it follows that the above map also identifies Lagrangian submanifolds of $T^*(\RR^n \times \RR_{>0})$ which are exact with respect to the primitive $\theta_{\RR^n \times \RR_{>0}}$ with Lagrangian submanifolds of the symplectisation $\RR \times J^1\RR^n$ which are exact with respect to the primitive $e^t \lambda_0$. In other words, an exact Lagrangian submanifold of the symplectisation $\RR \times J^1\RR^n$ has a Legendrian lift to $J^1(\RR^n\times \RR_{>0})$, and can thus be represented by the corresponding image in $(\RR^n \times \RR_{>0}) \times \RR$ under the front projection.

Assume that a part of the front projection is given as the graph of a smooth function $f(\mathbf{x},x_{n+1})$ above some region in $\RR^n\times \RR_{>0}$. It thus follows that the corresponding sheet of the Lagrangian submanifold is determined by
\[ \{y_i=\partial_{x_i}f(\mathbf{x},x_{n+1})\} \subset T^*(\RR^n\times\RR_{>0}),\]
on which the pull-back of $e^t \lambda_0$ has a primitive given by
\[-f(\mathbf{x},x_{n+1})+x_{n+1}y_{n+1}+C, \quad C \in \RR.\]
Finally, under above the identification of $T^*(\RR^n\times\RR_{>0})$ with $\RR \times J^1\RR^n$, this sheet corresponds to a subset of a cylinder over a Legendrian submanifold in $J^1\RR^n$ if and only if the above function is of the form
\[f(\mathbf{x},x_{n+1})=g(\mathbf{x})x_{n+1}+C, \quad C \in \RR.\]

\subsubsection{The model for an exact Lagrangian handle $W_{\epsilon,k}$}
\label{sec:model}
Here we construct the standard model $W_{\epsilon,k}$ of an elementary Lagrangian cobordism, where $k \in \{0,\hdots,n-1\}$.

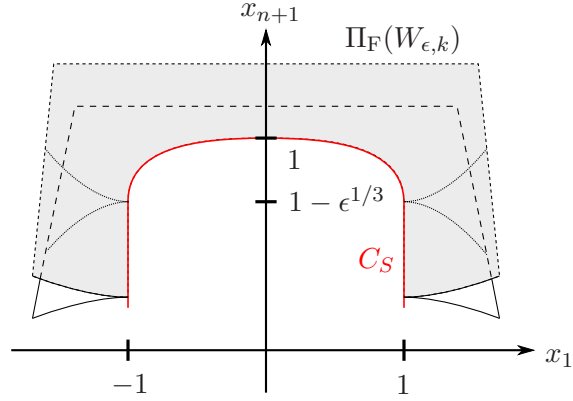
\begin{figure}[htp]
\centering
\labellist
\pinlabel $1-\epsilon^{1/3}$ at 130 73
\pinlabel $1$ at 114 88
\pinlabel $x_{n+1}$ at 105 143
\pinlabel $1$ at 155 3
\pinlabel $-1$ at 52 3
\pinlabel $x_1$ at 214 13
\pinlabel $\color{red}C_S$ at 145 50
\pinlabel $\Pi_{\OP{F}}(W_{\epsilon,k})$ at 155 133
\endlabellist
\includegraphics{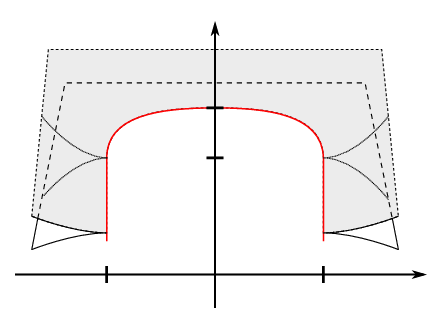}
\caption{A part of the front projection of the standard model $W_{\epsilon,k}$ of an elementary Lagrangian cobordism of index $k+1$ in the case when $k=0$ and $n=1$. The cusp-edge corresponds to the core disc $C_S$.}
\label{fig:cob}
\end{figure}

For each $0<\epsilon<1$, we let $\rho_\epsilon\co \RR \to [0,1]$ be a smooth symmetric function satisfying
\begin{itemize}
\item $\rho_\epsilon(0)=1$,
\item $\rho_\epsilon'(x) \ge 0$ for $x \ge 0$,
\item $\rho_\epsilon(x^2)=0$ for $x^2 \geq 1+(2/3)\epsilon$, and
\item $\frac{d^2}{dx^2}\left(x^2+\rho_\epsilon\left(x^2\right)(1+\epsilon/2)\right)>0$ for all $x \in \RR$.
\end{itemize}
In particular, for each $0 \le C \le 1+\epsilon/2$, the function $x^2+C\rho_\epsilon(x^2)$ is strictly convex, coincides with $x^2$ for $x^2 \geq 1+(2/3)\epsilon$, and takes the value $C$ at $x=0$.

For each $0 <\epsilon<1$ we also consider a smooth non-decreasing function $\sigma_\epsilon\co\RR_{>0} \to \RR_{\ge0}$ satisfying
\begin{itemize}
\item $0 \leq \sigma_\epsilon'(x) \leq (1+\epsilon)\epsilon^{-1/3}$,
\item $\sigma_\epsilon(x)=0$ for $x \leq 1-\epsilon^{1/3}$,
\item $\sigma_\epsilon(x)=1+\epsilon/2$ for $x \geq 1+\epsilon^{1/3}$, and
\item $\sigma_\epsilon(1)=1$, $\sigma_\epsilon'(1)>0$, $\sigma_\epsilon''(1) \ge 0$.
\end{itemize}

Finally, we define
\begin{gather*}
\varphi_\epsilon\co \RR^n \times \RR \to \RR,\\
\varphi_\epsilon(\mathbf{x},x_{n+1})=\rho_\epsilon\left(x_1^2 + \cdots + x_{k+1}^2\right)\sigma_\epsilon(x_{n+1}).
\end{gather*}
We now consider the front in $(\RR^n \times \RR_{>0}) \times \RR$ that above the domain
\[\left\{ \left(x_1^2+\cdots+x_{k+1}^2\right)-\left(x_{k+2}^2+\cdots+x_{n}^2\right)+\varphi_\epsilon(\mathbf{x},x_{n+1}) \geq  1
\right\} \subset \RR^n \times \RR_{>0},\]
is given by the graphs of the two functions $\pm F_\epsilon(\mathbf{x})x_{n+1}$, where
\begin{align*}
 F_\epsilon(\mathbf{x},x_{n+1})  :=  \big(\! &\left(x_1^2+\cdots+x_{k+1}^2\right)\\
 &-\left(x_{k+2}^2+\cdots+x_{n}^2\right)+\varphi_\epsilon(\mathbf{x},x_{n+1})-1\big)^{3/2},
\end{align*}
and which has a cusp-edge along the boundary of the same domain. We use $W_{\epsilon,k} \subset T^*(\RR^n \times \RR_{>0})$ to denote the corresponding (non-compact!) exact Lagrangian submanifold.

Observe that $W_{\epsilon,k}$ can be identified with an exact Lagrangian submanifold of $\RR \times J^1\RR^n$ which is cylindrical over a Legendrian submanifold above the complement of
\[ \left\{ \begin{array}{l}
x_1^2+\cdots+x_{k+1}^2 \leq 1+\epsilon,\\
 1-\epsilon^{1/3} \leq x_{n+1} \leq 1+\epsilon^{1/3} \end{array}\right\}
\subset \RR^n \times \RR_{>0}.
\]
Consequently, the primitive of the pull-back of the one-form
\[-y_1dx_1-\cdots-y_ndx_n+x_{n+1}dy_{n+1}\]
may be taken to vanish in the same set.

We also consider the exact Lagrangian submanifold $W_{0,k}$ corresponding to the following front in $(\RR^n\times \RR_{>0}) \times \RR$. Above the domain
\[\left\{ \left(x_1^2+\cdots+x_{k+1}^2\right)-\left(x_{k+2}^2+\cdots+x_{n}^2\right) \geq 1\right\} \subset \RR^n\times \RR_{>0},\]
we require it to correspond to the graphs of the two functions
\[ \pm \left(\left(x_1^2+\cdots+x_{k+1}^2\right)-\left(x_{k+2}^2+\cdots+x_{n}^2\right)-1\right)^{3/2}x_{n+1},\]
while it has a cusp-edge along the boundary of the same domain. Obviously, $W_{0,k}$ is cylindrical over a Legendrian submanifold, and the primitive of the pull-back of the one-form
\[-y_1dx_1-\cdots-y_ndx_n+x_{n+1}dy_{n+1}\]
may be taken to vanish. We use $L_{0,k}$ to denote this Legendrian submanifold. It follows that $W_{\epsilon,k}$ is an exact Lagrangian cobordism from $L_{0,k} \subset J^1\RR^n$ to a Legendrian submanifold that will be denoted by $L_{\epsilon,k} \subset J^1\RR^n$.

Observe that $L_{\epsilon,k}$ is diffeomorphic to the manifold obtained from $L_{0,k}$ by a surgery on the $k$-sphere
\[S_{0,k}:=\left\{\begin{array}{l} q_1^2 + \cdots + q_{k+1}^2 = 1, \\
q_{k+2}=\cdots =q_{n}=0,\\
\mathbf{p}=0,\\
z=0
\end{array}
\right\} \subset L_{0,k} \subset J^1\RR^n\]
with the choice of frame
\[N_{0,k}:=\langle q_1\partial_{p_1} + \cdots + q_{k+1}\partial_{p_{k+1}},\partial_{q_{k+2}},\hdots,\partial_{q_n} \rangle \]
of its normal bundle. The cobordism $W_{\epsilon,k}$ is diffeomorphic to the manifold obtained by a handle-attachment on $(-\infty,-1] \times L_{0,k}$ along
\[S_{0,k} \subset L_{0,k}=\partial((-\infty,-1] \times L_{0,k}).\]

By applying Formula \ref{eq:maslov}, one sees that the Maslov class vanishes for $W_{\epsilon,k}$, and hence the same is true for both $L_{0,k}$ and $L_{\epsilon,k}$.

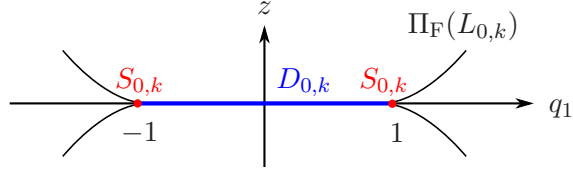
\begin{figure}[htp]
\centering
\labellist
\pinlabel $z$ at 100 60
\pinlabel $1$ at 150 13
\pinlabel $-1$ at 53 13
\pinlabel $\color{red}S_{0,k}$ at 145 32
\pinlabel $\color{red}S_{0,k}$ at 53 32
\pinlabel $\Pi_{\OP{F}}(L_{0,k})$ at 175 53
\pinlabel $\color{blue}D_{0,k}$ at 115 32
\pinlabel $q_1$ at 212 22
\endlabellist
\includegraphics{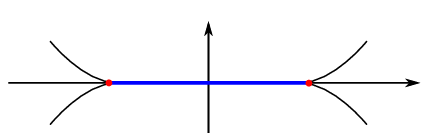}
\caption{The front projection of the standard model $L_{0,k}$ before the surgery in the case when $k=0$ and $n=1$.}
\label{fig:before}
\end{figure}

\begin{figure}[htp]
\centering
\labellist
\pinlabel $z$ at 100 60
\pinlabel $\color{red}c_0$ at 111 36
\pinlabel $-1-\epsilon/2$ at 4 10
\pinlabel $1+\epsilon/2$ at 192  10
\pinlabel $\Pi_{\OP{F}}(L_{\epsilon,k})$ at 175 53
\pinlabel $q_1$ at 212 22
\endlabellist
\includegraphics{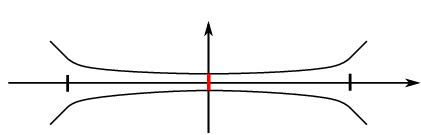}
\caption{The front projection of the standard model $L_{\epsilon,k}$ after the surgery in the case when $k=0$ and $n=1$.}
\label{fig:after}
\end{figure}

\subsubsection{The Reeb chords on $L_{0,k}$ and $L_{\epsilon,k}$}
\label{sec:newreeb}
Note that $L_{0,k}$ has no Reeb-chords, and that $L_{\epsilon,k}$ has exactly one Reeb chord $c_0$ above $\mathbf{q}=0$. One computes
\[ \ell(c_0)=\epsilon^{3/2}/\sqrt{2}.\]
Furthermore, this Reeb chord satisfies
\[|c_0|=n-k-1,\]
which can be seen by applying Formula \ref{eq:cz}. To that end, we observe that the front projection $\Pi_{\OP{F}}(L_{\epsilon,k})$ consists of two sheets being the graphs of the functions $f_s<0<f_e$, and that the critical point of $f_s-f_e$ has Morse index $n-(k+1)$ (see Figure \ref{fig:after}). Furthermore, in the case $n>1$, the obvious capping path of $c_0$ traverses exactly one cusp-edge in downward direction.

\subsubsection{The standard model of the isotropic surgery disc}
\label{sec:stdmodel}
Consider the embedded isotropic disc
\[D_{0,k}:=\left\{\begin{array}{l}q_1^2 + \cdots + q_{k+1}^2 \leq 1, \\
q_{k+2}=\cdots =q_{n}=0,\\
\mathbf{p}=0,\\
z=0
\end{array}
\right\} \subset J^1\RR^n\]

with boundary $\partial D_{0,k}=S_{0,k}$. If we endow the symplectic normal bundle $(TD_{0,k})^{d\lambda_0}/TD_{0,k}$ with the Lagrangian frame
\[\{ \partial_{q_{k+2}},\hdots,\partial_{q_n} \},\]
it is easily seen that $D_{0,k}$ becomes an isotropic surgery disc compatible with the above frame of the normal bundle of $S_{0,k} \subset L_{0,k}$.

\subsection{Controlling the size of the standard handle}
\label{sec:nbhd}

This subsection deals with controlling the ``size'' of the non-cylindrical part of the standard model of the handle $W_{\epsilon,k} \subset T^*(\RR^n \times \RR_{>0})$, which obviously depends on the parameter $\epsilon>0$. First, this will be necessary when defining the Legendrian ambient surgery construction, which is performed by importing this standard model into the given contact manifold. In later sections this will also turn out to be important when providing estimates the of symplectic areas of pseudo-holomorphic discs having boundary on the handle.

Consider the neighbourhood
\[
\widetilde{U}_\epsilon:= 
\left\{ \begin{array}{l}
x_1^2 + \cdots + x_{k+1}^2 < 1+\epsilon,\\
x_{k+2}^2 + \cdots + x_{n}^2 < 2\epsilon, \\
|y_i|< 3\sqrt{\epsilon(1+\epsilon)}x_{n+1},\quad i \leq n, \\
|y_{n+1}|< 10\epsilon^{1/6}
\end{array}
\right\}
\subset T^*(\RR^n \times \RR_{>0})
\]
which is invariant under translation of the $t$-coordinate of $\RR \times J^1\RR^n$ (using the above identification). The following lemma in particular shows that this neighbourhood contains the non-cylindrical part of the handle $W_{\epsilon,k}$.
\begin{lem}
\label{lem:nbhd}
For each $0<\epsilon<1$, $W_{\epsilon,k}$ and $W_{0,k}$ coincide outside of the open set $\widetilde{U}_{\epsilon'}$ given that $\epsilon'\geq (2/3)\epsilon$.
\end{lem}
\begin{proof}
We show that $W_{\epsilon,k}$ and $W_{0,k}$ coincide outside of the set $\widetilde{U}_{(2/3)\epsilon}$, from which the general statement obviously follows. We thus fix $0<\epsilon<1$ and set $\epsilon':=(2/3)\epsilon$.

Consider the function
\[f_y(x):=x^2+y\rho_\epsilon(x^2),\]
where $\rho_\epsilon$ is as defined in Section \ref{sec:standard-model}. Recall that $f_y(x)$ is symmetric in $x$, satisfies $f_y(x)=x^2$ for $x^2 \geq 1+\epsilon'$, and that $f_y''(x)>0$ holds for each $0 \leq y \leq 1+\epsilon/2$. From this we conclude that the inequality
\[y=f_y(0) \leq f_y(x) < 1+\epsilon'\]
holds for all $x^2 < 1+\epsilon'$, $0 \leq y \leq \eta+\epsilon/2$ or, in other words, that
\[0\le y = f_y(0) \leq x^2+y\rho_\epsilon\left(x^2\right)-1 < \epsilon'\]
holds in the same set.

This inequality implies that we have the inclusion
\begin{equation}
\label{eq:inclusion}
W_{\epsilon,k} \cap \left\{x_1^2 + \cdots + x_{k+1}^2 < 1+\epsilon' \right\} \subset \left\{x_{k+2}^2 + \cdots + x_{n}^2 < \epsilon' \right\}\end{equation}
and hence that the inequality
\begin{equation}
\label{eq:Fbound}
F_\epsilon(\mathbf{x},x_{n+1}) < (\epsilon')^{3/2}
\end{equation}
holds on the same set.

Furthermore, since $\rho_\epsilon(x^2)$ vanishes for $x^2 \geq 1+\epsilon'$, it follows that $W_{\epsilon,k}$ and $W_{0,k}$ coincide outside of the set
\[\left\{x_1^2 + \cdots + x_{k+1}^2  < 1+\epsilon' \right\},\]
and by (\ref{eq:inclusion}), outside of
\[ O := \left\{ \begin{array}{l} 
x_1^2 + \cdots + x_{k+1}^2 < 1+\epsilon', \\
x_{k+2}^2 + \cdots + x_{n}^2 < \epsilon'
\end{array}
\right\}.\]

Since $f_y(x)$ is symmetric in $x$ and has increasing derivative by assumption, and since $f_y'(\pm\sqrt{1+\epsilon'})=\pm2\sqrt{1+\epsilon'}$, the inequality
\begin{equation}
\label{eq:fprimbound}
|f_y'(x)| < 2\sqrt{1+\epsilon'}
\end{equation}
holds for all $ x^2 < 1+\epsilon'$.

Using (\ref{eq:Fbound}) and (\ref{eq:fprimbound}) one computes that
\begin{equation}\left|\frac{\partial}{\partial x_i}F_\epsilon(\mathbf{x},x_{n+1})\right|\leq (3/2)F_\epsilon(\mathbf{x},x_{n+1})^{1-2/3}2\sqrt{1+\epsilon'} < 3\sqrt{\epsilon'}\sqrt{(1+\epsilon')},\end{equation}
holds inside $O$ for each $i=1, \hdots, k+1$, and that
\begin{equation}\label{eq:boundp} \left|\frac{\partial}{\partial x_i}F_\epsilon(\mathbf{x},x_{n+1})\right|\leq (3/2)F_\epsilon(\mathbf{x},x_{n+1})^{1-2/3}2\sqrt{\epsilon'} < 3\sqrt{\epsilon'}\sqrt{(1+\epsilon')}\end{equation}
holds inside $O$ for each $i=k+2, \hdots, n$. In conclusion, $W_{\epsilon,k} \cap O$ is contained in the set
\[\left\{|y_i| < 3\sqrt{\epsilon'(1+\epsilon')}x_{n+1},\ 1 \leq i\leq n \right\}.\]

Finally, from the inequality
\[0 \leq \sigma_\epsilon'(x_{n+1})\leq (1+\epsilon)\epsilon^{-1/3},\]
together with the fact that $\sigma_\epsilon'(x_{n+1})=0$ holds for $x_{n+1} \ge 1+\epsilon^{1/3}$, we get the bound

\begin{align*}
 \left|\frac{\partial}{\partial x_{n+1}}(F_\epsilon(\mathbf{x},x_{n+1})x_{n+1})\right| & \leq  F_\epsilon(\mathbf{x},x_{n+1})+\frac{3}{2}F_\epsilon(\mathbf{x},x_{n+1})^{1-2/3}\sigma'_\epsilon(x_{n+1})x_{n+1} \\
& <  (\epsilon')^{3/2}+\frac{3}{2}(\epsilon')^{1/2}((1+\epsilon)\epsilon^{-1/3})(1+\epsilon^{1/3})\\
& <  (\epsilon')^{1/6}+6(\epsilon')^{1/2}\epsilon^{-1/3}\\
& =  (\epsilon')^{1/6}+6(\epsilon')^{1/2}((3/2)\epsilon')^{-1/3}\\
& <  (\epsilon')^{1/6}+ 6(\epsilon')^{1/6}
\end{align*}
on $O$, where we again have used the inequality (\ref{eq:Fbound}). In other words, $W_{\epsilon,k} \cap O$ is contained inside
\[ \left\{ |y_{n+1}| < 7(\epsilon')^{1/6} \right\}.\]
In conclusion, we have shown that $W_{\epsilon,k}$ and $W_{0,k}$ coincide outside of the set~$\widetilde{U}_{\epsilon'}$.
\end{proof}

\subsection{The definition of a Legendrian ambient surgery}
\label{sec:defcobordism}
Suppose that we are given a Legendrian submanifold $L \subset Y$ containing a framed embedded $k$-sphere $S \subset L$, together with a compatible isotropic surgery $(k+1)$-disc $D_S\subset Y$. We use $NS \subset TL|_S$ to denote the normal bundle of $S$.

We will construct an exact Lagrangian cobordism $V_{S,\epsilon} \subset \RR \times Y$ from $L$ to a Legendrian submanifold which will be denoted by $L_S$, and which is diffeomorphic to a manifold obtained from $L$ by $k$-surgery on $S$ with the above frame of $NS$. Furthermore, the cobordism $V_{S,\epsilon}$ will be diffeomorphic to the elementary cobordism of index $k+1$ induced by the surgery. We will say that $L_S$ is obtained from $L$ by a Legendrian ambient surgery on $S$.

\subsubsection{Identifying the isotropic surgery disc with the standard model}
\label{sec:standard-nbhd}
There is a neighbourhood $U \subset J^1\RR^n$ of $D_{0,k}$ for which there is a contact-form preserving diffeomorphism
\[ \phi \co (U,\lambda_0) \to (\phi(U),\lambda) \]
identifying $U$ with a neighbourhood $\phi(U) \subset Y$ of $D_S$. We can moreover require that
\begin{enumerate}

\item $\phi(D_{0,k})=D_S$, and $\phi(S_{0,k})=S$,

\item $D\phi$ maps the Lagrangian frame of $(TD_{0,k})^{d\lambda_0}/TD_{0,k}$ in the standard model (as given in Section \ref{sec:stdmodel}) to the choice of Lagrangian frame of $(TD_S)^{d\lambda}/TD_S$ for the isotropic surgery disc, and
\item $\phi(L_{0,k} \cap U)=L \cap \phi(U)$.
\end{enumerate}
The neighbourhood theorem \cite[Theorem 6.2.2]{IntroContact} for isotropic submanifolds gives neighbourhoods and a map $\phi$ as above satisfying (1) and (2). We may furthermore assume that (3) holds infinitesimally, i.e.~that
\[D\phi(TL_{0,k}|_{S_{0,k}})=TL|_S.\]
After a perturbation of $L$ by a Legendrian isotopy, and after possibly shrinking the above neighbourhoods, we may hence assume that (3) is satisfied as well.

\subsubsection{The constructions}
\label{sec:constructV}
The identification of a neighbourhood of $D_S$ with a neighbourhood of $D_{0,k}$ via the contact-form preserving map $\phi$ constructed above induces an exact symplectomorphism of the form
\[(\Id_\RR,\phi) \co (\RR \times U,d(e^t\lambda_0)) \to (\RR \times \phi(U),d(e^t\lambda))\]
from the neighbourhood
\[\RR \times U \supset \RR \times D_{0,k}\]
in $\RR \times J^1\RR^n$
to the neighbourhood
\[ \RR \times \phi(U) \supset \RR \times D_S \]
in $\RR \times Y$. Observe that, by construction, this symplectomorphism identifies $W_{0,k}$ with $\RR \times L$.

Recall the identification of $\RR \times J^1\RR^n$ with $T^*(\RR^n \times \RR_{>0})$ described in Section \ref{sec:standard-model}. After choosing a small enough $\epsilon>0$, we may assume that the cylindrical neighbourhood
\[\widetilde{U}_\epsilon \subset T^*(\RR^n \times \RR_{>0})\]
of $\RR \times D_{0,k}$ described in Section \ref{sec:nbhd} is identified with a neighbourhood $\RR \times U_\epsilon \subset \RR \times U$ under this identification.

The two exact Lagrangian cobordisms $W_{0,k},W_{\epsilon,k} \!\subset\! \RR \!\times\! J^1\RR^n$ constructed in Section \ref{sec:model} coincide outside of the set $\widetilde{U}_\epsilon$ as follows by Lemma \ref{lem:nbhd}. Replacing the set $(\RR \times L) \cap (\RR \times \phi(U_\epsilon))$ with $(\Id_\RR,\phi)(W_{\epsilon,k} \cap (\RR \times U_\epsilon))$ thus produces an exact Lagrangian cobordism that will be denoted by
\[V_{S,\epsilon} \subset \RR \times Y.\]
This is an exact Lagrangian cobordism from $L \subset Y$ to a Legendrian submanifold that we will denote by
\[L_{S,\epsilon} \subset Y.\]
Observe that $L_{S,\epsilon}$ is diffeomorphic to the manifold obtained by performing a $k$-surgery on $S$ with the specified frame of its normal bundle, and that $V_{S,\epsilon}$ is diffeomorphic to the corresponding elementary cobordism of index $k+1$.

Moreover, we will fix a choice of an embedding of the core disc
\[C_{S,\epsilon} := (\Id_\RR, \phi)(W_{\epsilon,k} \cap \{ \mathbf{p}=0, t \le 0\} ) \subset V_{S,\epsilon}\]
of the handle-attachment. It is easily checked that $C_{S,\epsilon} \subset \RR \times V_{S,\epsilon}$ is a smoothly embedded $(k+1)$-disc that coincides with $(-\infty,-1) \times S$ outside of a compact set (also, see Figure \ref{fig:cob}).

\begin{defn}
\label{def:main}
Given an isotropic surgery $(k+1)$-disc $D_S \subset Y$ compatible with a framed embedded $k$-sphere $S \subset L \subset Y$ inside a Legendrian manifold, we will say that the Legendrian submanifold $L_{S,\epsilon} \subset Y$ constructed above is obtained from $L$ by a \emph{Legendrian ambient surgery on $S$}. Furthermore, the construction provides the exact Lagrangian cobordism $V_{S,\epsilon} \subset \RR \times Y$ from $L$ to $L_{S,\epsilon}$, which will be called an \emph{elementary Lagrangian cobordism of index $k+1$}. 
\end{defn}

The above construction can be seen to satisfy the following properties.
\begin{rmk}
\label{rem:ambientsurgery}
\begin{enumerate}
\item The new Reeb chord on $L_S$ that corresponds to the unique Reeb chord on the local model $L_{0,k}$ will be denoted by $c_S$ and satisfies
\begin{gather*}
|\ell(c_S)|=n-k-1,\\
\ell(c_S)=\epsilon^{3/2}/\sqrt{2},
\end{gather*}
as follows from the description in Section \ref{sec:newreeb}.
\item Changing the parameter $\epsilon>0$ above does not change the Hamiltonian isotopy class of $V_{S,\epsilon}$, nor the Legendrian isotopy class of $L_{S,\epsilon}$. We will sometimes omit $\epsilon$ from the notation.
\item Assume that we are given two embedded spheres $S_i \subset L_i \subset Y$, $i=1,2$, with compatible isotropic surgery discs $D_{S_i}$. If there exists a contact isotopy taking $L_1 \cup D_{S_1}$ to $L_2 \cup D_{S_2}$, whose tangent-map moreover takes the Lagrangian frame of the symplectic normal bundle of $D_{S_1}$ to that of $D_{S_2}$,  it follows that $(L_1)_{S_1}$ is Legendrian isotopic to $(L_2)_{S_2}$.
\item In the case $k<n-1$, the isotropic surgery disc $D_S$ is subcritical. This means that, for generic data, there are no Reeb chords on $L \cup D_S$  starting or ending on $D_S$ (here we use the assumption that the periodic Reeb orbits of $(Y,\lambda)$ are non-degenerate). Given any $E>0$, after shrinking $\epsilon>0$, we may thus assume that there is a natural identification
\[ \{ c \in \mathcal{Q}(L_S); \: \ell(c) \le E \}=\{ c \in \mathcal{Q}(L); \: \ell(c) \le E \} \cup \{ c_S\}\]
of Reeb chords. In Section \ref{sec:contracting} we show that this also can be achieved in the case when $k=n-1$, given that we first deform $L$ by a Legendrian isotopy.
\item It can be seen that the symplectisation-coordinate $t$ restricted to $V_{S,\epsilon}$ is a Morse function with a unique critical point of index $k+1$. Our choice of core disc $C_S$ moreover passes through this critical point.
\end{enumerate}
\end{rmk}

\subsection{The effect of a Legendrian ambient surgery on\\ the Maslov class}
\label{sec:maslov}

The Maslov class of a Lagrangian cobordism pulls back to the Maslov classes of its respective Legendrian ends under the inclusion maps. It follows that
\[ \mu(H_1(L)), \mu(H_1(L_S)) \subset \mu(H_1(V_S)) \subset \ZZ. \]

\subsubsection{The case of a 0-surgery}
Suppose we are given an embedded isotropic 1-disk $D_S \subset Y$ that bounds $S \subset L$ and satisfies (a) and (b) in Definition \ref{def:surgerydisc}. We fix an arbitrary frame of the normal bundle of $S \subset L$. When $n>1$ we can always find a Lagrangian frame of the symplectic normal bundle of $D_S$ which makes it into an isotropic surgery disc compatible with $S$, as explained in part (2) of Remark \ref{rem:1disc}. Furthermore, the choice of such a Lagrangian frame lives in
\[\pi_1(U(n-1))\simeq \ZZ.\]

Under the additional assumption that the 1-handle attachment adds a generator $\gamma \in H_1(V_S)$, the following can be said. Consider two choices $m_1,m_2 \in \ZZ \simeq \pi_1(U(n-1))$ of Lagrangian frames of $D_S$. The two corresponding Legendrian ambient surgeries give rise to two diffeomorphic cobordisms, and the evaluation of the Maslov class on $\gamma$ for these two choices can be seen to differ by $2(m_1-m_2)$.

\subsubsection{The case of a $k$-surgery with $0<k<n-1$}
Let $X \cup (D^{k+1} \times D^{n-k})$ be a $(k+1)$-handle attachment on the manifold $X$ with non-empty boundary. From the associated long exact sequence in singular homology
\begin{align*}
\cdots &\rightarrow H_{2}(D^{k+1}\times D^{n-k},S^{k}\times D^{n-k}) \rightarrow H_1(X) \\
&  \rightarrow H_1(X \cup (D^{k+1} \times D^{n-k})) \rightarrow H_{1}(D^{k+1}\times D^{n-k},S^{k}\times D^{n-k})  \\
&  \rightarrow H_0(X) \rightarrow \cdots, 
\end{align*}

together with the fact that
\[H_{1}(D^{k+1}\times D^{n-k},S^{k}\times D^{n-k}) \simeq H_1(S^{k+1})\]
we conclude that the above map
\[H_1(X) \rightarrow H_1(X \cup (D^{k+1} \times D^{n-k}))\]
is surjective, since $k>0$.

Setting $X=(-\infty,-1] \times L$ and identifying $X \cup (D^{k+1} \times D^{n-k})$ with $V_S$, we conclude that
\[\mu(H_1(L))=\mu(H_1(V_S)) \subset \ZZ\]
whenever $L_S$ is obtained from $L$ by an ambient $k$-surgery with $k>0$.

Since $V_S$ equivalently can be viewed as a smooth manifold obtained by a $(n-k)$-handle attachment on $(-\infty , -1] \times L_S$, the analogous long exact sequence now shows that
\[\mu(H_1(L_S))=\mu(H_1(L))=\mu(H_1(V_S)) \subset \ZZ,\]
whenever $L_S$ is obtained from $L$ by a Legendrian ambient $k$-surgery with $0<k<n-1$.

\subsection{A radial contraction of the isotropic surgery disc}
\label{sec:contracting}
Assume that we are given an isotropic surgery disc $D_S$ of dimension $k+1$ which is compatible with a framed embedded sphere $S \subset L$ inside a Legendrian submanifold. Let $D_{S'} \subset D_S$ be a smooth embedding of a closed $(k+1)$-disc. Assume that there is an isotopy of $D_S$ taking $S$ to the boundary $\partial D_{S'}$.

This isotopy can be extended to a contact isotopy of the ambient contact manifold $Y$ (see e.g.~Lemma \ref{lem:ham}) having support in some arbitrarily small neighbourhood of $D_S$. We use $(L',S')$ to denote the image of $(L,S)$ induced by this isotopy, where $L'$ thus is Legendrian isotopic to $L$ and for which $S'=\partial D_{S'}$. Moreover, choosing this extension with some care, a Legendrian ambient surgery on $S' \subset L'$ determined by the isotropic surgery disc $D_{S'}$ may in fact be assumed to reproduce our original Legendrian ambient surgery, that is:
\[ (L')_{S'}=L_S.\]

Take any $E>0$. After choosing $D_{S'} \subset D_S$ to be sufficiently small, we may assume that all Reeb chords on $L' \cup D_{S'}$ that start or end on $D_{S'}$ have action at least $E$. We conclude that
\begin{lem}
\label{lem:contracting}
For $0 \le k \le n-1$ we may assume that $L_S=(L')_{S'}$ is obtained from $L'$ by a Legendrian ambient $k$-surgery on $S'$, where $(L',S')$ is isotopic to $(L,S)$ by a contact isotopy, and for which
\[ \{ c \in \mathcal{Q}(L_S); \: \ell(c) \le E \}=\{ c \in \mathcal{Q}(L'); \: \ell(c) \le E \} \cup \{ c_S\}.\]
\end{lem}
Observe that when $k<n-1$ this result also follows by a general position argument; see part (4) of Remark \ref{rem:ambientsurgery}.

\subsection{Well-definedness of the Legendrian ambient 0-surgery}

It was proven in \cite[Lemma 3.2]{OnConnectedSum} that cusp connected sum is a well-defined operation on Legendrian knots. Since Legendrian ambient 0-surgery is a generalisation of cusp connected sum, the below proposition provides a positive answer to a question posed in \cite{NonIsoLeg} concerning the well-definedness of cusp connected sums in higher dimensions.

For simplicity we here only consider the case when the contact manifold is the standard contact $(2n+1)$-space $(\CC^n \times \RR,dz-\sum_i y_idx_i)$. We also restrict ourselves to the case $n>1$. Recall that, given an framed embedded 0-sphere $S \subset L$, we can always find a compatible isotropic surgery disc $D_S$ (see part (2) of Remark \ref{rem:1disc}).

\begin{prop}\label{prop:well-def}
Let $L \subset \CC^n \times \RR$ be a Legendrian submanifold, where $n>1$, and let $S \subset L$ a framed embedded 0-sphere.
\begin{enumerate}
\item Suppose that $L \cap \{x_1=0\}=\emptyset$ and that the two points $S \subset L$ are separated by the hyperplane $\{x_1=0\} \subset \CC^n \times \RR$. It follows that the Legendrian isotopy class of $L_S$ is invariant under isotopy of the framed 0-sphere $S$ and independent of the choice of a compatible isotropic surgery disc $D_S$.
\item If $L$ is connected, then the Legendrian isotopy class of $L_S$ depends only on the Lagrangian frame of the symplectic normal bundle of the isotropic surgery disc $D_S$.
\end{enumerate}
\end{prop}
\begin{proof}
Consider the following set-up. Let $S,S' \subset L$ be isotopic framed embeddings of a 0-sphere, and let $D_S$ and $D_{S'}$ be isotropic surgery 1-discs compatible with $S$ and $S'$, respectively.

Since contact isotopies can be seen to act transitively on points in a Legendrian submanifold (see Lemma \ref{lem:ham}), after a contact isotopy we may assume that $S=S'$ and that the frames of $NS=NS'$ agree. This also shows that the outward-normal vector fields of the two isotropic surgery discs $D_S$ and $D_{S'}$ may be assumed to coincide. Moreover, after performing the contact isotopy with some additional care, we may assume that the isotropic surgery discs themselves agree in some neighbourhood of $S \subset Y$.

Here the ambient contact manifold is of dimension $2n+1 \ge 5$, and it thus follows by the h-principle for subcritical isotropic embeddings that the isotropic surgery discs are isotopic by a contact isotopy supported outside of some small neighbourhood of $L$ (also, see Remark \ref{rem:hprinciple}).

In view of the above, what remains is to investigate how the construction depends on the choice of a Lagrangian frame of the symplectic normal bundle of the isotropic surgery disc $D_S$. Recall that two Lagrangian frames of $D_S$ that are homotopic relative the boundary of $D_S$ determine Legendrian ambient surgeries whose resulting Legendrian submanifolds are Legendrian isotopic (see part (3) of Remark \ref{rem:ambientsurgery}). Also, these homotopy classes are in bijection with $\pi_1(U(n-1)) \simeq \ZZ$.

(1): By assumption we have $L \cap \{ -\epsilon < x_1 < \epsilon \} = \emptyset$ for some $\epsilon>0$, where the two components of $L$ containing the respective component of $S$ are separated by the hyperplane $\{x_1=0\}$. Without loss of generality, we may thus assume that
\[D_S \cap \{ -\epsilon < x_1 < \epsilon \} = \left\{\begin{array}{l}|x_1|< \epsilon, \\ x_2=\cdots=x_n=0, \\ y_1=\cdots=y_n=0, \\z=0 \end{array}  \right\}.\]

Consider the autonomous Hamiltonian
\begin{gather*}
H \co \CC^n \to \RR, \\
H(\mathbf{x},\mathbf{y}):=\psi(x_1)(x_2^2+y_2^2)/2,
\end{gather*}
where $\psi(x_1) \ge 0$ is a smooth function satisfying $\psi(x_1)=1$ for $x_1<0$ and $\psi(x_1)=0$ for $x_1 \ge \epsilon$. The corresponding Hamiltonian isotopy
\[\phi^t_H \colon \left(\CC^n,\sum_i dx_i\wedge dy_i\right) \to \left(\CC^n,\sum_i dx_i\wedge dy_i\right)\]
has a unique lift to a contact-form preserving isotopy
\begin{gather*}
\varphi^t \colon (\CC^n \times \RR,\lambda_0) \to (\CC^n \times \RR,\lambda_0),\\
\varphi^t(\mathbf{x},\mathbf{y},z) = (\phi^t_H(\mathbf{x},\mathbf{y}),z+g_t(\mathbf{x},\mathbf{y})),
\end{gather*}

satisfying $\varphi^0 = \Id$ and $\dot{g}_t=-\theta_{\RR^n}(\dot{\phi}^t_H)-H\circ \phi^t_H$. We also compute that
\begin{align*}
\varphi^{2\pi l}|_{\{ x_1 \le 0 \}}&=\Id|_{\{ x_1 \le 0 \}}, \quad  l \in \ZZ,\\
\varphi^t|_{\{ x_1 \ge \epsilon \}}&=\Id|_{\{ x_1 \ge \epsilon \}}, \quad t \in \RR,
\end{align*}

while $\varphi^t$ can be seen to preserve $D_S \cap \{ -\epsilon < x_1 < \epsilon \}$ for all $t \in \RR$ (clearly $dH$ vanishes along the latter arc, and $(\phi^{2\pi l}_H)|_{\{ x_1 \le 0 \}}=\Id$).

The subgroup $\{\varphi^{2\pi l}\}_{l\in \ZZ} \simeq \ZZ$ of contactomorphisms (all of which are contact isotopic to the identity by construction) hence fix $L \cup D_S$, and can moreover be seen to act transitively on the Lagrangian frames of the symplectic normal bundle of $D_S$ (relative its boundary).

(2): The computations in Section \ref{sec:maslov} show that different choices of homotopy classes of Lagrangian frames induce Legendrian embeddings of $L_S \subset Y$ that have different Maslov classes with respect to a fixed parametrisation of $L_S$. It follows that the parametrised Legendrian isotopy class of $L_S \subset Y$ indeed depends on the homotopy class of this Lagrangian frame.
\end{proof}

\section{Pseudo-holomorphic discs with boundary on the elementary cobordism}
The goal in this section is computing the DGA morphism induced by an elementary Lagrangian cobordisms $V_S$ from $L$ to $L_S$, and thereby proving Theorem \ref{thm:surjection}. This is done by analysing the behaviour of pseudo-holomorphic discs in $\RR \times Y$ having boundary on $V_S$. To that end it will be crucial to control the ``size'' of the non-cylindrical part of the handle, which we do by considering the family $V_{S,\epsilon}$ depending on the parameter $\epsilon>0$. We start by recalling the construction of $V_S$, from which the existence of this one-parameter family is immediate.

\subsection{A one-parameter family of elementary Lagrangian cobordisms}
\label{sec:oneparam}
Assume that we are given the Legendrian ambient surgery data consisting of the isotropic surgery disc $D_S$ with boundary on $S \subset L$, where $S$ is a framed embedded $k$-sphere inside the $n$-dimensional Legendrian submanifold $L\subset (Y,\lambda)$.

The corresponding elementary Lagrangian cobordism is constructed in Section \ref{sec:constructV} by, starting from the trivial cobordism $\RR \times L$, excising a cylindrical neighbourhood of $\RR \times D_S \subset \RR \times Y$ and replacing it with the standard model of the handle. In order to construct the one-parameter family of elementary Lagrangian cobordisms, it will be necessary to recall some details of this construction.

First, as in Section \ref{sec:nbhd}, we consider the neighbourhood
\[
\RR \times U_{\epsilon_S} := 
\left\{ \begin{array}{l}
q_1^2 + \cdots + q_{k+1}^2 < 1+\epsilon_S,\\
q_{k+2}^2 + \cdots + q_{n}^2 < 2\epsilon_S, \\
|p_i|< 3\sqrt{\epsilon_S(1+\epsilon_S)},\quad i \leq n, \\
|z|< 10\epsilon_S^{1/6}
\end{array}
\right\}
\subset \RR \times J^1\RR^n
\]
for any $\epsilon_S>0$. The main feature of $\RR \times U_{\epsilon_S} $ is that it contains the non-cylindrical part of the standard model $W_{\epsilon,k} \subset \RR \times J^1\RR^n$ of the handle for each $0 < \epsilon \le \epsilon_S$ (see Lemma \ref{lem:nbhd}).

For $\epsilon_S>0$ sufficiently small, the Legendrian ambient surgery data determines an exact symplectomorphism
\[ (\Id_\RR,\phi) \co \RR \times U_{\epsilon_S} \to \RR \times \phi(U_{\epsilon_S}),\]
where $\phi(U_{\epsilon_S}) \subset Y$ is a neighbourhood of $D_S$ (see Section \ref{sec:constructV}). Using this identification, the sought one-parameter family of elementary Lagrangian cobordisms
\[ V_{S,\epsilon} \subset (\RR \times Y,d(e^t\lambda)), \quad 0<\epsilon \le \epsilon_S, \]
from $L$ to $L_{S,\epsilon}$ is now defined as follows:
\begin{itemize}
\item We have an equality
\[V_{S,\epsilon} \setminus (\Id_\RR,\phi)(\RR \times U_{\epsilon_S})=(\RR \times L) \setminus (\Id_\RR,\phi)(\RR \times U_{\epsilon_S})\]
of cylindrical subsets, while,
\item $V_{S,\epsilon} \cap (\Id_\RR,\phi)(\RR \times U_{\epsilon_S}) = (\Id_\RR,\phi)(W_{\epsilon,k})$ is the image of the standard model of the handle.
\end{itemize}
Here Lemma \ref{lem:nbhd} is crucial. Finally, we write $V_S := V_{S,\epsilon_S}$.

We will later need to produce estimates of the symplectic areas of the pseudo-holomorphic discs entering the handle. In order to do this, we want to fix an embedding of the normal sphere-bundle of $S \subset L$ (i.e. an embedding of the unit normal-bundle induced by some choice of metric) that moreover is contained in a subset over which the cobordisms $V_{S,\epsilon}$, $0<\epsilon\le \epsilon_S$, all are cylindrical. More precisely, we fix an open embedding
\begin{gather*}
\Sigma \hookrightarrow L \setminus \phi(U_{\epsilon_S}) \subset Y,\\
\Sigma \simeq S \times S^{n-k-1} \times (-1,1),
\end{gather*}
for which the induced embedding of $\Sigma_0:=S \times S^{n-k-1} \times \{0\}$ can be identified with an embedding of the normal sphere-bundle of $S \subset L$.

\subsubsection{Assumptions on the Reeb dynamics near the surgery region}
\label{sec:Reebassumptions}
We will in the following only consider Reeb chords on $L$ below some fixed action $E>0$. We furthermore assume that all Reeb chords on $L \cup D_S$ that start or end on $D_S$ have action at least $3E$. When $k<n-1$ this can be arranged by a general position argument, while for $k=n-1$ one may have to first perform a Legendrian isotopy as described in Section \ref{lem:contracting}. Given this assumption, it is now possible to choose $\epsilon_S>0$ sufficiently small, so that
\[ \{ c \in \mathcal{Q}(L_{S,\epsilon}); \: \ell(c) \le E \}=\{ c \in \mathcal{Q}(L); \: \ell(c) \le E \} \cup \{ c_S\}\]
holds for each $0 < \epsilon \le \epsilon_S$.

Also, the above assumption has the following important consequence. Recall the above choice of neighbourhood $\Sigma \subset L$ of the normal sphere-bundle of $S \subset L$. After shrinking $\epsilon_S>0$ further (and choosing $\Sigma$ even closer to $S$), the Reeb flow $\phi^t_R$ on $(Y,\lambda)$ can be assumed to have the property that
\[\phi^{[-2E,2E] \setminus \{0\}}_R(\Sigma) \cap (L \cup \phi(U_{\epsilon_S}))=\emptyset\]
holds for the above constant $E>0$.

\subsection{Preliminaries}
\label{sec:prel}
We start by recalling the formulas for the different forms of energies of pseudo-holomorphic discs in $\RR \times Y$ having boundary on $V_{S,\epsilon}$, $0<\epsilon \le \epsilon_S$. While some of the results will be valid for an arbitrary cylindrical almost complex structure $J_{\OP{cyl}}$, for others we will need to use a cylindrical almost complex structure $J_S$ of $\RR \times Y$ which has been constructed with some special care. This is also done below.

\subsubsection{The energies of pseudo-holomorphic discs with boundary on $V_S$}
By construction, $V_{S,\epsilon}$ is cylindrical outside of the set $I_\epsilon \times Y$ with
\[I_\epsilon=[\log{(1-\epsilon^{1/3})} ,\log{(1+\epsilon^{1/3})} ].\]
Furthermore, the primitive of $e^t\lambda$ pulled back to $V_{S,\epsilon}$ may be supposed to vanish outside of a compact set. Let $J_{\OP{cyl}}$ denote a cylindrical almost complex structure on $\RR \times Y$. Proposition \ref{prop:energy} can be applied, giving us
\begin{align}
\label{eq:denergy}
0 & \leq  E_{d(\varphi_{I_\epsilon}\lambda)}(u)= \frac{1+\epsilon^{1/3}}{1-\epsilon^{1/3}}\ell(a) - (\ell(b_1) + \cdots +\ell(b_m)),\\
\label{eq:lenergy}
0 & <  E_{\lambda,I_\epsilon}(u)  = \frac{1+\epsilon^{1/3}}{1-\epsilon^{1/3}}\ell(a),\\
\label{eq:totenergy}
0 & <  E_{I_\epsilon}(u)=2\frac{1+\epsilon^{1/3}}{1-\epsilon^{1/3}}\ell(a) - (\ell(b_1) + \cdots +\ell(b_m)),
\end{align}
for any $J_{\OP{cyl}}$-holomorphic disc $u \in \mathcal{M}_{a;\mathbf{b};A}(V_{S,\epsilon};J_{\OP{cyl}})$ having boundary on $V_{S,\epsilon}$, a positive puncture at $a$, and negative punctures at $\mathbf{b}=b_1 \cdots  b_m$.

\subsubsection{A cylindrical almost complex structure that is integrable near the core disc}
\label{sec:integrable}

There is a unique cylindrical almost complex structure $J_0$ on $\RR \times J^1\RR^n$ having the property that the canonical projection
\[\pi \co \RR \times J^1\RR^n \to T^*\RR^n = \CC^n \]
is $(J_0,i)$-holomorphic. Furthermore, it can be checked that the identification
\begin{gather*}
(\RR \times J^1\RR^n,J_0) \to (\CC^n \times \CC,i), \\
(t,((\mathbf{q},\mathbf{p}),z)) \mapsto (\mathbf{q}+i\mathbf{p},t-\|\mathbf{p}\|^2/2+iz),
\end{gather*}
is a biholomorphism. In particular, the function
\begin{gather*}
\pi_\CC \co \RR \times (J^1\RR^n,J_0) \to (\CC,i), \\
(t,((\mathbf{q},\mathbf{p}),z)) \mapsto (t-\|\mathbf{p}\|^2/2+iz),
\end{gather*}
is holomorphic.

We use the exact symplectomorphism
\[(\Id_\RR,\phi) \colon \RR \times U_{\epsilon_S} \to \RR \times \phi(U_{\epsilon_S}) \]
used in the construction of $V_{S,\epsilon}$ (see Section \ref{sec:oneparam}), where $U_{\epsilon_S} \subset J^1\RR^n$ is a neighbourhood, to push forward $J_0$ to an almost complex structure on $\RR \times \phi(U_{\epsilon_S})\subset \RR \times Y$.

\subsubsection{A cylindrical almost complex structure $J_S$ that admits a local anti-holomorphic involution fixing $\Sigma$}
\label{sec:involution}
We will construct the sought cylindrical almost complex structure $J_S$ by requiring it to coincide with the push-forward of the above integrable almost complex structure $J_0$ in the neighbourhood $\RR \times \phi(U_{\epsilon_S})$, while we will prescribe it to have the following behaviour in a neighbourhood of
\[\RR \times \Sigma \subset \RR \times (L \setminus \phi(U_{\epsilon_S})).\]
Recall the neighbourhood $\Sigma \subset L \setminus \phi(U_{\epsilon_S})$ of $\Sigma_0 \subset L$ chosen in Section \ref{sec:oneparam}, where $\Sigma_0$ can be identified with the normal sphere-bundle of $S \subset L$.

A standard result implies that there exists an integrable almost complex structure $J_\Sigma$ on $T^*\Sigma$ for which the zero-section $0_\Sigma \subset T^*\Sigma$ is real-analytic. After shrinking the neighbourhood of the zero-section further, we may moreover assume that there exists an anti-holomorphic involution $\iota$ defined on this neighbourhood that fixes the zero-section pointwise.

It follows by Proposition \cite[Proposition 2.15]{SteinWeinstein} that, on some sufficiently small neighbourhood $D^*\Sigma \supset 0_\Sigma$, the squared distance $\rho \co D^*\Sigma \to \RR_{\ge 0}$ to the zero-section induced by some choice of Hermitian metric on $T^*\Sigma$ is a smooth plurisubharmonic function. In other words, writing
\[\theta_\rho:=-(d\rho)\circ J_\Sigma,\]
it follows that $d\theta_\rho$ is an exact symplectic form defined on $D^*\Sigma$ for which $J_\Sigma$ is a compatible almost complex structure. The restriction of $\theta_\rho$ to $0_\Sigma$ vanishes, which implies that the zero-section is an exact Lagrangian submanifold. Clearly the primitive of the pull-back of $\theta_\rho$ to $0_\Sigma$ can moreover be taken to \emph{vanish}.

The neighbourhood theorem \cite[Theorem 6.2.2]{IntroContact} for isotropic submanifolds can now be used to identify a neighbourhood of $\Sigma \subset L \subset Y$ with the contact manifold
\[(D^*\Sigma \times [-2E,2E], dz+\theta_\rho)\]
by a contact-form preserving diffeomorphism that moreover identifies
\[0_\Sigma \times \{0\} \subset D^*\Sigma \times [-2E,2E]\]
with $\Sigma \subset L$. Here have used the assumption in Section \ref{sec:Reebassumptions} regarding the behaviour of the Reeb flow of $(Y,\lambda)$ restricted to $\Sigma \subset Y$ to conclude that the neighbourhood may be taken on this very form, where $E>0$ is the number fixed above.

Again, there is a unique cylindrical almost complex structure $J_1$ on the symplectisation $\RR \times (D^*\Sigma \times [-2E,2E])$ defined by the requirement that the canonical projection
\[ (\RR \times (D^*\Sigma \times [-2E,2E]),J_1) \to (D^*\Sigma,J_\Sigma)\]
is holomorphic. We will require the cylindrical almost complex structure $J_S$ to coincide with the push-forward of $J_1$ under the above identification. It thus follows that there exists a neighbourhood $\RR \times O_\Sigma \subset \RR \times Y$ of $\Sigma \times \RR$ for which
\begin{itemize}
\item $\phi^{[-2E,2E]}_R(\Sigma) \subset O_\Sigma$;
\item $O_\Sigma \cap L = \Sigma$; and such that
\item there exists an anti-holomorphic involution of $(\RR \times O_\Sigma,J_S)$ that fixes $\RR \times \Sigma$ pointwise.
\end{itemize}

To see the last property, we argue as follows. It is readily checked that there is a holomorphic open embedding
\begin{gather*}
(\RR \times (D^*\Sigma \times [-2E,2E]),J_1) \to (D^*\Sigma\oplus \CC, J_\Sigma \oplus i),\\
(t,((\mathbf{q},\mathbf{p}),z)) \mapsto ((\mathbf{q},\mathbf{p}),t-\rho(\mathbf{q},\mathbf{p})+iz).
\end{gather*}
Using this embedding, the anti-holomorphic involution $\iota$ of $(D^*\Sigma,J_\Sigma)$ can be seen to lift to the required anti-holomorphic involution of $(\RR \times (D^*\Sigma \times [-2E,2E]),J_1)$.

\subsection{Properties of pseudo-holomorphic discs with\\ boundary on $V_S$}

Recall that, by the construction in Section \ref{sec:oneparam}, the exact Lagrangian cobordisms $V_{S,\epsilon}$, $0 < \epsilon \le \epsilon_S$, are all cylindrical outside of the subset $\RR \times \phi(U_{\epsilon_S})$.

\begin{lem}
\label{lem:mono}
There is a constant $E_0>0$ that only depends on the cylindrical almost complex structure $J_{\OP{cyl}}$ such that, for each $0 <\epsilon \le \epsilon_S$, any solution $u \in \mathcal{M}_{a;\mathbf{b};A}(V_{S,\epsilon};J_{\OP{cyl}})$ passing through $\RR \times \partial \phi(U_{\epsilon_S})$ satisfies
\[ E_{I_\epsilon}(u) \ge E_0>0.\]
\end{lem}
\begin{proof}
For each $t_0 \in \RR$ we consider the symplectic form
\[ \omega_{t_0} = e^{-(t_0+1/2)}d(e^t \lambda)=e^{t-(t_0+1/2)}dt \wedge \lambda+ e^{t-(t_0+1/2)}d\lambda\]
on $\RR \times Y$ and observe that $J_{\OP{cyl}}$ is compatible with $\omega_{t_0}$.

The subset $\partial \phi( U_{(3/4)\epsilon_S}) \subset Y$ is compact and, moreover, has a compact neighbourhood $O \subset Y$ satisfying
\[ V_{S,\epsilon} \cap (\RR \times O) =V_{S,\epsilon_S} \cap (\RR \times O)=\RR \times (L \cap O)\]
for each $0<\epsilon \le \epsilon_S$ (see Section \ref{sec:oneparam}).

The monotonicity property for the $\omega_{t_0}$-area of $J_{\OP{cyl}}$-holomorphic curves with and without boundary \cite[Propositions 4.3.1 and 4.7.2]{SomeProp} applies to pseudo-holomorphic curves $u$ as above passing through a point in $\{t_0\} \times \partial \phi( U_{(3/4)\epsilon_S})$.

More precisely, there is a constant $C>0$ only depending on $J_{\OP{cyl}}$ for which the following holds. Any non-trivial $J_{\OP{cyl}}$-holomorphic curve $u$ inside the compact set $\left[t_0-\frac{1}{2},t_0+\frac{1}{2}\right] \times O$ that passes through $\{t_0\} \times \partial \phi( U_{(3/4)\epsilon_S})$ and whose boundary is contained inside the compact set set
\[([t_0-1/2,t_0+1/2] \times (L \cap O)) \cup \partial([t_0-1/2,t_0+1/2] \times O)\]
has $\omega_{t_0}$-area satisfying the lower bound
\[ E_u:=\int_{u \cap \left(\left[t_0-\frac{1}{2},t_0+\frac{1}{2}\right] \times O\right)} \omega_{t_0} \ge C >0.\]

Now, let $u$ be a $J_{\OP{cyl}}$-holomorphic curve as in the assumption of the lemma. It follows that the above inequality holds for its $\omega_{t_0}$-area contained in the set $\left[t_0-\frac{1}{2},t_0+\frac{1}{2}\right] \times O$.

By the construction of the total energy, it follows that the inequality
\[2E_{I_\epsilon}(u) \ge \int_u \rho^-(t)dt \wedge \lambda+  \int_u \rho^+(t)dt \wedge \lambda+\frac{2}{1-\epsilon^{1/3}}\int_u d(\varphi_{I_\epsilon}(t) \lambda)\]
holds for any choice of non-negative functions $\rho^-(t)$ and $\rho^+(t)$ supported in the sets $\{t \le \log{(1-\epsilon^{1/3})}\}$ and $\{t \ge \log{(1+\epsilon^{1/3})}\}$, respectively, and satisfying $\int_{\RR} \rho^\pm(t)dt=1$. Together with the inequalities $\frac{2}{1-\epsilon^{1/3}}\varphi_{I_\epsilon}(t) \ge 2$ and
\[\frac{2}{1-\epsilon^{1/3}} \varphi_{I_\epsilon}'(t) \ge  \begin{cases}
2, & t \in I_\epsilon=[\log{(1-\epsilon^{1/3})} ,\log{(1+\epsilon^{1/3})} ],\\
0, & t \notin I_\epsilon=[\log{(1-\epsilon^{1/3})} ,\log{(1+\epsilon^{1/3})} ],
\end{cases}\]
together with appropriate choices of functions $\rho^\pm$, the inequality
\[ 2E_{I_\epsilon}(u) \ge E_u\]
can now readily be seen to follow. The statement thus holds for the choice of constant $E_0 := C/2>0$.
\end{proof}

Using (\ref{eq:totenergy}) and Remark \ref{rem:ambientsurgery} we conclude that $u \in \mathcal{M}_{c_S;\mathbf{b};A}(V_{S,\epsilon};J_{\OP{cyl}})$ satisfies
\[0 < E_{I_\epsilon}(u) = 2\frac{1+\epsilon^{1/3}}{1-\epsilon^{1/3}}\ell(c_S)-(\ell(b_1)+\cdots+\ell(b_m)) \le 2\frac{1+\epsilon^{1/3}}{1-\epsilon^{1/3}}\epsilon^{3/2}/\sqrt{2}\]
for any $u \in \mathcal{M}_{c_S;\mathbf{b};A}(V_{S,\epsilon};J_{\OP{cyl}})$. Together with Lemma \ref{lem:mono}, it immediately follows that such a disc must be disjoint from $\RR \times \partial \phi( U_{(3/4)\epsilon_S})$ given that $0 < \epsilon \le \epsilon_S$ is sufficiently small. In other words
\begin{cor}
\label{cor:smalldisc}
For sufficiently small $0<\epsilon \le \epsilon_S$, every $J_{\OP{cyl}}$-holomorphic disc $u \in \mathcal{M}_{c_S;\mathbf{b};A}(V_{S,\epsilon};J_{\OP{cyl}})$ is contained inside $\RR \times \phi( U_{(3/4)\epsilon_S})$. In particular, it must satisfy $\mathbf{b}=\emptyset$.
\end{cor}

Loosely speaking, the above corollary tells us that the small discs that have a positive puncture inside of the handle never exit to the cylindrical part. We proceed to show that certain small discs having a positive puncture in the cylindrical part never enter the non-trivial part of the handle (and hence are trivial strips). As it turns out, this will require a somewhat deeper analysis.

\begin{lem}
\label{lem:uniquedisc}
Let $J_S$ be a cylindrical almost complex structure as constructed in Section \ref{sec:integrable}. For $0<\epsilon \le \epsilon_S$ sufficiently small and $a \in \mathcal{Q}(L)$, the only solution in $\cup_{A \in H^1(V_{S,\epsilon})}\mathcal{M}_{a;a;A}(V_{S,\epsilon};J_S)$ is the trivial strip $\RR \times a$.
\end{lem}

\begin{proof}
We will show that any sequence $u_i \in \mathcal{M}_{a;a;A}(V_{S,\epsilon_i};J_S)$ of strips with $\lim_{i\to +\infty} \epsilon_i=0$ consists of the trivial strips $u_i=\RR \times a$ for all $i \gg 0$ sufficiently large.

First, the above formulas for the energy here yield
\begin{align*}
E_{d(\varphi_{I_{\epsilon_i}}\lambda)}(u_i) &= \frac{1+\epsilon_i^{1/3}}{1-\epsilon_i^{1/3}}\ell(a)-\ell(a)=\frac{2\epsilon_i^{1/3}}{1-\epsilon_i^{1/3}}\ell(a),\\
E_{\lambda,I_{\epsilon_i}}(u_i) &= \frac{1+\epsilon_i^{1/3}}{1-\epsilon_i^{1/3}}\ell(a),
\end{align*}
and, thus, in particular
\begin{align}
\label{eq:energylimit} \lim_{i \to +\infty} E_{d(\varphi_{I_{\epsilon_i}}\lambda)}(u_i)&= 0, \\
\label{eq:energylimit2} \lim_{i \to +\infty} E_{\lambda,I_{\epsilon_i}}(u_i)&= \ell(a).
\end{align}

We argue by contradiction. Assume that there is an infinite subsequence of $\{u_i\}$ satisfying the property that the boundary of the strip $u_i$ passes through
\[\{t_i\} \times \Sigma_0 \subset \{ t_i \} \times \Sigma \subset \{t_i\} \times L.\]
Recall that $\Sigma_0 \subset \Sigma \subset L$ is a smooth embedding of $S \times S^{n-k-1}$ that can be identified with the normal sphere-bundle of $S$, and that $\RR \times \Sigma \subset \RR \times Y$ is fixed by an anti-holomorphic involution defined in the neighbourhood $\RR \times O_\Sigma$.

The target-local version of Gromov's compactness theorem \cite[Theorem A]{TargetLocal} can be applied to the induced subsequence
\[u_i(\dot{D}^2) \cap ([t_i-1,t_i+1] \times O_\Sigma) \subset [t_i-1,t_i+1] \times O_\Sigma\]
of $J_S$-holomorphic curves to extract a convergent subsequence (we give\linebreak a justification of this at the end of this proof). After a translation of\linebreak the $t$-coordinate, the limit $\widetilde{u}_\infty$ may be considered to be a non-trivial $J_S$-holomorphic curve in $[-2,2] \times O_\Sigma$ having boundary on
\[([-2,2] \times \Sigma) \cup \partial([-2,2] \times O_\Sigma)\]
passing through $\RR \times \Sigma_0 \subset \RR \times \Sigma$.

By Formula (\ref{eq:energylimit}), the limit $\widetilde{u}_\infty$ must be contained inside a trivial strip $\RR \times c$, where $c$ is some (possibly disconnected) integral curve of the Reeb vectorfield. By the construction of the neighbourhood $O_\Sigma \subset Y$ in Section~\ref{sec:involution} we have $\phi^{[-2E,2E]}_R(\Sigma) \subset O_\Sigma$, which now can be seen to imply the estimate
\[\int_{\widetilde{u}_\infty \cap \{ t \in I \}} dt \wedge \lambda \ge 2E > 2\ell(a)\]
for any sub-interval $I \subset [-2,2]$ being of length one. Observe that the latter inequality holds by definition, since we only consider Reeb chords of action less than $E$. This shows that $\{u_i\}$ has an infinite subsequence for which
\[\int_{u_i \cap \{ t \in I_i\}} dt \wedge \lambda > 2\ell(a)\]
is satisfied as well, where $I_i \subset [t_i-2,t_i+2]$ again denotes any interval of length one. However, the latter inequality clearly contradicts the limit (\ref{eq:energylimit2}) of the $\lambda$-energy of the solutions $u_i$.

This contradiction shows that the strips $u_i$ have boundary disjoint from $\RR \times \Sigma_0$ for each $i\gg 0$ sufficiently large. Since each such solution $u_i$ thus has boundary contained inside the cylindrical part
\[V_{S,\epsilon} \setminus (\RR \times \phi(U_{(3/4)\epsilon_S})) \subset \RR \times L\]
of the cobordism (see Section \ref{sec:oneparam}) as follows from topological considerations, such a solution must satisfy $E_{d\lambda}(u_i)=0$. In conclusion, we must in fact have $u_i=\RR \times a$ for all $i \gg 0$ sufficiently large.

We end by arguing that the target-local version of Gromov's compactness theorem indeed can be applied in this situation. First, by the assumptions on $J_S$ in Section \ref{sec:involution}, there exists an anti-holomorphic involution of $(\RR \times O_\Sigma,J_S)$ fixing $(\RR \times O_\Sigma) \cap V_{S,\epsilon}$ pointwise. This involution can be used to perform a Schwarz-reflection of each curve $u_i(\dot{D}^2) \cap (\RR \times O_\Sigma)$, thus producing a curve $C_i \subset [t_i-2,t_i+2] \times O_\Sigma$ which is $J_S$-holomorphic and whose boundary is contained in $\partial([t_i-2,t_i+2] \times O_\Sigma)$. The sequence $C_i$ can moreover be seen to satisfy the following properties:
\begin{itemize}
\item After translating the $t$-coordinate, each $C_i$ may be identified with a $J_S$-holomorphic curve inside $[-2,2] \times O_\Sigma$ having boundary on $\partial([-2,2] \times O_\Sigma)$;
\item Each $C_i$ is of genus zero; and
\item Since there is a uniform bound on the total energy $E_{I_{\epsilon_i}}(u_i)$ of a solution $u_i$, there hence is a uniform bound on the $d(e^t\lambda)$-area of a curve $C_i$.
\end{itemize}\vspace{-.5em}
\end{proof}

\subsection{Pseudo-holomorphic discs in the non-cylindrical part\\ of the handle}
\label{sec:cobordismdiscs}

We proceed to investigate more closely the behaviour of pseudo-holomorphic discs in $\RR\times Y$ having boundary on $V_{S,\epsilon}$ and a positive puncture asymptotic to the new Reeb chord $c_S$. Here the results will rely on the particular choice of almost complex structure $J_S$ made above. In particular, it is important that $J_S$ was constructed to be integrable inside the non-cylindrical part of $V_{S,\epsilon}$ (see Section \ref{sec:integrable}).

\begin{lem}
\label{lem:disccontained}
For sufficiently small $0<\epsilon \le \epsilon_S$, every $u \in \mathcal{M}_{c_S;\emptyset;A}(V_{S,\epsilon};J_S)$ is contained in the set
\[\RR \times \phi\left( U_\epsilon \cap \left\{\begin{array}{l}
q_1=\cdots=q_{k+1}=0,\\
p_1=\cdots=p_{k+1}=0,\\
t\geq 0
\end{array}
\right\}\right).\]
\end{lem}
\begin{proof}
By Corollary \ref{cor:smalldisc} we may assume that the image of $u$ is contained in $\Id_\RR \times \phi(U_{(3/4)\epsilon_S})$. We can thus use $(\Id_\RR,\phi)^{-1}$ to identify $u$ with a $J_0$-holomorphic disc in $\RR \times U_{\epsilon_S} \subset \RR \times J^1\RR^n$ having boundary on the exact Lagrangian cobordism $W_{\epsilon,k}$ as defined in Section \ref{sec:model}. We will again use $u$ to denote this $J_0$-holomorphic disc.

Consider the holomorphic projections
\[ \pi_i:=(q_i,p_i) \co \RR \times U_{\epsilon_S} \to \CC \]
for $i=1,\hdots,n$. Observe that, by the construction of $J_S$ in the neighbourhood $\RR \times U_{\epsilon_S}$, the composition $\pi_i\circ u$ is a holomorphic map from the closed unit disc $\dot{D}^2$ with one boundary point removed. The image
\[\pi_i(W_{\epsilon,k}) \subset \CC, \quad i=1,\hdots,k+1,\]
is shown in Figure \ref{fig:hyp}. Since $\pi_i \circ u$ maps the boundary to a compact subset of $\pi_i(W_{\epsilon,k})$, the open mapping theorem implies that
\begin{equation}\label{eq:openmap}
\pi_i \circ u(\dot{D}^2) \subset \pi_i(W_{\epsilon,k}), \quad i=1,\hdots,k+1.
\end{equation}

By contradiction, we assume that $\pi_i\circ u$ does not vanishing identically for some $i=1,\hdots, k+1$. Using the asymptotic properties of $u$, and the fact that the boundary-condition has a discontinuity near the puncture, it can be seen that $\pi_i\circ u$ fills some part of the corner either above or below $c_S$ shown in Figure \ref{fig:hyp}. This however contradicts the above inclusion (\ref{eq:openmap}). In other words, $u$ is contained in the set
\[\left\{\begin{array}{l}
q_1=\cdots=q_{k+1}=0,\\
p_1=\cdots=p_{k+1}=0
\end{array}\right\}\]
and, by examining the construction of $W_{\epsilon,k}$ in Section \ref{sec:model}, its boundary is thus contained in
\[W_{\epsilon,k} \cap \left\{\begin{array}{l}
q_1=\cdots=q_{k+1}=0,\\
p_1=\cdots=p_{k+1}=0
\end{array}\right\} \subset \{ t \ge 0\}.\]

Recall that there is a holomorphic projection
\begin{gather*} \pi_\CC \co (J^1\RR^n,J_0) \to (\CC,i), \\
(t,((\mathbf{p},\mathbf{q}),z)) \mapsto t-\|\mathbf{p}\|^2/2+iz,
\end{gather*}
and that the real-part $\mathfrak{Re}( \pi_\CC \circ u)$ thus is harmonic. Using the fact that
\begin{equation}
\label{eq:ineqre}
\mathfrak{Re}( \pi_\CC \circ u) \ge 0
\end{equation}
holds along the boundary, as well as in some neighbourhood of the boundary puncture, we conclude that (\ref{eq:ineqre}) holds on all of the domain $\dot{D}^2$. In particular, $t \circ u \ge 0$ holds everywhere, from which the statement of the lemma follows.

It remains to establish the inequality (\ref{eq:ineqre}) along the boundary $\partial\dot{D}^2$. We will do this via certain estimates in terms of the functions used in the construction of $W_{\epsilon,k}$, and we refer to Section \ref{sec:model} for their definitions.

The inequality in Formula (\ref{eq:boundp}) implies that
\[|p_i|=\left|\frac{\partial}{\partial x_i}F_\epsilon(\mathbf{x},x_{n+1})\right|\leq (3/2)F_\epsilon(\mathbf{x},x_{n+1})^{1-2/3}2\sqrt{(2/3)\epsilon}\]
holds along the boundary of $u$ for each $i=k+2, \hdots, n$, which can be translated into the inequality
\[|p_i|^2\leq 9\left(\left(x_{1}^2+\cdots+x_{k+1}^2\right)-\left(x_{k+2}^2+\cdots+x_{n}^2\right)+\varphi_\epsilon(\mathbf{x},x_{n+1})-1\right)(2/3)\epsilon\]
along the boundary. Since the boundary of $u$ is contained in the set
\[\left\{ \begin{array}{l}q_1=\cdots=q_{k+1}=0,\\
t\ge 0\end{array}\right\}=\left\{ \begin{array}{l}x_1=\cdots=x_{k+1}=0,\\
x_{n+1} \ge 1\end{array}\right\},\]
as was shown above, this estimate becomes
\[|p_i|^2\leq 6\epsilon\left(-\left(x_{k+2}^2+\cdots+x_{n}^2\right)+\sigma_\epsilon(x_{n+1})-1\right) \le6\epsilon(\sigma_\epsilon(x_{n+1})-1).\]
Observe that the right-hand side vanishes for $x_{n+1}=e^t$ when $t=0$. Using the inequality $0 \le \sigma_\epsilon'(x_{n+1}) \le (1+\epsilon)\epsilon^{-1/3}$ and setting $x_{n+1}=e^t$, for any $A>0$ we now obtain the inequality
\[(t/n-|p_i|^2/2)\circ u \ge 0\]
along $\partial\dot{D}^2 \cap u^{-1}\{t  \in [0,A]\}$ given that $0<\epsilon\le\epsilon_S$ is chosen sufficiently small. This immediately gives the sought inequality $\mathfrak{Re}( \pi_\CC \circ u) \ge 0$ along all of the boundary $\partial \dot{D}^2$.
\end{proof}

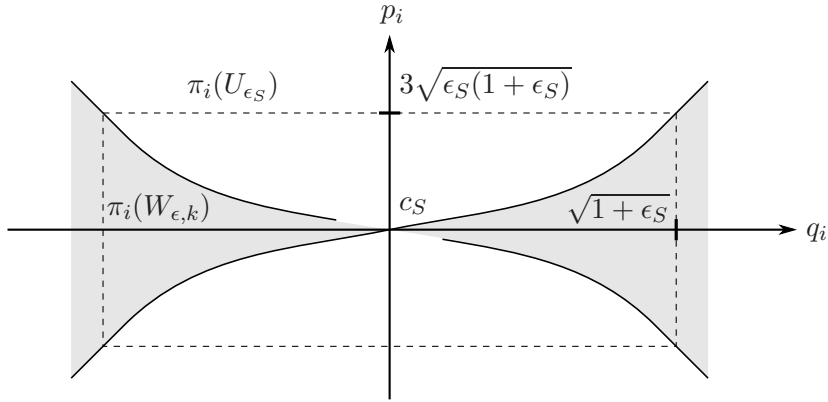
\begin{figure}[htp]
\centering
\labellist
\pinlabel $c_S$ at 153 73
\pinlabel $\pi_i(U_{\epsilon_S})$ at 85 118
\pinlabel $\pi_i(W_{\epsilon,k})$ at 57 72
\pinlabel $q_i$ at 305 63
\pinlabel $p_i$ at 145 145 
\pinlabel $\sqrt{1+\epsilon_S}$ at 230 72
\pinlabel $3\sqrt{\epsilon_S(1+\epsilon_S)}$ at 180 119
\endlabellist
\includegraphics{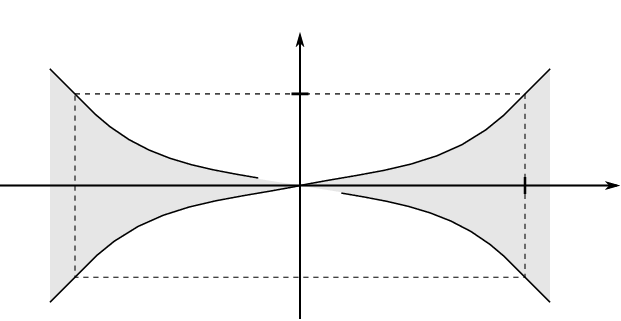}
\caption{The image of $W_{\epsilon,k}$ under the holomorphic projection $\pi_i$ for $i=1,\hdots,k+1$.}
\label{fig:hyp}
\end{figure}

\begin{figure}[htp]
\centering
\labellist
\pinlabel $c_S$ at 90 83
\pinlabel $\pi_i(U_{\epsilon_S})$ at 35 145
\pinlabel $\pi_i(W^+)$ at 61 90
\pinlabel $\pi_i(W^-)$ at 61 64
\pinlabel $q_i$ at 217 76
\pinlabel $p_i$ at 106 161
\pinlabel $3\epsilon$ at 118 113
\pinlabel $3\sqrt{\epsilon_S(1+\epsilon_S)}$ at 142 147
\pinlabel $\sqrt{\epsilon}$ at 168 85
\pinlabel $\sqrt{2\epsilon_S}$ at 212 89
\endlabellist
\includegraphics{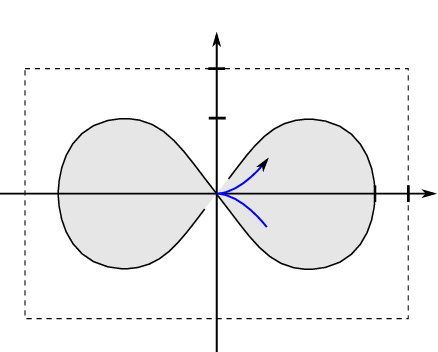}
\caption{The image of $W=W^+ \cup W^-$ under the holomorphic projection $\pi_i$ for $i=k+2,\hdots,n$.}
\label{fig:ell}
\end{figure}

Recall the construction of the core disc $C_{S,\epsilon} \subset V_{S,\epsilon}$ of the handle described in Section \ref{sec:constructV}.

\begin{lem}
\label{lem:handledisc}
For sufficiently small $0<\epsilon \le \epsilon_S$, each $J_J$-holomorphic disc $u \in \mathcal{M}_{c_S;\emptyset;A}(V_{S,\epsilon};J_S)$ whose boundary intersects $C_{S,\epsilon}$ is contained in the holomorphic strip $\RR_{\ge 0} \times c_S \subset \RR \times Y$. In particular, there is a unique such disc $u_0$.
\end{lem}
\begin{proof}
Let $u$ be such a disc. Using Corollary \ref{cor:smalldisc} we may identify $u$ with a $J_0$-holomorphic disc in $\RR \times U_{\epsilon_S}$ having boundary on $W_{\epsilon,k}$. Furthermore, Lemma \ref{lem:disccontained} implies that $u$ necessarily intersects $C_{S,\epsilon}$ in the set
\[C_{S,\epsilon} \cap \{ q_1=\cdots=q_{k+1}=0 \}=\{ \mathbf{q}=\mathbf{p}=z=t=0\}. \]

Recall that, by the construction of $J_S$, there are holomorphic projections
\[\pi_i:=(q_i,p_i)\co \RR \times U_{\epsilon_S} \to \CC.\]
We define
\begin{align*}
& W := W^+ \cup W^-,\\
& W^\pm := W_{\epsilon,k} \cap \left\{\begin{array}{l}
q_1=\cdots=q_{k+1}=0,\\
\pm z \ge 0,\\
t\geq 0
\end{array}\right\}.
\end{align*}
The image
\[\pi_i(W)\subset \CC, \quad i=k+2,\hdots,n,\]
is a filled figure-eight curves as shown in Figure \ref{fig:ell}. Since $\pi_i \circ u$ is holomorphic and maps the boundary into $\pi_i(W)$, the open mapping theorem implies that
\begin{equation}
\label{eq:openmap2}
\pi_i \circ u(\dot{D}^2 ) \subset \pi_i(W), \quad i=k+2,\hdots,n.
\end{equation}

We have seen that a boundary point $p$ mapped into $C_{S,\epsilon}$ by $u$ necessarily is mapped to the origin by $\pi_i \circ u$, and that the same is true for the holomorphic map
\[\pi_\CC \circ u=(t-\|\mathbf{p}\|^2/2+iz)\circ u \co \dot{D}^2 \to \CC.\]

Also, recall that $\mathfrak{Re}(\pi_\CC \circ u) \ge 0$ holds by (\ref{eq:ineqre}) established in the proof of Lemma \ref{lem:disccontained}. From this it follows that $\mathfrak{Im}(\pi_\CC \circ u)$ has the following behaviour at the boundary arc $\theta \mapsto e^{i\theta}p \in \partial D^2$ near $\theta=0$. There is some sufficiently small $\delta>0$ for which
\[z\circ u(e^{i\theta}p)=\mathfrak{Im} (\pi_\CC \circ u) (e^{i\theta}p) \in \begin{cases}  \RR_{>0}, & 0>\theta>-\delta, \\
\{0\}, & \theta=0,\\
\RR_{<0}, & \delta > \theta>0.
\end{cases}\]

By contradiction, we assume that $\pi_i \circ u$ does not vanish identically for $i=k+2,\hdots,n$. It readily follows that one of $\pm \pi_i \circ u$ must map the oriented boundary near $p$ as schematically depicted by the arrow in Figure \ref{fig:ell} (also, see the proof of Lemma \ref{lem:whitney}). However, in this case, we can use the open mapping theorem to get a contradiction with the inclusion (\ref{eq:openmap2}) established above. This contradiction thus shows that $u$ is contained inside the trivial strip $\RR_{\ge 0} \times c_S$ as claimed.
\end{proof}

Using the asymptotic properties, the above lemma shows that a solution $u_0 \in \mathcal{M}_{c_S;\emptyset;A}(V_{S,\epsilon};J_S)$ having boundary intersecting $C_{S,\epsilon}$ must be the unique embedded disc contained inside $\RR_{\ge 0} \times c_S$ and having boundary on $V_{S,\epsilon}$.

\begin{lem}
\label{lem:transversedisc}
For sufficiently small $0<\epsilon \le \epsilon_S$, the $J_S$-holomorphic disc $u_0 \in \mathcal{M}_{c_S;\emptyset;A}(V_{S,\epsilon};J_S)$ contained in $\RR_{\ge0} \times c_S$ is transversely cut out. Furthermore, the evaluation-map
\[\OP{ev} \co \mathcal{M}_{c_S;\emptyset;A}(V_{S,\epsilon};J_S) \times \partial \dot{D}^2 \to V_{S,\epsilon}\]
from the boundary is transverse to $C_{S,\epsilon}$ in some neighbourhood of $\{u_0\} \times \partial \dot{D}^2$.
\end{lem}
\begin{proof}
Again, we will identify $u_0$ with the corresponding $J_0$-holomorphic disc in $\RR \times U_{\epsilon_S} \subset \RR \times J^1\RR^n$. Since $u_0$ is an embedding, we will identify the domain of $u_0$ with its image.

Use
\[\left(\begin{array}{l}
\mathbf{q}=(\mathbf{q}_1,\mathbf{q}_2)=((q_1,\hdots,q_{k+1}),(q_{k+2},\hdots,q_n)),\\
\mathbf{p}=(\mathbf{p}_1,\mathbf{p}_2)=((p_1,\hdots,p_{k+1}),(p_{k+2},\hdots,p_n)),\\
z
\end{array}\right)\]
to denote the standard coordinates on $J^1(\RR^{k+1} \times \RR^{n-k-1})$. Similarly to the map considered in Section \ref{sec:integrable}, there is a biholomorphism
\begin{gather*}
\psi =(\psi_1,\psi_2,\psi_3) \co (\RR \times J^1\RR^n,J_0) \to (\CC^{k+1} \times \CC^{n-(k-1)} \times \CC,i), \\
(t,((\mathbf{q},\mathbf{p}),z)) \mapsto (\mathbf{q}_1+i\mathbf{p}_1,\mathbf{q}_2+i\mathbf{p}_2,t-\|\mathbf{p}\|^2/2+iz).
\end{gather*}

Since the almost complex structure is integrable in a neighbourhood of $u_0$, it follows that the linearisation of the Cauchy-Riemann operator at $u_0$ is the standard Cauchy-Riemann operator $\overline{\partial}$ acting on the trivial holomorphic vectorbundle $(u_0)^*(T\CC^{n+1})$.

We parametrise the punctured boundary of the domain of $u_0$ by
\begin{gather*}
(0,2\pi) \to \partial \dot{D}^2,\\
\theta \mapsto e^{i\theta},
\end{gather*}
for some choice of holomorphic coordinate that moreover satisfies the property that $e^{i\pi} \in \partial \dot{D}^2$ is the unique point evaluating to $u_0(e^{i\pi}) \in C_{S,\epsilon}$. Using the above biholomorphism, the linearised boundary condition
\[A(\theta):=T_{u_0(\theta)} W_{\epsilon,k} \subset T_{u_0(\theta)}(\CC^{k+1}\times \CC^{n-k-1} \times \CC) =\CC^{k+1} \oplus \CC^{n-k-1} \oplus \CC\]
can be seen to split into the direct sum
\[A(\theta)=A_1(\theta) \oplus A_2(\theta) \oplus A_3(\theta) \subset \CC^{k+1} \oplus \CC^{n-k-1} \oplus \CC\]
respecting the above decomposition of $\CC^{n+1}$. Furthermore, $A_i$ are families of Lagrangian subspaces of the form
\begin{align*}
 A_1(\theta)&= \RR\langle e^{i\varphi_1(\theta)} \mathbf{e}_j; \ 1\leq j\leq k+1 \rangle \subset \CC^{k+1},\\
 A_2(\theta)&= \RR\langle e^{i\varphi_2(\theta)} \mathbf{e}_j; \ k+2\leq j\leq n \rangle \subset \CC^{n-k-1},\\
 A_3(\theta)&= \RR e^{i\varphi_3(\theta)} \mathbf{e}_{n+1} \subset \CC,
\end{align*}
where $\{\mathbf{e}_i\}_{i=1,\hdots,{n+1}}$ denotes the standard basis of $\CC^{n+1}$. It can be checked that that $\varphi_1$ is non-increasing, while $\varphi_2$ is non-decreasing. These facts will be important in the argument below.

Since the above linearised boundary condition splits, the kernel of the linearised problem has an induced splitting $K=K_1 \oplus K_2 \oplus K_3$. Moreover, by elliptic regularity, this is a finite-dimensional space consisting of holomorphic functions
\begin{gather*}
\zeta \co D^2 \to \CC^{n+1},\\
\zeta(e^{i\theta}) \in A(\theta),
\end{gather*}
that moreover are continuous up to the boundary. 
Let
\[ \pi_i \co \CC^{k+1} \oplus \CC^{n-k-1} \oplus \CC \to \CC \]
be the orthogonal projection onto the $i$:th component. By the argument principle, together with the fact that $\varphi_1$ above is non-increasing, we immediately see that $\pi_i \circ \zeta \equiv 0$ vanishes identically for each $i=1, \hdots, k+1$. In other words, $K_1=0$. Furthermore, the solutions $\pi_{n+1} \circ \zeta \in K_3$ correspond to infinitesimal reparametrisations of the domain, and hence $\dim K_3=2$.

Investigating the Fredholm index of the linearised boundary-value problem we get
\[  \dim K_2 +2 = \dim K  \ge (n-k-1)+2,\]
where equality holds if and only if the cokernel vanishes, i.e.~if the solution $u_0$ is transversely cut out. To show transversality it therefore suffices to show that $\dim K_2=n-k-1$. We will simultaneously show that the evaluation map is transverse to $C_{S,\epsilon}$ in a neighbourhood of $\{u_0\} \times \partial \dot{D}^2 \subset  \mathcal{M}_{c_S;\emptyset;A}(V_{S,\epsilon};J_S) \times \partial \dot{D}^2$.

Recall that
\[(u_0,e^{i\pi}) \in \mathcal{M}_{c_S;\emptyset;A}(V_{S,\epsilon};J_S) \times \partial \dot{D}^2\]
is the unique point evaluating to $C_{S,\epsilon}$. Consider the linear map
\[D\OP{ev} =(D\OP{ev})_1+(D\OP{ev})_2 \co K \times \RR \to A(\pi) \subset \CC^{k+1} \times \CC^{n-k-1} \times \CC,\]
where
\begin{align*}
D\OP{ev}_1 (\zeta,\theta)&=\zeta(\pi),\\
D\OP{ev}_2(\zeta,\theta)&=\theta \partial_s u(e^{is})|_{s=\pi}.
\end{align*}
Note that the map $D\OP{ev}$ is indeed the differential of the evaluation-map $\OP{ev}$ at the point $(u_0,e^{i\pi}) \in  \mathcal{M}_{c_S;\emptyset;A}(V_{S,\epsilon};J_S) \times \partial \dot{D}^2$, given that the latter space is smooth near $u_0$.

Under the above biholomorphism, the tangent plane
\[T_{\OP{ev}(u_0,e^{i\pi})}C_{S,\epsilon} \subset T_{\OP{ev}(u_0,e^{i\pi})}V_{S,\epsilon}\]
is identified with
\[ \mathfrak{Re}{\CC^{k+1}} \oplus 0 \oplus 0 \subset \CC^{k+1} \oplus \CC^{n-k-1} \oplus \CC,\]
while the tangent plane $T_{\OP{ev}(u_0,e^{i\pi})}V_{S,\epsilon}$ is identified with
\[ A(\pi)= \mathfrak{Re}{\CC^{k+1}} \oplus \mathfrak{Re}(\CC^{n-k-1}) \oplus \mathfrak{Im}(\CC) \subset \CC^{k+1} \oplus \CC^{n-k-1} \oplus \CC.\]

Consider the linear subspace
\[V:=\mathrm{im}(D\OP{ev}) \subset A(\pi) = \mathfrak{Re}{\CC^{k+1}} \oplus \mathfrak{Re}(\CC^{n-k-1}) \oplus \mathfrak{Im}(\CC).\]
First, observe that $0 \oplus 0 \oplus \mathfrak{Im}(\CC) \subset V$ since the boundary of $u_0$ is embedded and tangent to this subspace at the boundary point $e^{i\pi} \in \partial \dot{D}^2$. Since $K_1=0$ we moreover conclude that
\[ V \subset 0\oplus\mathfrak{Re}(\CC^{n-k-1}) \oplus \mathfrak{Im}(\CC).\]
Both the property that $\dim K_2=n-k-1$ and the transversality of the evaluation map will thus follow if we manage to show that the linear map
\begin{gather*}
(D\OP{ev})_1|_{K_2} \co K_2 \to 0\oplus\mathfrak{Re}(\CC^{n-k-1}) \oplus 0,\\
\zeta \mapsto \zeta(\pi),
\end{gather*}
is injective. To see this, recall that we have the inequality $\dim K_2 =\dim K \ge n-k-1$ by the above considerations of the Fredholm index and, therefore, if $(D\OP{ev})_1|_{K_2}$ is injective it is necessarily also an isomorphism. The transversality of the evaluation map would thus also follow from this, since
\[ V = 0\oplus\mathfrak{Re}(\CC^{n-k-1}) \oplus \mathfrak{Im}(\CC) \subset \mathfrak{Re}{\CC^{k+1}} \oplus \mathfrak{Re}(\CC^{n-k-1}) \oplus \mathfrak{Im}(\CC)\]
is transverse to
\[ \mathfrak{Re}{\CC^{k+1}} \oplus 0 \oplus 0 \subset \mathfrak{Re}{\CC^{k+1}} \oplus \mathfrak{Re}(\CC^{n-k-1}) \oplus \mathfrak{Im}(\CC)\]
in this case.

To establish the injectivity of $(D\OP{ev})_1|_{K_2}$ we proceed as follows. Consider the holomorphic projections
\[ \pi_i := (x_i,y_i) \co \CC^{k+1} \times \CC^{n-k-1} \times \CC \to \CC,\quad i=k+2,\hdots,n.\]
Take any solution $\zeta \in K_2$ to the linearised problem. Because of the boundary condition, the holomorphic map $\pi_i \circ \zeta$ has boundary values inside the cone $\pi_i(A_2(0,2\pi))$ shown in Figure \ref{fig:linearised}. By the open-mapping theorem, it thus follows that
\begin{equation*}
\label{eq:opentrans}
\pi_i \circ \zeta(D^2 \setminus \partial D^2) \subset \OP{int}\pi_i(A_2(0,2\pi)).
\end{equation*}

Observe that $\pi_i \circ \zeta(e^{i\theta}) \in \pi_i (A_2(\theta))$, and that $\arg(\pi_i \circ \zeta(e^{i\theta}))=\varphi_2(\theta)$ thus is non-decreasing. Since
\[\pi_i \circ \zeta(e^{i\pi}) \in \pi_i(A_2(\pi))= \mathfrak{Re}{\CC},\]
similarly as in the proof of Lemma \ref{lem:handledisc}, the open-mapping theorem can again be used to show that $\pi_i \circ \zeta(e^{i\pi})=0$ if and only if $\pi_i \circ \zeta \equiv 0$ vanishes identically. In other words, the above map $(D\OP{ev})_1|_{K_2} $ into $\mathfrak{Re}\CC^{n-k-1}$ has no kernel.
\end{proof}

\begin{figure}[htp]
\centering
\labellist
\pinlabel $x_i$ at 208 47
\pinlabel $\pi_i(A_2(0,\pi])$ at 144 36
\pinlabel $\pi_i(A_2[\pi,2\pi))$ at 147 58
\pinlabel $y_i$ at 97 109
\endlabellist
\includegraphics{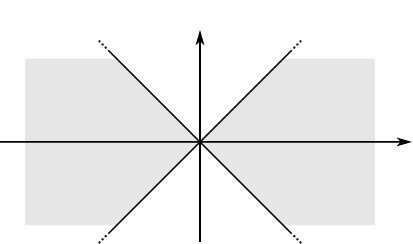}
\caption{The $i$:th component of the linearised boundary condition $\pi_i (A_2(0,2\pi))$ for $i=k+2,\hdots,n$.}
\label{fig:linearised}
\end{figure}
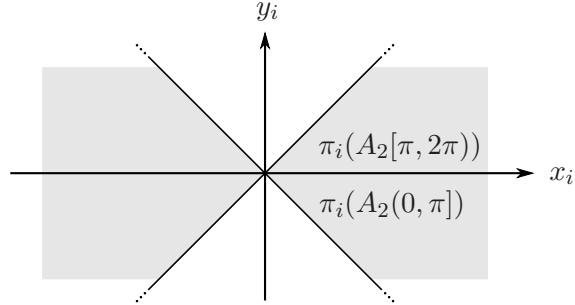

\begin{lem}
\label{lem:transversedisc2}
The moduli space $\mathcal{M}_{c_S;\emptyset;A}(\RR \times L_{S,\epsilon};J_S)/\RR$ satisfies the following properties when $0<\epsilon \le \epsilon_S$ is sufficiently small.
\begin{itemize}
\item $k<n-2$: There is a unique $A \in H_1(L_{S,\epsilon})$ for which this moduli space is non-empty. In this case, it has expected dimension at least one.
\item $k=n-2$: There is a unique $A \in H_1(L_{S,\epsilon})$ for which this moduli space is non-empty. In this case, it consists of precisely two transversely cut out solutions.
\item $k=n-1$: In this case the moduli space is empty.
\end{itemize}
\end{lem}
\begin{proof}
After a translation of the $t$-coordinate, we may assume that $u \in \mathcal{M}_{c_S;\emptyset;A}(\RR \times L_{S,\epsilon};J_S)$ also is contained in $\mathcal{M}_{c_S;\emptyset;A}(V_{S,\epsilon};J_S)$. Since Lemma~\ref{lem:disccontained} thus applies, it follows that $u$ has image contained in the subset 
\[\RR \times \phi\left( U_{\epsilon_S} \cap \left\{\begin{array}{l}
q_1=\cdots=q_{k+1}=0,\\
p_1=\cdots=p_{k+1}=0\\
\end{array}
\right\}\right) \subset \RR \times Y.\]
Again, we will identify $u$ with the corresponding $J_0$-holomorphic disc in $\RR \times U_{\epsilon_S} \subset \RR \times J^1\RR^n$ having boundary on $\RR \times L_{\epsilon,k}$. From this fact we obtain the result in the cases $k \neq n-2$, and it thus remains to consider the case $k=n-2$.

Since the projection
\[\pi_n := (q_n,p_n) \co \RR \times U_{\epsilon_S} \to \CC \]
is holomorphic, and since $\pi_n(\RR \times (L_{\epsilon,k} \cap U_{\epsilon_S}))$ is a figure-eight curve (see Figure \ref{fig:ell}), it follows that $\pi_n \circ u$ is a holomorphic disc having boundary on this figure-eight curve and precisely one corner at the double-point.

Up to parametrisation, there are exactly two such holomorphic polygons in $\CC$, which moreover are embedded, as follows by the assumption that there is a unique positive puncture. As in the proof of Lemma \ref{lem:transversedisc} it can be checked that the two corresponding holomorphic polygons in
\[\left\{\begin{array}{l}
q_1=\cdots=q_{k+1}=0,\\
p_1=\cdots=p_{k+1}=0\\
\end{array}
\right\} \subset T^*\RR^n = \CC^n\]
having boundary on $\Pi_{\OP{Lag}}(L_{\epsilon,k}) \subset T^*\RR^n=\CC^n$ are transversely cut out.

Using \cite[Theorem 1.2]{Lifting} it follows the above holomorphic polygons inside $(T^*\RR^n =\CC^n,\Pi_{\OP{Lag}}(L_{\epsilon,k}))$ can be lifted to the symplectisation, where this lift moreover induces a bijective correspondence between such polygons and $J_0$-holomorphic discs in $\RR \times U_{\epsilon_S} \subset \RR \times J^1\RR^n$ having boundary on $\RR \times L_{\epsilon,k}$ (up to translation), and where the lifted discs are transversely cut out as well. Finally, under the map $(\Id_\RR,\phi)$, these two $J_0$-holomorphic discs are in bijective correspondence with the solutions in $\mathcal{M}_{c_S;\emptyset;A}(\RR \times L_{S,\epsilon};J_S)/\RR$.
\end{proof}

\subsection{Proof of Theorem \ref{thm:surjection}}
\label{proof:surjection}
Let $J_S$ be the cylindrical almost complex structure on $\RR \times Y$ constructed in Section \ref{sec:involution}. Part (2) of Proposition \ref{prop:trans} implies that we may assume $J_S$ to be regular for the moduli spaces $\mathcal{M}_{a;\mathbf{b};A}(\RR \times L;J_S)$ and $\mathcal{M}_{a;\mathbf{b};A}(\RR \times L_{S,\epsilon};J_S)$ whenever $a \neq c_S$. Moreover, the moduli spaces being of the form $\mathcal{M}_{c_S;\emptyset;A}(\RR \times L_{S,\epsilon};J_S)$ and of expected dimension one are transversely cut out by Lemma \ref{lem:transversedisc2}. In conclusion, $J_S$ may be assumed to be regular simultaneously for the moduli spaces in the definition of the Chekanov-Eliashberg algebras of $L$ and $L_{S,\epsilon}$.

We begin with the proof of surjectivity. Consider an element $a \in \mathcal{Q}(L_{S,\epsilon})$. The inequality (\ref{eq:denergy}) implies that, for $\epsilon>0$ small enough, a moduli space of the form $\mathcal{M}_{a;\mathbf{b};A}(V_{S,\epsilon};J_S)$ is empty whenever $\mathbf{b}$ is a word containing a generator $b \in \mathcal{Q}(L)$ satisfying $\ell(b)>\ell(a)$.

For sufficiently small $\epsilon>0$ and $a \neq c_S$, Lemma \ref{lem:uniquedisc} implies that the trivial strip is the only solution in $\mathcal{M}_{a;a;A}(V_{S,\epsilon};J_S)$. By an explicit calculation of the cokernel of the linearised $\overline{\partial}_{J_S}$-operator along a trivial strip $\RR \times a$, the fact that $J_S$ is cylindrical implies that such a solution is transversely cut out.

In order to achieve transversality for the moduli spaces $\mathcal{M}_{a;\mathbf{b};A}(V_{S,\epsilon};J_S)$ in the case when $\mathbf{b} \neq a$, we might have to perturb $J_S$ by a compactly supported (and thus non-cylindrical) perturbation. However, after a sufficiently small such perturbation, the above counts $|\mathcal{M}_{a;a;A}(V_{S,\epsilon};J_S)|=1$ still remain true. For such a choice of $J_S$, we have thus concluded that
\begin{align*}
\Phi_{V_{S,\epsilon}}(a)&= a+B(a), \quad a \neq c_S,\\
\Phi_{V_{S,\epsilon}}(c_S) &\in \ZZ_2,
\end{align*}
where $B(a)$ is spanned by words of generators having action strictly less than $\ell(a)$. In particular the surjectivity of $\Phi_{V_{S,\epsilon}}$ follows.

It remains to investigate $\ker \Phi_{V_{S,\epsilon}}$. Since the induced map
\[ \Phi_{V_{S,\epsilon}}\co (\mathcal{A}(L_{S,\epsilon}),\partial_S)/\langle c_S-\Phi_{V_{S,\epsilon}}(c_S) \rangle \to (\mathcal{A}(L),\partial_S)\]
can be seen to be an isomorphism of DGAs, we get that
\[\ker \Phi_{V_{S,\epsilon}}=\langle c_S-\Phi_{V_{S,\epsilon}}(c_S)  \rangle,\]
where the latter denotes the two-sided ideal generated by $c_S-\Phi_{V_{S,\epsilon}}(c_S)$. Take a disc $u \in \mathcal{M}_{c_S;\emptyset;A}(V_{S,\varepsilon};J_S)$ contributing to $\Phi_{V_{S,\epsilon}}(c_S)$. Lemma \ref{lem:disccontained} implies that the image of $u$ is contained in the set
\[\RR \times \phi\left( U_\epsilon \cap \left\{\begin{array}{l}
q_1=\cdots=q_{k+1}=0,\\
p_1=\cdots=p_{k+1}=0,\\
t\geq 0
\end{array}
\right\}\right).\]

{\bf Case $k<n-1$:} By the above property of $u$, we compute that the moduli space $\mathcal{M}_{c_S;\emptyset;A}(V_{S,\varepsilon};J_S)$ has expected dimension at least one (whenever it is non-empty). In particular, we have
\[ \Phi_{V_{S,\epsilon}}(c_S) =0.\]

{\bf Case $k=n-1$:} In this case there is a unique such disc $u$ which moreover is contained inside $\RR_{\ge0} \times c_S$. Furthermore, Lemma \ref{lem:transversedisc} shows that $\mathcal{M}_{c_S;\emptyset;A}(V_{S,\varepsilon};J_S)=\{u\}$ is transversely cut out. In other words, we have shown that
\[\Phi_{V_{S,\epsilon}}(c_S)=1.\]

\section{The Chekanov-Eliashberg algebra twisted by a\\ submanifold}
\label{sec:twist-homol-algebr}
Let $L \subset (Y,\lambda)$ be a chord-generic Legendrian submanifold of a contact manifold of dimension $2n+1$. In the following we assume that $S \subset L$ is an embedded submanifold of dimension $k$ that admits a non-vanishing normal vectorfield $\mathbf{v} \subset NS \subset TL$. We here require that $k<n-1$, so that the codimension of $S \subset L$ is at least two. Furthermore, we assume that there are no Reeb chords on $L$ having endpoints on $S$. This property can be achieved after a generic smooth perturbation of $S \subset L$.
\subsection{Definitions}
In this section we construct the \emph{Chekanov-Eliashberg algebra of $L$ twisted by $S$}, which is a differential graded algebra that will be denoted by $(\mathcal{A}(L;\!S),\!\partial_{S,\mathbf{v}})$. 
\subsubsection{The graded algebra}
Consider the unital non-commutative algebra
\[ \mathcal{A}(L;S):= \ZZ_2 \langle \mathcal{Q}(L) \cup \{s\} \rangle,\]
freely generated over $\ZZ_2$ by the Reeb chords on $L$ together with a formal generator $s$. We give the Reeb-chord generators the usual grading induced by the Conley-Zehnder index (see Section \ref{sec:grading}), while we grade the formal variable by
\[ |s|=n-k-1.\]

\subsubsection{The boundary map}
For a number $\delta>0$ and a Riemannian metric $g$ on $L$, consider the moduli spaces $\mathcal{M}^{g,\delta}_{a;\mathbf{b},\mathbf{w};A}(L;S,\mathbf{v};J_{\OP{cyl}})$ constructed in Section \ref{sec:moduli}, for a regular cylindrical almost complex structure $J_{\OP{cyl}}$. We define the boundary operator by
\begin{eqnarray*}
\lefteqn{\partial_{S,\mathbf{v}}(a):= }\\
& & \sum_{|a|-|\mathbf{b}|-w|s|+\mu(A)=1}  |\mathcal{M}^{g,\delta}_{a;\mathbf{b},\mathbf{w};A}(L;S,\mathbf{v};J_{\OP{cyl}})/\RR|s^{w_1}b_1s^{w_2}\cdots s^{w_m}b_ms^{w_{m+1}},\\
\lefteqn{\partial_{S,\mathbf{v}}(s) := 0,}
\end{eqnarray*}
where $a \in \mathcal{Q}(L)$, $\mathbf{b}=b_1 \cdots  b_m$ is a (possibly empty) word of Reeb chords on $L$, $A \in H_1(L)$, $\mathbf{w} \in (\ZZ_{\ge 0})^{m+1}$, and $w:=w_1 +\cdots+w_{m+1}$. We extend $\partial_{S,\mathbf{v}}$ to $\mathcal{A}(L;S)$ using the Leibniz rule. It immediately follows that $\partial_{S,\mathbf{v}}$ has degree $-1$.

\begin{lem}
\label{lem:bdy}
For a generic $J_{\OP{cyl}}$ the boundary map $\partial_{S,\mathbf{v}}$ is a well-defined differential. Moreover, there is a canonical identification
\[  (\mathcal{A}(L;S),\partial_{S,\mathbf{v}})/\langle s \rangle = (\mathcal{A}(L),\partial)\]
of DGAs.
\end{lem}
\begin{proof}
For a generic $J_{\OP{cyl}}$ the union of moduli spaces in the definition of $\partial_{S,\mathbf{v}}(a)$ is a compact zero-dimensional manifold, as follows by Proposition~\ref{prop:dim} together with Theorem \ref{thm:transcomp}. It follows that the above count makes sense.

By Theorem \ref{thm:transcomp}, the coefficient in front of the word
\[ s^{w_1}b_1 s^{w_2} \cdots s^{w_m}b_m s^{w_{m+1}} \]
in the expression $\partial_{S,\mathbf{v}}^2(a)$ is given by the count of boundary points of the one-dimensional compact moduli space $\mathcal{M}^{g,\delta}_{a;\mathbf{b},\mathbf{w};A}( L;S,\mathbf{v};J_{\OP{cyl}})$, from which $\partial_{S,\mathbf{v}}^2(a)=0$ now follows.

The last statement is immediate from the definition of the Chekanov-Eliashberg algebra of $L$.
\end{proof}

\subsection{Maps induced by cobordisms}
\label{sec:smallinvariance}
We suppose that $V$ is an exact Lagrangian cobordism from $L_-$ to $L_+$ and that $M \subset V$ is a $(k+1)$-dimensional  submanifold coinciding with
\[((-\infty,-A) \times S_-) \cup ((B,+\infty) \times S_+)\]
outside of a compact set. Moreover, we assume that $\mathbf{v}$ is a non-vanishing normal vectorfield to $M$ in $TV$ that moreover coincides with translation invariant vectorfields $\mathbf{v}_- \subset TL_-$ and $\mathbf{v}_+ \subset TL_+$ in the sets $\{ t \le -A\}$ and $\{ t \ge B\}$, respectively.

We use $s_\pm$ to denote the formal generators of $\mathcal{A}(L_\pm;S_\pm)$. For a choice of regular compatible almost complex structure $J$ which, moreover, is cylindrical in the subsets $(-\infty,A] \times Y$ and $[B,+\infty) \times Y$, we define
\begin{eqnarray*}
\lefteqn{\Phi_{V;M,\mathbf{v}}(a):= }\\
&& \sum_{|a|-|\mathbf{b}|-w|s_-|+\mu(A)=0}  |\mathcal{M}^{g,\delta}_{a;\mathbf{b},\mathbf{w};A}(V;M,\mathbf{v};J)|s_-^{w_1}b_1s_-^{w_2} \cdots s_-^{w_m}b_ms_-^{w_{m+1}}, \\
\lefteqn{\Phi_{V;M,\mathbf{v}}(s_+):=s_-,}
\end{eqnarray*}
where $a \in \mathcal{Q}(L_+)$, $\mathbf{b}=b_1 \cdots  b_m$ is a (possibly empty) word of Reeb chords on $L_-$, $A \in H_1(V)$, $\mathbf{w} \in (\ZZ_{\ge 0})^{m+1}$, and $w:=w_1 +\cdots+w_{m+1}$. We extend $\Phi_{V;M,\mathbf{v}}$ to a unital algebra map
\[ \Phi_{V;M,\mathbf{v}} \co \mathcal{A}(L_+;S_+) \to \mathcal{A}(L_-;S_-).\]

\begin{prop}
For a generic compatible almost complex structure $J$ the above map
\[ \Phi_{V;M,\mathbf{v}} \co (\mathcal{A}(L_+;S_+),\partial_{S_+,\mathbf{v}_+}) \to (\mathcal{A}(L_-;S_-),\partial_{S_-,\mathbf{v}_-}) \]
is well defined unital DGA morphism. Furthermore, it descends to the DGA morphism
\[ \Phi_V \co (\mathcal{A}(L_+),\partial_+) \to (\mathcal{A}(L_-),\partial_-) \]
under the natural projections to the respective Chekanov-Eliashberg algebras.
\end{prop}
\begin{proof}
As in the proof of Lemma \ref{lem:bdy}, the fact that this map is well-defined again follows from Proposition \ref{prop:dim} together with Theorem \ref{thm:transcomp}.

Furthermore, by Theorem \ref{thm:transcomp} the coefficient in front of the word
\[ s^{w_1}b_1 s^{w_2} \cdots s^{w_m}b_m s^{w_{m+1}} \]
in the expression
\[(\Phi_{V;M,\mathbf{v}}\circ \partial_{S_+,\mathbf{v}_+}-\partial_{S_-,\mathbf{v}_-}\circ \Phi_{V;M,\mathbf{v}})(a)\]
is given by the count of boundary points of the one-dimensional compact moduli space $\mathcal{M}^{g,\delta}_{a;\mathbf{b},\mathbf{w};A}(V;M,\mathbf{v};J)$.
The chain-map property now follows.

Finally, the statement concerning the induced maps on the (non-twisted) Chekanov-Eliashberg algebras follows immediately from the definition of $\Phi_V$ together with Lemma \ref{lem:bdy}
\end{proof}

\begin{prop}
\label{prop:twistiso}
The map
\[ \Phi_{\RR \times L;M,\mathbf{v}} \co (\mathcal{A}(L;S_+),\partial_{S_+,\mathbf{v}_+}) \to  (\mathcal{A}(L;S_-),\partial_{S_-,\mathbf{v}_-}) \]
is a tame isomorphism of DGAs for any regular almost complex structure which is a sufficiently small compactly supported perturbation of a cylindrical almost complex structure.
\end{prop}
\begin{proof}
Recall that trivial strips $\RR \times a$ are rigid pseudo-holomorphic discs in the case when the almost complex structure is cylindrical. By the formula for the $d\lambda$-area in Proposition \ref{prop:energy}, it thus follows that
\[  \Phi_{\RR \times L;M,\mathbf{v}}(a)=a+B(a),\]
where $B(a)$ is a linear combination of words of generators having action strictly less than $\ell(a)$ (here we have prescribed $\ell(s_\pm):=0$). Finally, we have
\[\Phi_{\RR \times L;M,\mathbf{v}}(s_+)=s_-\]
by definition.
\end{proof}

Observe that the required regular almost complex structure exists by Theorem \ref{thm:transcomp}. Alternatively, one can perturb the cobordism $M \subset \RR \times L$ together with the Riemannian metric $g$ to achieve transversality for a \emph{cylindrical} almost complex structure, using a standard finite-dimensional transversality argument applied to the evaluation maps from the involved moduli spaces.

\begin{cor}
\label{cor:nullcob}
Let $S \subset L$ be a submanifold with non-vanishing normal vectorfield $\mathbf{v}$. If $S$ admits an embedded null-cobordism $(M,\partial M) \subset ([0,1] \times L,\{1\} \times S)$ along which $\mathbf{v}$ moreover extends as a non-vanishing normal vectorfield, then there is a tame isomorphism
\[ \Phi_{V;M,\mathbf{v}} \co (\mathcal{A}(L;S),\partial_{S,\mathbf{v}}) \to (\mathcal{A}(L;\emptyset),\partial_{\emptyset,\emptyset})\]
of DGAs.
\end{cor}

\subsection{Proof of Theorem \ref{thm:isom}}
\label{sec:isomorphism}
We assume that $L_{S,\epsilon} \subset Y$ is obtained from $L \subset Y$ by a Legendrian ambient surgery on the framed sphere $S \subset L$ and let $V_{S,\epsilon}$ denote the induced exact Lagrangian cobordism from $L$ to $L_{S,\epsilon}$. Recall the definition of the core disc $C_{S,\epsilon} \subset V_{S,\epsilon}$ in Section \ref{sec:constructV}. In particular, $C_{S,\epsilon}$ coincides with
\[ (-\infty,-1) \times S \subset (-\infty,-1) \times L \subset V_{S,\epsilon} \]
outside of a compact set.  We also fix a non-vanishing normal vectorfield $\mathbf{v}$ to $C_{S,\epsilon} \subset V_{S,\epsilon}$ which, outside of a compact set, has the property that it is
\begin{itemize}
\item invariant under translations of the $t$-coordinate, and
\item coincides with a non-vanishing normal vectorfield of $S \subset L_{S,\epsilon}$.
\end{itemize}
In particular, it follows that $\mathbf{v}$ is homotopic to a constant vectorfield with respect to the frame of the normal bundle of $S$ used for the surgery.

By construction, we have an inclusion
\[(\mathcal{A} (L_{S,\epsilon}),\partial_{L_{S,\epsilon}}) \subset (\mathcal{A} (L_{S,\epsilon};\emptyset,\emptyset),\partial_{\emptyset,\emptyset})\]
of DGAs. We show that the DGA morphism
\[ \Psi := \Phi_{V_{S,\epsilon};C_{S,\epsilon},\mathbf{v}} \co (\mathcal{A} (L_{S,\epsilon}),\partial_{L_{S,\epsilon}}) \to (\mathcal{A}(L;S),\partial_{S,\mathbf{v}}),\]
obtained as the restriction of the above DGA morphism satisfies the required properties, given that we choose the almost complex structure with some care. To that end, we will use the cylindrical almost complex structure $J_S$ constructed in Section \ref{sec:involution}; this is also the almost complex structure used in the proof of Theorem \ref{thm:surjection}.

First, for $0<\epsilon \le \epsilon_S$ sufficiently small, the proof of Theorem \ref{thm:surjection} immediately generalises to give the following. For each $a \neq c_S$ we have
\[ \Psi(a)=a+B(a),\]
where $B(a)$ is a linear-combination of words consisting of generators of action strictly less than $\ell(a)$ (here we have defined $\ell(s):=0$).

Second, the disc count in Lemma \ref{lem:handledisc} together with the transversality result in \ref{lem:transversedisc} shows that
\[ \Psi(c_S)=s,\]
given that $0<\epsilon \le \epsilon_S$ is sufficiently small.

In particular, it follows that $\Psi$ is the required tame isomorphism of DGAs for (a suitable perturbation of) the above almost complex structure~$J_S$.

\section{The Chekanov-Eliashberg algebra twisted by a\\ hypersurface}
\label{sec:critical}
We now consider the case when $S \subset L$ is a closed co-oriented hypersurface, and where $L \subset (Y,\lambda)$ is an oriented Legendrian submanifold. The requirement that the boundary of a pseudo-holomorphic disc $u \co (\dot{D}^2,\partial\dot{D}^2) \to (\RR \times Y,\RR \times L)$ intersects $\RR \times S$ transversely is in this case an open condition on the space of maps. Hence, introducing such a boundary-point constraint does not cut down the dimension of a moduli space and, consequently, the algebraic formalism in Section \ref{sec:twist-homol-algebr} cannot be expected to work. Here we present an alternative construction that, unfortunately, does not recover the full statement analogous to Theorem \ref{thm:isom} in general.

There is a version of the Chekanov-Eliashberg algebra of $L$ whose coefficients are taken in the group-ring $\ZZ_2[H_1(L)]$ (see e.g.~\cite{ContHomR}). In \cite[Section 2.3.3]{KnotContHom} this construction was refined to a version where the underlying algebra is the free product $\mathcal{A}(L) * \ZZ_2[H_1(L)]$, i.e.~where the group-ring elements do not commute with the Reeb chord generators. The Chekanov-Eliashberg algebra twisted by a hypersurface, as defined below, can in fact be obtained from this latter version of the Chekanov-Eliashberg algebra with Novikov coefficients.

For simplicity we will here only consider the case when there are finitely many Reeb chords, and we assume that $S$ is disjoint from the Reeb chords on $L$. Consider the algebra $\mathcal{A}(L;S)$ as defined in Section \ref{sec:twist-homol-algebr}. Recall that the generator $s$ is graded by $|s|=n-k-1=0$ in this case. We extend this algebra to
\[\widetilde{\mathcal{A}}(L;S):=\mathcal{A}(L;S)*\langle s^{-1} \rangle \supset \mathcal{A}(L;S),\]
by adding multiplicative inverses $s^{-l}$ to $s^l$; in particular, this algebra is no longer free.

Consider the decomposition
\[\mathcal{M}_{a;\mathbf{b};A}(L;J_{\OP{cyl}})=\bigsqcup_{\mathbf{w} \in \ZZ^{m+1}} \mathcal{M}_{a;\mathbf{b},\mathbf{w};A}(L;S;J_{\OP{cyl}}) \]
into different components, where $\mathcal{M}_{a;\mathbf{b},\mathbf{w};A}(L;S;J_{\OP{cyl}})$ consists of those $J_{\OP{cyl}}$-holomorphic discs whose boundary arc between the $i$:th and $(i+1)$:th puncture, starting the count at the positive puncture, has algebraic intersection-number $w_i \in \ZZ$ with $\RR \times S$. One can define a differential on $\widetilde{\mathcal{A}}(L;S)$ by the formulas
\begin{align*}
\partial_{S}(a)&:= \sum_{|a|-|\mathbf{b}|+\mu(A)=1 \atop {\mathbf{w} \in \ZZ^{m+1}}} |\mathcal{M}_{a;\mathbf{b},\mathbf{w};A}(L;S;J_{\OP{cyl}})/\RR|s^{w_1}b_1s^{w_2}\cdots s^{w_m}b_ms^{w_{m+1}},\\
\partial_{S}(s^{\pm1})&:=0,
\end{align*}
and extend the definition to all of $\widetilde{\mathcal{A}}(L;S)$ using the Leibniz rule. The below proposition can be seen to follow from the invariance result for the Chekanov-Eliashberg algebra with Novikov coefficients.
\begin{prop}
The homotopy type of the DGA $(\widetilde{\mathcal{A}}(L;S),\partial_S)$ is invariant under Legendrian isotopy and the choice of cylindrical almost complex structure $J_{\OP{cyl}}$. In fact, this DGA can be obtained as a quotient of the DGA as defined in \cite[Section 2.3.3]{KnotContHom}. Furthermore
\[ (\widetilde{\mathcal{A}}(L;S),\partial_S) / \langle s-1\rangle \simeq (\mathcal{A}(L),\partial).\]
\end{prop}

We now assume that $S \subset L$ is a co-oriented sphere for which there is an isotropic surgery disc $D_S \subset Y$ compatible with $S \subset L$. We moreover assume that there are no Reeb chords on $L \cup D_S$ starting or ending on $D_S$. According to Lemma \ref{lem:contracting}, this can be achieved after a Legendrian isotopy of~$L$.

Consider the decomposition
\[\mathcal{M}_{a;\mathbf{b};A}(V_S;J)=\bigsqcup_{\mathbf{w} \in \ZZ^{m+1}} \mathcal{M}_{a;\mathbf{b},\mathbf{w};A}(V_S;C_S;J) \]
into different components, where $\mathcal{M}_{a;\mathbf{b},\mathbf{w};A}(L;C_S;J)$ consists of those $J_{\OP{cyl}}$-holomorphic discs whose boundary arc between the $i$:th and $(i+1)$:th puncture, starting the count at the positive puncture, has algebraic intersection-number $w_i \in \ZZ$ with the core disc $C_S \subset V_S$.

Fix a generic almost complex structure $J_S$ as in Section \ref{sec:integrable}, which is used in the proof of Theorem \ref{thm:isom}. We use this almost complex structure together with the above decomposition of the moduli space to define a DGA morphism
\[ \Psi\co(\mathcal{A}(L_S),\partial_{L_S}) \to (\widetilde{\mathcal{A}}(L;S),\partial_S) \]
by prescribing it to take the value
\[ \Psi(a)=\sum_{|a|-|\mathbf{b}|+\mu(A)=0 \atop {\mathbf{w} \in \ZZ^{m+1}}} |\mathcal{M}_{a;\mathbf{b},\mathbf{w};A}(V_S;C_S;J)| s^{w_1}b_1s^{w_2}\cdots  s^{w_m}b_ms^{w_{m+1}}\]
on generators, and then extending it to a unital algebra map. The proof of Theorem \ref{thm:isom} given in Section \ref{sec:isomorphism} generalises to show that
\begin{prop}
\label{prop:critical}
The map
\[ \Psi\co (\mathcal{A}(L_S),\partial_{L_S}) \to (\widetilde{\mathcal{A}}(L;S),\partial_S)\]
is a well-defined and unital DGA morphism that, moreover, is injective and satisfies
\[ \Psi(c_S)=s.\]
\end{prop}

\begin{rmk} \begin{enumerate}
\item After adding formal inverses $c_S^{-l}$ to $c_S^l$ in the Chekanov-Eliashberg algera $\mathcal{A}(L_S)$ of $L_S$, the above map $\Psi$ becomes an isomorphism.
\item There are circumstances when the above map $\Psi$ has its image in the free sub-algebra
\[\mathcal{A}(L;S) \subset \widetilde{\mathcal{A}}(L;S).\]
In this case, it follows that
\[\Psi\co (\mathcal{A}(L_S),\partial) \to (\mathcal{A}(L;S),\partial_S)\]
is an isomorphism and, in particular, $\mathcal{A}(L;S) \subset \widetilde{\mathcal{A}}(L;S)$ is a sub-DGA.
\item There are certain geometric constraints on the isotropic surgery disc $D_S$ for which the assumptions in (2) above can be seen to hold. In particular, this is true in situations when the boundary of a pseudo-holomorphic disc in the above moduli spaces always has positive local intersection number with $C_S$ (appropriately oriented) at each intersection point.
\end{enumerate}
\end{rmk}

\section{The moduli spaces with boundary-point constraints}
\label{sec:moduli}
We here define the moduli spaces of pseudo-holomorphic discs in a symplectisation $(\RR \times Y,d(e^t\lambda))$ with boundary in an exact Lagrangian cobordism $V \subset \RR \times Y$ satisfying boundary-point constraints on parallel copies of a submanifold $M \subset V$ of codimension at least two. We also establish a transversality result for these moduli spaces.

\subsection{The definitions of the moduli space}
Similarly to the moduli spaces in \cite[Section 8.2.D]{EffectLegendrian}, we make the following definition. Again, we assume that $(Y,\lambda)$ is a $(2n+1)$-dimensional contact manifold. Let $V \subset (\RR \times Y,d(e^t\lambda))$ be an exact Lagrangian cobordism from $L_-$ to $L_+$ and let $M \subset V$ a $(k+1)$-dimensional submanifold of codimension at least two. We fix a non-vanishing normal vectorfield $\mathbf{v} \subset TV|_M$ to $M$, a Riemannian metric $g$ on $V$, and a compatible almost complex structure $J$ on $\RR \times Y$. We moreover require there to exist numbers $A<B$ for which:
\begin{itemize}
\item The compatible almost complex structure $J$ is cylindrical outside of the set $[A,B] \times Y$;
\item $V \cap \{ t \notin [A,B]\}$ coincides with the cylinders
\[((-\infty,A) \times L_-) \cup ((B,+\infty) \times L_+);\]
\item $M \cap \{ t \notin [A,B]\} \subset V$ coincides with
\[((-\infty,A) \times S_-) \cup ((B,+\infty) \times S_+)\]
where $S_\pm \subset L_\pm$ and we, moreover, require $S_\pm$ to be disjoint from the end-points of the Reeb chords on $L_\pm$;
\item $\mathbf{v}$ restricted to the sets $\{ t \le A \}$ and $\{ t \ge B \}$ is given by $\mathbf{v}_- \in TL_-$ and $\mathbf{v}_+ \in TL_+$, respectively, where $\mathbf{v}_\pm$ moreover is invariant under translations of the $t$-coordinate; and
\item $g$ is the product metric $dt^2 + g_\pm$ outside outside of $V \cap \{ t \in [A,B] \}$, where $g_\pm$ is a Riemannian metric on $L_\pm$.
\end{itemize}

Let $a$ be a Reeb chord on $L_+$, $\mathbf{b}=b_1  \cdots  b_m$ a word of Reeb chords on $L_-$, and $A \in H_1(V)$. We also fix a tuple
\[\mathbf{w}=(w_1,\hdots,w_{m+1}) \in (\ZZ_{\ge 0})^{m+1}\]
and write
\[w=w_1+\cdots+w_{m+1}.\]
\begin{defn}
For any $\delta>0$ the moduli space
\[\mathcal{M}^{g,\delta}_{a;\mathbf{b},\mathbf{w};A}(V;M,\mathbf{v};J) \subset \mathcal{M}_{a;\mathbf{b};A}(V;J)\]
consists of the solutions
\[u \co (\dot{D}^2,\partial\dot{D}^2) \to (\RR \times Y,V)\]
satisfying the following boundary-point constraints. Recall that $\dot{D} = D \setminus \{ p_0,\hdots,p_m\}$ denotes the unit disc with $m+1$ fixed boundary points removed.

We require there to be $w$ boundary points on $\partial \dot{D}^2$ that are mapped to the submanifolds
\[M_{i,\delta} :=\exp_M((i-1) \delta \mathbf{v}), \quad i=1,\hdots,w.\]
Moreover, the $i$:th point with respect to the order on $\partial D^2 \setminus \{ p_0 \}$ induced by the orientation, is here required to be mapped to $M_{i,\delta}$ and we require there to be $w_{i+1}$ such points on the boundary arc in $\partial \dot{D}^2$ starting at $p_i \in \partial D^2$ and ending at $p_{i+1} \in \partial D^2$ (here we set $p_{m+1}:=p_0$).
\end{defn}

The case when $V=\RR \times L$, $\mathbf{v}$, and
\[M_{j,\delta} = \RR \times  \exp_S((j-1) \delta \mathbf{v}) \subset \RR \times L,\quad j=1,\hdots,w,\]
all are invariant under translations of the $t$-coordinate will be referred to as the \emph{cylindrical setting}. When the almost complex structure $J=J_{\OP{cyl}}$ also is cylindrical, it follows that there is a natural action by $\RR$ on the above moduli spaces induced by translation of the $t$-coordinate. In this case we will write
\[\mathcal{M}^{g,\delta}_{a;\mathbf{b},\mathbf{w};A}(L;S,\mathbf{v};J_{\OP{cyl}}) := \mathcal{M}^{g,\delta}_{a;\mathbf{b},\mathbf{w};A}(\RR \times L;\RR \times S,\mathbf{v};J_{\OP{cyl}}).\]

\subsubsection{Definition in terms of the evaluation map}

In the case when $m+1 \ge 3$ there are no conformal orientation-preserving reparametrisations of $\dot{D}^2$ that fix the puncture $p_0$. In other words, there is a unique map $u \co \dot{D}^2 \to \RR \times Y$ representing a solution $u \in \mathcal{M}_{a;\mathbf{b};A}(V;J)$.

In the case when $m+1<3$, there is a $(2-m)$-dimensional family of conformal automorphisms of $\dot{D}^2$. In order to construct an evaluation-map on these moduli-spaces, we can either choose a representative $u \co \dot{D}^2 \to \RR \times Y$ for each $u \in \mathcal{M}_{a;\mathbf{b};A}(V;J)$ (this has to be done in a way that continuously depends on the point in the moduli space), or one can choose to work on the level of moduli spaces of \emph{parametrised maps}. We will choose the latter approach, and we will use $\widetilde{\mathcal{M}}_{a;\mathbf{b};A}(V;J)$ to denote the moduli spaces of parametrised solutions.

There is a well-defined smooth evaluation-map
\begin{gather*}
\OP{ev}_w \co \widetilde{\mathcal{M}}_{a;\mathbf{b};A}(V;J) \times \partial \dot{D}^w \to V^w,\\
(u,(e^{i\theta_1},\hdots,e^{i\theta_w})) \mapsto (u(e^{i\theta_1}),\hdots,u(e^{i\theta_w})).
\end{gather*}
In other words, the above moduli spaces $\mathcal{M}^{g,\delta}_{a;\mathbf{b},\mathbf{w};A}(V;M,\mathbf{v};J)$ can be identified with an appropriate connected component of a quotient of
\[ \OP{ev}_w^{-1}(M_{1,\delta} \times \cdots \times M_{w,\delta}),\]
where the quotient identifies two maps $u$ and $u'$ that differ by a holomorphic reparametrisation of the domain.

An important feature of $M_{1,\delta} \times \cdots \times M_{w,\delta} \subset V^w$ is that this submanifold is disjoint from the generalised diagonal
\[\{ (v_1,\hdots,v_w) \in V^w;\  v_i=v_j, \: i \neq j\} \subset V^w.\]
This property simplifies the transversality argument considerably.

\subsection{Transversality results}
The property of being transversely cut out for the moduli-space
\[\mathcal{M}^{g,\delta}_{a;\mathbf{b},\mathbf{w};A}(V;M,\mathbf{v};J) \subset \mathcal{M}_{a;\mathbf{b};A}(V;J)\]
can be reformulated into the requirements that
\begin{itemize}
\item $\mathcal{M}_{a;\mathbf{b};A}(V;J)$ is transversely cut out, and
\item the evaluation map
\[\OP{ev}_w \co \widetilde{\mathcal{M}}_{a;\mathbf{b};A}(V;J) \times \partial \dot{D}^w \to V^w\] is transverse to $M_{1,\delta} \times \cdots \times M_{w,\delta} \subset V^w$.
\end{itemize}
Recall that the first property holds for a generic choice of $J$ by the results in Section \ref{sec:bgtrans}.

\begin{prop}
\label{prop:dim}
In the case when
\[\OP{ev}_w \co \widetilde{\mathcal{M}}_{a;\mathbf{b};A}(V;J) \times \partial \dot{D}^w \to V^w\] is transverse to $M_{1,\delta} \times \cdots \times M_{w,\delta} \subset V^w$, it follows that
\[ \dim \mathcal{M}^{g,\delta}_{a;\mathbf{b},\mathbf{w};A}(V;M,\mathbf{v};J) = |a|-|\mathbf{b}|+\mu(A)-(n-k-1)w.\]
\end{prop}
\begin{proof}
Recall that, for regular $J$, we have
\[\dim \mathcal{M}_{a;\mathbf{b};A}(V;J)=|a|-|\mathbf{b}|+\mu(A).\]
The domain of $\OP{ev}_w$ thus has dimension
\begin{align*}
\dim (\widetilde{\mathcal{M}}_{a;\mathbf{b};A}(V;J) \times \partial \dot{D}^w)&=|a|-|\mathbf{b}|+\mu(A)+w+\eta_m(2-m),\\
\eta_m &:= \begin{cases} 0, & m \ge 2,\\
1, & m < 2,
\end{cases}
\end{align*}
while the codimension of $M_{1,\delta} \times \cdots \times M_{w,\delta} \subset V^w$ is given by
\[w(n+1)-w(k+1)=w(n-k).\]
The statement now follows.
\end{proof}

We will need the following standard result.
\begin{lem}
\label{lem:ham}
Any compactly supported smooth isotopy $\phi^s \co V \to V$, where $V \subset (X,\omega)$ is a Lagrangian submanifold, extends to a Hamiltonian isotopy $\phi^s_{H_t}$ of $X$. Similarly, any compactly supported smooth isotopy $\phi^s \co L \to L$, where $L \subset (Y,\xi)$ is a Legendrian submanifold, extends to a contact isotopy of $Y$.
\end{lem}
\begin{proof}
We begin with the first case. Take a Weinstein neighbourhood of $V$ symplectomorphic to a neighbourhood of the zero-section of the cotangent bundle $(T^*V,d\theta_V)$, under which $V$ moreover is identified with the zero-section.  Let $\Gamma_s \subset TV$ be the one-parameter family of vectorfields generating the isotopy $\phi^s$. Consider the time-dependent Hamiltonian
\begin{gather*}
H_s \co T^*V \to \RR, \\
\eta \mapsto \eta(\Gamma_s).
\end{gather*}
It is easily checked, e.g.~in a choice of Darboux coordinates, that the Hamiltonian flow $\phi^s_{H_s}$ of $H_s$ is given by $\phi^s$ along the zero-section. This Hamiltonian can then be suitably cut off to generate a Hamiltonian flow on $X$ with the required properties.

The statement about the contact isotopy follow similarly. Since $L$ has a neighbourhood that is contactomorphic to $(J^1L,dz+\theta_L)$, one can lift the above Hamiltonian isotopy on $T^*L$ to a (suitably cut off) contact isotopy of $J^1L$. It can moreover be checked that this contact isotopy preserves the zero-section of $J^1L$. We again obtain the sought contact isotopy by cutting off the induced contact Hamiltonian.
\end{proof}

Let $\mathcal{J}$ denote a Banach manifold of compatible almost complex structures on $(\RR \times Y,d(e^t\lambda))$ that are cylindrical outside of some compact set.
The transversality results in Section \ref{sec:bgtrans} are established via the intermediate result that the so-called universal moduli space
\[\widetilde{\mathcal{M}}_{a;\mathbf{b};A}(V;J) \times_J \mathcal{J} \to \mathcal{J}\]
is a Banach manifold, where the projection onto $\mathcal{J}$ is a Fredholm map. Here we have to use a suitable functional analytic set-up, which we omit from the discussion. It then follows from the Sard-Smale theorem that there is a Baire subset of regular almost complex structures in $\mathcal{J}$.

This result also holds in the case when $V=\RR \times L$ and $\mathcal{J}=\mathcal{J}_{\OP{cyl}}$ consists of cylindrical almost complex structures, as follows from the proof of Proposition \ref{prop:trans}. This result is a generalisation of \cite[Theorem 1.8]{FredholmTheory}, which considers the case of moduli-spaces of pseudo-holomorphic maps from a closed Riemann surface into $\RR \times Y$.

\begin{prop}
\label{prop:submersion}
The evaluation map
\[\OP{ev}_w \co \widetilde{\mathcal{M}}_{a;\mathbf{b};A}(V;J) \times_J \mathcal{J} \times \partial \dot{D}^w \to V^w\] from the universal moduli space is transverse to the submanifold  $M_{1,\delta} \times \cdots \times M_{w,\delta} \subset V^w$. In the cylindrical setting, the same is true when $\mathcal{J}=\mathcal{J}_{\OP{cyl}}$ is a suitable space of cylindrical almost complex structures.
\end{prop}
\begin{proof}
We write $u=(a,v)$. Assume that
\[\OP{ev}_w(u,J,(e^{i\theta_1},\hdots,e^{i\theta_w})) \in M_{1,\delta} \times \cdots \times M_{w,\delta}.\]
We will show that the differential $D_{(u,J,(e^{i\theta_1},\hdots,e^{i\theta_w}))}\OP{ev}_w$ is a surjection onto the space
\[ W_1 \times \cdots \times W_w \subset T_{\OP{ev}_w(u,J,(e^{i\theta_1},\hdots,e^{i\theta_w}))}V^w,\]
defined by
\[  T_{u(e^{i\theta_j})}V \supset W_j=\begin{cases}
T_{u(e^{i\theta_j})}L_-, & a(e^{i\theta_j}) < A, \\
T_{u(e^{i\theta_j})}V, & A \le a(e^{i\theta_j}) \le B, \\
T_{u(e^{i\theta_j})}L_+, & a(e^{i\theta_j}) > B.
\end{cases}\]
Since each $M_{j,\delta} \subset \RR \times Y$ is invariant under translations of the $t$-coordinate outside of the set $[A,B] \times Y$, this will imply the claim.

Recall that $\OP{ev}_w(u,J,(e^{i\theta_1},\hdots,e^{i\theta_w}))$ is disjoint from the generalised diagonal
\[\{ (v_1,\hdots,v_w) \in V^w;\  v_i=v_j, \: i \neq j\} \subset V^w\]
by the construction of $M_{j,\delta}$. The result can now be established by the following standard argument (see e.g.~\cite[Section 3.4]{JholCurves}).

Take any vector
\[(\zeta_1,\hdots,\zeta_w) \in W_1 \times \cdots \times W_w,\]
and write $(x_1,\hdots,x_w):=\OP{ev}_w(u,J,(e^{i\theta_1},\hdots,e^{i\theta_w}))$. Using Lemma \ref{lem:ham} it follows that there is a Hamiltonian isotopy $\phi^s$ of $\RR \times Y$ satisfying $\phi^s(V)=V$ as well as
\[\frac{d}{ds} (\phi^s(x_1),\hdots,\phi^s(x_w))=(\zeta_1,\hdots,\zeta_w).\]
Furthermore, outside of the set $[A,B] \times Y$, this Hamiltonian may be taken to be of the form $\phi^s=(\Id_\RR,\widetilde{\phi}^s)$, where $\widetilde{\phi}^s$ is a contact isotopy.

Observe that $\phi^s$ induces a one-parameter family
\[s \mapsto (\phi^s \circ u, (D\phi^s)\circ J \circ (D\phi^s)^{-1}) \in \widetilde{\mathcal{M}}_{a;\mathbf{b};A}(V;J) \times_J \mathcal{J} \]
of solutions in the universal moduli space. Finally, differentiating $\OP{ev}_w$ along this curve can be seen to give the vector $(\zeta_1, \hdots, \zeta_w)$ as required.
\end{proof}

Using the Sard-Smale theorem for Banach manifolds, it follows that there is a Baire subset of compatible almost complex structures $J \in \mathcal{J}$ for which the moduli space $\mathcal{M}^{g,\delta}_{a;\mathbf{b};\mathbf{w};A}(V;M,\mathbf{v};J)$ is transversely cut-out. Together with the Gromov-Hofer compactness in \cite{CompSFT} (also see Section \ref{sec:bgcomp}) for the moduli spaces $\mathcal{M}_{a;\mathbf{b};A}(V;J)$, we conclude that
\begin{thm}
\label{thm:transcomp}
There is a Baire subset of $J \in \mathcal{J}$ for which the moduli space $\mathcal{M}^{g,\delta}_{a;\mathbf{b};\mathbf{w};A}(V;M,\mathbf{v};J)$ is transversely cut-out, and for which the compactification of
\[ \mathcal{M}^{g,\delta}_{a;\mathbf{b};\mathbf{w};A}(V;M,\mathbf{v};J) \subset \overline{\mathcal{M}}^{g,\delta}_{a;\mathbf{b};A}(V;J) \]
is transverse to the boundary. In the cylindrical setting, the same is true when $\mathcal{J}=\mathcal{J}_{\OP{cyl}}$ consists of cylindrical almost complex structures.
\end{thm}

\section{Acknowledgements}
This paper is based upon the author's Ph.D.~thesis \cite{MyThesis}. The author would like to express his deepest gratitude to Tobias Ekholm, his Ph.D.~advisor, who suggested the problem and and who was very helpful during the writing of thesis and also to Frédéric Bourgeois, the main opponent at its examination, for pointing out mistakes and providing valuable comments. Also, the author would like to thank Kai Zehmisch for interesting discussions concerning boundary-point constraints on pseudo-holomorphic curves.

\address{Georgios Dimitroglou Rizell\\Centre for Mathematical Sciences, University of Cambridge\\
Wilberforce Road, Cambridge CB3 0WB, UK\\
\email{g.dimitroglou@maths.cam.ac.uk}\\
\end{document}